\documentclass[10pt]{amsart}
\usepackage{latexsym,amsmath,amssymb,graphicx}
\usepackage[bf, small]{caption}
\usepackage[all,web]{xy}
\usepackage{color}
\usepackage{hyperref}
\usepackage{xcolor}

\def\xyellowspace{%
  \sbox0{\colorbox{yellow}{\strut\ }}%
  \dimen0=\wd0\relax
  \hskip0pt\cleaders\box0\hskip\dimen0\hskip0pt}

\begingroup
\catcode`\ =\active%
\gdef\makeyellowspace{\let \xyellowspace\catcode`\ =\active}%
\endgroup

\def\?#1{\colorbox{yellow}{\strut#1}}

\def\urlfont{\DeclareFontFamily{OT1}{cmtt}{\hyphenchar\font='057}
              \normalfont\ttfamily \hyphenpenalty=10000}

\DeclareFontFamily{OT1}{rsfs10}{}
\DeclareFontShape{OT1}{rsfs10}{m}{n}{ <-> rsfs10 }{}
\DeclareMathAlphabet{\mathscript}{OT1}{rsfs10}{m}{n}

\DeclareMathOperator{\im}{Im}       % \im    = immagine (di una funzione)
\DeclareMathOperator{\Proj}{Proj}   % \Proj  = Proj di un anello graduato
\DeclareMathOperator{\Hom}{Hom}     % \Hom   = gruppo Hom (diritto)
\DeclareMathOperator{\Tors}{Tors}    % \Hom   = sottogruppo di torsione
\DeclareMathOperator{\Pic}{Pic}     % \Pic   = gruppo di Picard
\DeclareMathOperator{\Cl}{Cl}       % \Cl    = gruppo dei divisori di Weil mod. linear equivalence
\DeclareMathOperator{\rk}{rk}       % \rk    = rango
\DeclareMathOperator{\Mov}{Mov}     % \Mov   = moving cone
\DeclareMathOperator{\Nef}{Nef}     % \Nef   = cono dei divisori nef
\DeclareMathOperator{\Eff}{Eff}     % \Eff   = cono dei divisori effettivi
\DeclareMathOperator{\NE}{NE}       % \Eff   = cono di Mori
\DeclareMathOperator{\Relint}{Relint}  % \Relint   = interno di un cono nel suo span lineare
\DeclareMathOperator{\codim}{codim} % \codim = codimensione
\DeclareMathOperator{\diag}{diag}   % \diag  = matrice diagonale
\DeclareMathOperator{\lcm}{lcm}     % \lcm   = minimo comun denominatore
\DeclareMathOperator{\HNF}{HNF}     % \HNF   = Hermite normal form
\DeclareMathOperator{\SNF}{SNF}     % \SNF   = Hermite normal form
\DeclareMathOperator{\REF}{REF}     % \REF   = Hermite normal form
\DeclareMathOperator{\PC}{PC}       % \PC    = Primitive collections

\title[Classification of $\Q$--factorial projective toric varieties]{A Batyrev type classification of \\$\Q$--factorial projective toric varieties}

\author[M. Rossi and L.Terracini]{Michele Rossi and Lea Terracini}

\date{\today}

\address{Dipartimento di Matematica, Universit\`a di Torino,
via Carlo Alberto 10, 10123 Torino} \email{michele.rossi@unito.it,
lea.terracini@unito.it}

\thanks{The authors were partially supported by the MIUR-PRIN 2010-11 Research Funds ``Geometria delle Variet\`{a} Algebriche''. The first author is also supported by the I.N.D.A.M. as a member of the G.N.S.A.G.A.}


\def \a{\alpha }
\def \b{\beta }
\def \d{\delta }
\def \l{\lambda }
\def \s{\sigma }
\def \k{\kappa}

\def \Ga{\Gamma }
\def \Si{\Sigma }
\def \g{\gamma}
\def \e{\mathbf{e}}

\def \q{\mathbf{q}}
\def \pp{\mathbf{p}}

\def \v{\mathbf{v}}
\def \n{\mathbf{n}}
\def \w{\mathbf{w}}

\def \1{\mathbf{1}}
\def \0{\mathbf{0}}

\def\P{{\mathbb{P}}}

\def\p2{\mathbb{P}^2}
\def\p3{\mathbb{P}^3}
\def\p4{\mathbb{P}^4}

\def\rk{\operatorname{rk}}

\def\GL{\operatorname{GL}}

\def\NE{\operatorname{NE}}

\def\Z{\mathbb{Z}}

\def\C{\mathbb{C}}
\def\R{\mathbb{R}}

\def\Q{\mathbb{Q}}
\def\N{\mathbb{N}}

\def\CQ{\mathcal{Q}}
\def\SF{\mathcal{SF}}
\def\pc{\mathcal{P}}
\def\gkz{\mathcal{Q}}

\def\G{\mathcal{G}}
\def\Ga{\Gamma}

\def\Weil{\mathcal{W}_T}

\theoremstyle{plain}
\newtheorem{theorem}{Theorem}[section]
\newtheorem{proposition}[theorem]{Proposition}

\newtheorem{thm-def}[theorem]{Theorem--Definition}
\newtheorem{corollary}[theorem]{Corollary}

\newtheorem{lemma}[theorem]{Lemma}

\newtheorem*{claim}{Claim}

\newtheorem*{a-proposition}{Proposition}

\theoremstyle{remark}
\newtheorem{remark}[theorem]{Remark}

\newtheorem{example}[theorem]{Example}

\theoremstyle{definition}
\newtheorem{definition}[theorem]{Definition}

\newtheorem*{step I}{Step I}
\newtheorem*{step II}{Step II}
\newtheorem*{step III}{Step III}
\newtheorem*{step IV}{Step IV}

\newtheorem*{acknowledgements}{Acknowledgements}

\newcommand{\oneline}{\vskip12pt}
\newcommand{\halfline}{\vskip6pt}
\newcommand{\ka}{K\"{a}hler }

\newcommand{\longmapsfrom}{\mathrel{\reflectbox{\ensuremath{\longmapsto}}}}

\begin{document}

 %\pagestyle{empty}
 %\DefineParaStyle{Maple Heading 4}
 %\DefineParaStyle{Maple Heading 2}
 %\DefineParaStyle{Maple Text Output}
 %\DefineParaStyle{Maple Bullet Item}
 %\DefineParaStyle{Maple Warning}
 %\DefineParaStyle{Maple Error}
 %\DefineParaStyle{Maple Dash Item}
% \DefineParaStyle{Maple Heading 3}
 %\DefineParaStyle{Maple Heading 1}
 %\DefineParaStyle{Maple Title}
 %\DefineParaStyle{Maple Normal}
% \DefineCharStyle{Maple 2D Input}
% \DefineCharStyle{Maple Maple Input}
 %\DefineCharStyle{Maple 2D Output}
 %\DefineCharStyle{Maple 2D Math}
 %\DefineCharStyle{Maple Hyperlink}

\begin{abstract} The present paper is devoted to generalizing, inside the class of projective toric varieties, the classification \cite{Batyrev91}, performed by Batyrev in 1991 for smooth complete toric varieties, to the singular $\Q$--factorial case.
\end{abstract}

\maketitle

\tableofcontents

\section*{Introduction}
The present paper is the third part of a longstanding tripartite study aimed at realizing, for $\Q$--factorial projective toric varieties, a classification inspired by what V.~Batyrev did in 1991 for smooth complete toric varieties \cite{Batyrev91}. The first part of this study is given by \cite{RT-LA&GD}, in which we studied Gale duality from the $\Z$--linear point of view and defined poly weighted spaces (PWS, for short: see the following Definition \ref{def:PWS}) as $\Q$-factorial complete toric varieties whose classes group is free. The second part is given by \cite{RT-QUOT}, in which we exhibited a canonic covering PWS $Y$ for every $\Q$--factorial complete toric variety $X$, such that the covering map $Y\to X$ is a torus--equivariant Galois covering, induced by the multiplicative action of the finite group $\mu(X):=\Hom(\Tors(\Cl(X)),\C^*)$ on $Y$ and ramified in codimension greater than or equal to 2.
The reader will often be referred to these papers for notation, preliminaries and results.

Considerably simplifying the situation, the main results of the present paper could be summarized as follows:
\begin{theorem}\label{thm:intro}

Given a $\Q$--factorial projective toric variety $X$ satisfying \emph{some good conditions} on an associated \emph{weight matrix} (see \cite[Def.~3.9]{RT-LA&GD} and Definition \ref{def:Wmatrix} below) then $X$ is birational and isomorphic in codimension 1 to a finite abelian quotient of a PWS which is a \emph{toric cover} (see Definition~\ref{def:toricover}) of a \emph{weighted projective toric bundle (WPTB)} (see \S~\ref{sssez:WPTB}).

Moreover $X$ is isomorphic to a finite abelian quotient of a PWS which is a toric cover of a WPTB if and only if its fan is associated with a chamber of the secondary fan which is \emph{maximally bordering (maxbord)} (see Definition~\ref{def:bordering}) inside the Gale dual (or GKZ) cone $\gkz$.

Finally $X$ is isomorphic to a finite abelian quotient of a PWS produced from a toric cover of a weighted projective space (WPS) by a sequence of toric covers of weighted projective toric bundles if and only if its fan chamber is \emph{recursively maxbord} (see Definition~\ref{def:totally maxbord}) inside the Gale dual cone $\gkz$.

In any case the finite abelian quotient is trivial if and only if $\Cl(X)$ is a free abelian group meaning that $X$ is a PWS (recall \cite[Thm.~2.1]{RT-QUOT}).
\end{theorem}

This statement is the patching of theorems \ref{thm:quot-maxbord}, \ref{thm:quot-birWPTB} and \ref{thm:quot-recmaxbord} which are immediate consequences, by \cite[Thm.~2.2]{RT-QUOT}, of theorems \ref{thm:maxbord}, \ref{thm:birWPTB} and \ref{thm:recmaxbord}, respectively.

Before clarifying the meaning of emphasized terms in the statement above, let us recall that this kind of results are mostly well known in the context of smooth complete toric varieties. By this point of view, the first important result is pro\-ba\-bly given the Kleinschmidt classification of smooth projective toric varieties with Picard number (in the following called \emph{rank}) $r\leq 2$ as suitable projective toric bundles (PTB, for short) over a projective space of smaller dimension \cite{Kleinschmidt}. Later P.~Kleinschmidt and B.~Sturmfels proved that every smooth complete toric variety of rank $r\leq 3$ is necessarily projective \cite{Kleinschmidt-Sturmfels}, so extending the Kleinschmidt classification to the range of smooth complete toric variety of rank $r\leq2$. In 1991 V.~Batyrev generalized the Kleinschmidt classification by introducing the concepts of \emph{pri\-mi\-ti\-ve collection} and of associated \emph{primitive relation} \cite[Definitions~2.6,7,8]{Batyrev91} (see also the following \S~\ref{sssez:primitive2}): namely he proved that a smooth complete toric variety $X(\Si)$ is a PTB over a toric variety of smaller dimension if and only if the fan $\Si$ admits a primitive collection with focus 0 (in the following also called \emph{nef}: see \ref{sssez:primitive2}) which is disjoint from any other primitive collection of $\Si$ \cite[Prop.~4.1]{Batyrev91}. Consequently a smooth complete toric variety $X(\Si)$ is produced from a projective space by a sequence of PTB if and only if $\Si$ is a \emph{splitting fan} \cite[Def.~4.2, Thm.~4.3, Cor.~4.4]{Batyrev91}.

Let us first of all underline that Batyrev's techniques are deeply connected with the smoothness hypothesis. In fact the starting step of the induction proving \cite[Thm.~4.3]{Batyrev91} does not more hold in the singular set up, even for projective varieties: there exist projective $\Q$--factorial toric varieties, of rank $r\geq 2$, do not admitting any numerically effective primitive collection although all their primitive collections are disjoint pair by pair. Example~\ref{ex:noconverse} gives an account of this situation. Even for rank $r=1$ the singular case looks to be significantly more intricate than the smooth one, since the former necessarily involves some finite covering: in  fact on the one hand the unique smooth complete toric variety with $r=1$ is given by the projective space and on the other hand a $\Q$--factorial complete toric variety with $r=1$ is a quotient of a weighted projective space (WPS, for short), as proved by V.~Batyrev and D.~Cox \cite{BC} and by H.~Conrads \cite{Conrads}.

As already recalled, the latter result has been extended to every rank $r$ by \cite[Thm.~2.2]{RT-QUOT}, here recalled by Theorem~\ref{thm:covering&quotient}, allowing us to reducing the classification of $\Q$--factorial complete toric varieties to the problem of classifying their covering PWS i.e. to classifying $\Q$--factorial complete toric varieties with free classes group.

Bypassing counterexample \ref{ex:noconverse} means characterizing those PWS admitting a nef primitive collection. This is done by stressing remarks of C.~Casagrande \cite{Casagrande} and of D.~Cox and C.~von~Renesse \cite{Cox-vRenesse}, revising the original Batyrev definition of primitive relation: \S~\ref{sssez:primitive2} is largely devoted to this purpose. The idea is that of dually thinking of the numerical class of a primitive relation as a hyperplane in $\Cl(X)\otimes\R$, which we called the \emph{supporting hyperplane} of the primitive collection (see Definition~\ref{def:support}). By applying $\Z$--linear Gale duality de\-ve\-lo\-ped in \cite{RT-LA&GD}, in \S~\ref{ssez:GKZ} a linear algebraic interpretation of the secondary (or GKZ) fan is proposed. More precisely, given a $F$--matrix $V$ (see Definition~\ref{def:Fmatrice}) we can choose a Gale dual $W$--matrix $Q=\G(V)$ (see Definition~\ref{def:Wmatrix} and \cite[\S~3.1]{RT-LA&GD})  such that $Q$ is a positive and in row echelon form (REF) matrix (see \cite[Thm.~3.18]{RT-LA&GD} and the following Proposition~\ref{prop:positivit?}). The secondary fan can then be thought of as a suitable fan whose support is given by the strongly convex cone $\gkz=\langle Q\rangle$, called the \emph{Gale dual} cone and generated by the columns of the weight matrix $Q$. This gives a $\Z$--linear algebraic interpretation of the duality linking simplicial fans ge\-ne\-ra\-ted by the columns of the fan matrix $V$ and \emph{bunches of cones}, in the sense of \cite{Berchtold-Hausen}, inside the Gale dual cone $\gkz$, in terms of the $\Z$--linear Gale duality linking submatrices of $V$ and $Q$ exhibited by \cite[Thm.~3.2]{RT-LA&GD}: in particular this gives a bijection between simplicial fans $\Si$ giving $\Q$--factorial projective toric varieties $X$ whose fan matrix is $V$ and $r$--dimensional subcones $\g$ of $\gkz$ (called \emph{chambers}) obtained as intersection of the cones in the corresponding bunch of cones. In particular $\gkz$ turns out to be contained in the positive orthant $F^r_+$ of $\Cl(X)\otimes\R$ and the properties of a primitive collection $\pc$ for $\Si$ can be thought of in terms of mutual position of the corresponding support hyperplane $H_{\pc}$ with respect to the fan chamber $\g$ (see Propositions~\ref{prop:primitive} and \ref{prop:normal-inward}). E.g. $\pc$ is numerically effective if and only if $H_{\pc}$ cut out a facet of the Gale dual cone $\gkz$ i.e. $\pc$ is a \emph{bordering} primitive collection, in the sense of Definition~\ref{def:bordering}. Moreover a chamber $\g\subseteq\gkz$ is called \emph{(maximally) bordering} if it admits a (facet) face lying on the boun\-da\-ry $\partial\gkz$. Theorem~\ref{thm:intbord-pc} exhibits precise relation between bordering chambers and bordering primitive collections so giving the characterization of those PWS admitting a nef primitive collection we are looking for: it is the generalization of \cite[Prop.~3.2]{Batyrev91} to the singular $\Q$--factorial set up. Then extension of the Batyrev classification to the singular $\Q$--factorial case is given by Theorem~\ref{thm:maxbord}: in particular the latter together Proposition~\ref{prop:maxbord vs primitive} generalizes \cite[Prop.~4.1]{Batyrev91}, together Proposition~\ref{prop:bimaxbord} generalizes \cite[Thm.~4.3]{Batyrev91} and together Theorem~\ref{thm:recmaxbord} generalizes \cite[Cor.~4.4]{Batyrev91}.

\halfline
Let us then come back to explain the obscure hypothesis in the statement of Theorem \ref{thm:intro} above: \emph{good conditions on the associated weight matrix} means that $Q$ can be set in a positive REF such that, by deleting the bottom row and the last $s+1$ columns on the right, one still get an (almost) $W$--matrix $Q'$ which will give a weight matrix of the $s$--fibration basis. In other words this condition means that the Gale dual cone $\gkz$ may admit a maximally bordering fan chamber, which is the generalization of Batyrev's condition requiring the existence of a nef primitive collection disjoint from any further primitive collection of the same fan.

From the geometric point of view, a maximally bordering chamber corresponds to giving a fibering morphism whose fibers are a suitable abelian quotient of a weighted projective space (called a \emph{fake WPS}): this is a well known fact which is essentially rooted in Reid's work \cite{Reid83}. See also \cite[Prop.\,1.11]{Hu-Keel} and \cite[\S\,2]{Casagrande13} for more recent results suggesting possible interesting applications, of techniques here presented, in the more general setup of Mori Dream Spaces. As explained in \S~\ref{ssez:geometrico}, Remark~\ref{rem:geometrico} and Remark~\ref{rem:geometrico_quot}, this fibering morphism gives the Stein factorization of the toric cover of a WPTB exhibited by Theorem~\ref{thm:maxbord}, and more generally by the previous Theorem~\ref{thm:intro}, so obtaining a commutative diagram
\begin{equation*}
  \xymatrix{X(\Si)\ar[d]_-{\begin{array}{c}
                             _{\text{fake WPS}} \\
                             ^{\text{fibering}}
                           \end{array}
}^-{\phi}\ar[r]_f^-{\text{finite}}&\P^W(\mathcal{E})\ar[d]_-{\varphi}^-{\text{WPTB}}\\
            X_0(\Si_0)\ar[r]^-{f_0}_-{\text{finite}}&X'(\Si')}
\end{equation*}
whose vertical morphisms have connected fibers and whose horizontal ones are finite morphisms of toric varieties. Notice that if $X$ is smooth then both the finite toric morphisms $f$ and $f_0$ are trivial giving that $\phi=\varphi$ is precisely the Batyrev's projective toric bundle. The right hand side factorization $\varphi\circ f$ presents the great advantage of being constructively described, giving a procedural approach to an ef\-fec\-ti\-ve determination of all the morphisms and varieties involved, as examples in \S~\ref{ssez:esempi} show.
The last procedure can be easily implemented in any computer algebra package (we used Maple to perform all the necessary computations).

\halfline
Let us now describe the structure of this paper and quickly summarize the further obtained results.

\S~\ref{sez:preliminari} is firstly devoted to introduce notation and preliminaries: the long list of notation in \S~\ref{ssez:lista} recalls lot of symbols defined in \cite{RT-LA&GD} and \cite{RT-QUOT}. Then \S~\ref{ssez:GKZ} is dedicated to introduce the above mentioned $\Z$--linear algebraic interpretation of the secondary fan. Theorem~\ref{thm:GKZ} bridges between the linear algebraic secondary fan defined in Definition~\ref{def:GKZ&Mov} and the usual secondary fan of a $\Q$--factorial complete toric variety. The bijection and $\Z$--linear Gale duality between projective fans and GKZ chambers are established by Theorem~\ref{thm:Gale sui coni}.

The long \S~\ref{sez:Batyrev-rivisto} is the main part of the present paper, in which the above described Batyrev--type classification of PWS is performed. \S~\ref{sssez:primitive2} is devoted to revising the concept of a primitive collection and to introducing all the bordering notion for collections and chambers with respect to the Gale dual cone $\gkz$. In \S~\ref{ssez:t-bundles&covers} we introduce the main ingredients useful for the classification. \S~\ref{sssez:WPTB} is dedicated to define a \emph{weighted projective toric bundle} (WPTB) $\P^W(\mathcal{E})$ as the $\Proj$ of the $W$--weighted symmetric algebra $S^W(\mathcal{E})$ over a locally free sheaf $\mathcal{E}$. In Proposition~\ref{prop:fan fibrato} we describe the fan of a WPTB, as a $\Q$--factorial toric variety, along the lines of what is done in \cite[Prop.~7.3.3]{CLS} for a projective toric bundle (PTB). \S~\ref{sssez:Tcovers} is devoted to recall the concept of a \emph{toric cover}, as defined in \cite{AP}. \S~\ref{ssez:maxbord&WPTB} is the core of the present paper with Theorem~\ref{thm:maxbord} and, from the birational point of view (i.e. up to toric flips as defined in \S~\ref{ssez:toricflip}), Theorem~\ref{thm:birWPTB}. The geometric meaning of a maxbord chamber is explained in \S~\ref{ssez:geometrico}, as already described above.  In \S~\ref{ssez:maxbord&splitting} we give the generalization, in the singular $\Q$--factorial setting, of Batyrev's concept of a splitting fan, giving rise to Theorem~\ref{thm:recmaxbord}. In particular, when $r\leq 3$, Theorem~\ref{thm:3pc} and Remark \ref{rem:r-2} give a partial extension to the singular $\Q$--factorial case of Batyrev's results on the number of primitive relations (see \cite[\S~5 and 6]{Batyrev91}). \S~\ref{ssez:contraibile} is devoted to giving some partial generalization, to the $\Q$--factorial set up, of results about contractible classes on smooth projective toric varieties due to C.~Casagrande \cite{Casagrande} and H.~Sato \cite{Sato}: our study is limited to the case of numerically effective classes (see Proposition~\ref{prop:contraibile} and Theorem~\ref{thm:contraibile}). The last subsection \S~\ref{ssez:esempi} is devoted to give an extensive treatment of applications of all the techniques described, by means of five examples:  here it is rather important for the reader to be equipped with some computer algebra package which has the ability to produce Hermite and Smith normal forms of matrices and their switching matrices. For example, using Maple, similar procedures are given by \texttt{HermiteForm} and \texttt{SmithForm}  with their output options.

\noindent Let us notice that the last Example \ref{ex:WPTB(c)} exhibits the case of a (4--dimensional) $\Q$--factorial complete toric variety of Picard number $r=3$ whose $\Nef$ cone is $0$, i.e. which does not admit any non--trivial numerical effective divisor: we think this is a significant and new example since in the smooth case O.~Fujino and S.~Payne proved that this is not possible for $r\leq 4$, at least for dimension $\leq 3$ \cite{FP}. For further considerations about this subject see Remark \ref{rem:FP}.

Finally the conclusive \S~\ref{sez:Qfproj} is devoted to apply results obtained in the previous \S~\ref{sez:Batyrev-rivisto} for PWS to the case of a general $\Q$--factorial projective toric variety. The above Theorem~\ref{thm:intro} is the patching of results here stated. This section ends up with a further example aimed to classifying a $\Q$--factorial projective variety which is not a PWS.

\begin{acknowledgements}
  We would like to thank Cinzia Casagrande for helpful conversations and suggestions. We are also indebt with Brian Lehmann who pointed out to us the reference \cite{FS}. Last but not least we thank Daniela Caldini for her invaluable contribution in making the figures of the present paper.
\end{acknowledgements}

\section{Preliminaries and notation}\label{sez:preliminari}

The present paper is the third episode of a long study on $\Q$--factorial complete toric varieties. On the one hand, this is a further application of the $\Z$--linear Gale Duality de\-ve\-lo\-ped in the first paper \cite{RT-LA&GD}, to which the reader is referred for notation and all the necessary preliminary results. In particular for what concerning notation on toric varieties, cones and fans, the reader is referred to \cite[\S~1.1]{RT-LA&GD}, for linear algebraic preliminaries about normal forms of matrices (Hermite and Smith normal forms - HNF and SNF for short) to \cite[\S~1.2]{RT-LA&GD}. $\Z$--linear Gale Duality and what is concerning fan matrices ($F$--matrices) and weight matrices ($W$-matrices) is developed in \cite[\S~3]{RT-LA&GD}. On the other hand, the results here presented are a consequence of the fact that a $\Q$--factorial complete toric variety $X$ is always a finite geometric quotient of a poly weighted space (PWS) $Y$, which turns out to be the \emph{universal $1$--connected in codimension 1 covering} (\emph{$1$--covering}) of $X$ \cite[Def.~1.5, Thm.~2.2]{RT-QUOT}.

Every time the needed nomenclature will be recalled either directly by giving the necessary definition or by reporting the precise reference. Here is a list of main notation and relative references:

\subsection{List of notation}\label{ssez:lista}\hfill\\
Let $X(\Si)$ be a $n$--dimensional toric variety and $T\cong(\C^*)^n$ the acting torus, then
\begin{eqnarray*}
% \nonumber to remove numbering (before each equation)
  &M,N,M_{\R},N_{\R}& \text{denote the \emph{group of characters} of $T$, its dual group and}\\
  && \text{their tensor products with $\R$, respectively;} \\
  &\Si\subseteq N_{\R}& \text{is the fan defining $X$;} \\
  &\Si(i)& \text{is the \emph{$i$--skeleton} of $\Si$, which is the collection of all the}\\
  && \text{$i$--dimensional cones in $\Si$;} \\
  &|\Si|& \text{is the \emph{support} of the fan $\Si$, i.e.}\ |\Si|=\bigcup_{\s\in\Si}\s \subseteq N_{\R};\\
  &\det(\s)& :=|\det(V_{\s})|\ \text{for a simplicial cone $\s\in\Si(n)$ whose primitive}\\
  && \text{generators give the columns of $V_{\s}$};\\
  &\text{\emph{unimodular} $\s$}& \text{if $\det(\s)=1$};\\
  &r=\rk(X)&\text{is the Picard number of $X$, also called the \emph{rank} of $X$};\\
  &\mathfrak{P}=\mathfrak{P}(1,\ldots,n+r)&\text{is the power set of the set of indexes $\{1,\ldots,n+r\}$};\\
  &F^r_{\R}& \cong\R^r,\ \text{is the $\R$--linear span of the free part of $\Cl(X(\Si))$};\\
  &F^r_+& \text{is the positive orthant of}\ F^r_{\R}\cong\R^r;\\
  &\langle\v_1,\ldots,\v_s\rangle\subseteq\N_{\R}& \text{denotes the cone generated by the vectors $\v_1,\ldots,\v_s\in N_{\R}$;}\\
  && \text{if $s=1$ then this cone is also called the \emph{ray} generated by $\v_1$;} \\
  &\mathcal{L}(\v_1,\ldots,\v_s)\subseteq N& \text{denotes the sublattice spanned by $\v_1,\ldots,\v_s\in N$\,.}
\end{eqnarray*}
Let $A\in\mathbf{M}(d,m;\Z)$ be a $d\times m$ integer matrix, then
\begin{eqnarray*}
% \nonumber to remove numbering (before each equation)
  &\mathcal{L}_r(A)\subseteq\Z^m& \text{denotes the sublattice spanned by the rows of $A$;} \\
  &\mathcal{L}_c(A)\subseteq\Z^d& \text{denotes the sublattice spanned by the columns of $A$;} \\
  &A_I\,,\,A^I& \text{$\forall\,I\subseteq\{1,\ldots,m\}$ the former is the submatrix of $A$ given by}\\
  && \text{the columns indexed by $I$ and the latter is the submatrix of}\\
  && \text{$A$ whose columns are indexed by the complementary }\\
  && \text{subset $\{1,\ldots,m\}\backslash I$;} \\
  &_sA\,,\,^sA& \text{$\forall\,1\leq s\leq d$\ the former is the submatrix of $A$ given by the}\\
  && \text{lower $s$ rows and the latter is the submatrix of $A$ given by}\\
  && \text{the upper $s$ rows of $A$;} \\
  &\HNF(A)\,,\,\SNF(A)& \text{denote the Hermite and the Smith normal forms of $A$,}\\
  && \text{respectively;}\\
  &\REF& \text{Row Echelon Form of a matrix;}\\
  &\text{\emph{positive}}\ (\geq 0)& \text{a matrix (vector) whose entries are non-negative.}\\
  &\text{\emph{strictly positive}}\ (> 0)& \text{a matrix (vector) whose entries are strictly positive.}
\end{eqnarray*}
Given a $F$--matrix $V=(\v_1,\ldots,\v_{n+r})\in\mathbf{M}(n,n+r;\Z)$ (see Definition \ref{def:Fmatrice} below), then
\begin{eqnarray*}
% \nonumber to remove numbering (before each equation)
  &\langle V\rangle& =\langle\v_1,\ldots,\v_{n+r}\rangle\subseteq N_{\R}\ \text{denotes the cone generated by the columns of $V$;} \\
  &\SF(V)& =\SF(\v_1,\ldots,\v_{n+r})\ \text{is the set of all rational simplicial fans $\Si$ such that}\\
  && \text{$\Sigma(1)=\{\langle\v_1\rangle,\ldots,\langle\v_{n+r}\rangle\}\subset N_{\R}$ \cite[Def.~1.3]{RT-LA&GD};}\\
  &\P\SF(V)&:=\{\Si\in\SF(V)\ |\ X(\Si)\ \text{is projective}\};\\
  &\G(V)&=Q\ \text{is a \emph{Gale dual} matrix of $V$ \cite[\S~3.1]{RT-LA&GD};} \\
  &\CQ&=\langle\G(V)\rangle \subseteq F^r_+\ \text{is a \emph{Gale dual cone} of}\ \langle V\rangle:\ \text{it is always assumed to be}\\
  && \text{generated in $F^r_{\R}$ by the columns of a positive $\REF$ matrix $Q=\G(V)$}\\
  && \text{(see Proposition \ref{prop:positivit?} below)}.\\
  &V^{\text{red}}&\text{is the \emph{reduced} matrix of $V$ \cite[Def.~3.13]{RT-LA&GD} whose columns are given by}\\
  && \text{the primitive generators of $\langle\v_i\rangle$, with $1\leq i\leq n+r$.}\\
  &Q^{\text{red}}&=\G\left(V^{\text{red}}\right)\ \text{is the \emph{reduced} matrix of $Q=\G(V)$ \cite[Def.~3.14]{RT-LA&GD}.}
\end{eqnarray*}
\oneline
\noindent Let us start by recalling four fundamental definitions:

\begin{definition} A $n$--dimensional $\Q$--factorial complete toric variety $X=X(\Si)$ of rank $r$ is the toric variety defined by a $n$--dimensional \emph{simplicial} and \emph{complete} fan $\Si$ such that $|\Si(1)|=n+r$ \cite[\S~1.1.2]{RT-LA&GD}. In particular the rank $r$ coincides with the Picard number i.e. $r=\rk(\Pic(X))$.
\end{definition}

\begin{definition}[\cite{RT-LA&GD}, Def.~3.10]\label{def:Fmatrice} An \emph{$F$--matrix} is a $n\times (n+r)$ matrix  $V$ with integer entries, satisfying the conditions:
\begin{itemize}
\item[a)] $\rk(V)=n$;
\item[b)] $V$ is \emph{$F$--complete} i.e. $\langle V\rangle=N_{\R}\cong\R^n$ \cite[Def.~3.4]{RT-LA&GD};
\item[c)] all the columns of $V$ are non zero;
\item[d)] if ${\bf  v}$ is a column of $V$, then $V$ does not contain another column of the form $\lambda  {\bf  v}$ where $\lambda>0$ is real number.
\end{itemize}
A \emph{$CF$--matrix} is a $F$-matrix satisfying the further requirement
\begin{itemize}
\item[e)] the sublattice ${\mathcal L}_c(V)\subset\Z^n$ is cotorsion free, which is ${\mathcal L}_c(V)=\Z^n$ or, equivalently, ${\mathcal L}_r(V)\subset\Z^{n+r}$ is cotorsion free.
\end{itemize}
A $F$--matrix $V$ is called \emph{reduced} if every column of $V$ is composed by coprime entries \cite[Def.~3.13]{RT-LA&GD}.
\end{definition}
 The most significant example of $F$-matrix is given by a matrix $V$ whose columns are    integral vectors generating the rays of the $1$-skeleton $\Sigma(1)$ of a rational fan $\Sigma$. In the following a similar matrix $V$ will be called a \emph{fan matrix} of $\Sigma$; when every column of $V$ is composed by coprime entries, it will be called a \emph{reduced fan matrix}.

\begin{definition}\cite[Def.~3.9]{RT-LA&GD}\label{def:Wmatrix} A \emph{$W$--matrix} is an $r\times (n+r)$ matrix $Q$  with integer entries, satisfying the following conditions:
\begin{itemize}
\item[a)] $\rk(Q)=r$;
\item[b)] ${\mathcal L}_r(Q)$ has not cotorsion in $\Z^{n+r}$;
\item[c)] $Q$ is \emph{$W$--positive}, which is $\mathcal{L}_r(Q)$ admits a basis consisting of positive vectors (see list \ref{ssez:lista} and \cite[Def.~3.4]{RT-LA&GD}).
\item[d)] Every column of $Q$ is non-zero.
\item[e)] ${\mathcal L}_r(Q)$   does not contain vectors of the form $(0,\ldots,0,1,0,\ldots,0)$.
\item[f)]  ${\mathcal L}_r(Q)$ does not contain vectors of the form $(0,a,0,\ldots,0,b,0,\ldots,0)$, with $ab<0$.
\end{itemize}
A $W$--matrix is called \emph{reduced} if $V=\G(Q)$ is a reduced $F$--matrix \cite[Def.~3.13, Thm.~3.15]{RT-LA&GD}.
\end{definition}
In the following, if $V$ is a fan matrix of a rational fan $\Si$ then $Q=\G(V)$ is called a \emph{weight matrix} of $\Si$. If $V$ is reduced then $Q$ is a called a \emph{reduced weight matrix}.

\begin{definition}[\cite{RT-LA&GD} \S 2.2]\label{def:PWS} A \emph{poly weighted space} (PWS) is a $n$--dimensional $\Q$--factorial complete toric variety $Y(\widehat{\Si})$ of rank $r$, whose \emph{reduced fan matrix} $\widehat{V}$ (see \cite[Def.~3.13]{RT-LA&GD}) is a $CF$--matrix i.e. if
\begin{itemize}
  \item $\widehat{V}$ is a $n\times(n+r)$ $CF$--matrix,
  \item $\widehat{\Si}\in\SF(\widehat{V})$.
\end{itemize}
Let us recall that a $\Q$--factorial complete toric variety $Y$ is a PWS if and only if it is \emph{1-connected in codimension 1} (or simply \emph{1-connected}): since $Y$ is normal it is equivalent to ask for $\pi_1(Y_\text{reg})\cong\Tors(\Cl(Y))=0$ \cite[Cor.~1.8, Thm.~2.1]{RT-QUOT}, where $Y_\text{reg}\subseteq Y$ is the Zariski open subset of regular points.
\end{definition}

\begin{example}\label{ex:1} In order to explain the introduced notation, consider a smooth and complete toric variety $X(\Si)$, of dimension and rank equal to 3, with reduced fan matrix $V$ given by
\begin{equation*}
    V=\left(
        \begin{array}{cccccc}
          1 & 0 & 0 & 0 & -1 & 1 \\
          0 & 1 & 0 & 0 & -1 & 1 \\
          0 & 0 & 1 & -1 & -1 & 1 \\
        \end{array}
      \right)\
\end{equation*}
i.e. such that $\Si\in\SF(V)$.
One can check that $V$ supports only two complete and simplicial rational fans admitting every column of $V$ as a ray generator; that is $\SF(V)=\{\Si_1,\Si_2\}$, where $\Si_1$ and $\Si_2$ are the fans of cones obtained as all possible faces of the following lists of maximal cones:
\begin{eqnarray*}
% \nonumber to remove numbering (before each equation)
  &\Si_1(3) = \{\{1, 4, 5\}, \{1, 3, 5\}, \{2, 4, 5\}, \{2, 3, 5\}, \{2, 4, 6\}, \{2, 3, 6\}, \{1, 4, 6\}, \{1, 3, 6\}\}& \\
  &\Si_2(3) = \{\{1, 4, 5\}, \{1, 3, 5\}, \{2, 4, 5\}, \{2, 3, 5\}, \{1, 2, 4\}, \{2, 3, 6\}, \{1, 3, 6\}, \{1, 2, 6\}\}&
\end{eqnarray*}
(here a maximal simplicial cone $\langle V_I\rangle$ is identified to the subset of column indexes $I\subseteq\{1,\ldots, n+r\}$).

Both $\Si_1$ and $\Si_2$ are smooth, so giving two possible choices for $X(\Si)$. Moreover \cite{Kleinschmidt-Sturmfels} guarantees that those fans are both projective, that is $\P\SF(V)=\SF(V)$.
\noindent A weight matrix of $X$ is given by the choice of a Gale dual matrix of $V$
\begin{equation*}
       Q=\left(
                             \begin{array}{cccccc}
                               1 & 1 & 1 & 0 & 1 & 0 \\
                               0 & 0 & 1 & 1 & 0 & 0 \\
                               0 & 0 & 0 & 0 & 1 & 1 \\
                             \end{array}
                           \right)=\G(V)\,.
\end{equation*}
Both $V$ and $Q$ are reduced and $V$ is even a $CF$-matrix, so giving that $X$ is a PWS.
\end{example}

\subsection{The secondary fan}\label{ssez:GKZ}
Let us here introduce a linear algebraic interpretation of the \emph{se\-con\-da\-ry} (or \emph{GKZ}) fan of a toric variety $X$.
For any further detail about the secondary fan of a toric variety $X(\Si)$, we refer the interested reader to the comprehensive monograph \cite{CLS} and references therein: among them let us recall the original sources (\cite{GKZ}, \cite{GKZ2} and \cite{Oda-Park}).

Let $V=(\v_1,\ldots,\v_{n+r})$ be a reduced $F$--matrix and $Q:=\G(V)=(\q_1,\ldots,\q_{n+r})$ an associated $W$--matrix. Consider the cone generated by the columns of $Q$
\begin{equation*}
    \CQ=\langle Q\rangle:=\langle\q_1,\ldots,\q_{n+r}\rangle.
\end{equation*}
For every $\Si\in\SF(V)$, one gets $|\Si|=\langle V\rangle=N_{\R}$. Then $\CQ$ turns out to be a strongly convex cone in $F^r_{\R}:=F^r\otimes\R$, where $F^r=\text{Free}(\Cl(X(\Si)))\cong\Z^r$ \cite[Lemma~14.3.2]{CLS}. Recalling \cite[Theorems~3.8,~3.18]{RT-LA&GD} we are able to improve this fact:

\begin{proposition}\label{prop:positivit?} Let $F^r_+$ denote the positive orthant of $F^r_{\R}$. Then $\langle V\rangle=N_{\R}$ if and only if there exists a positive $\REF$--matrix $Q$ such that $Q=\G(V)$ and $\CQ=\langle Q\rangle\subset F^r_+$. In particular, for every $\Si\in\SF(V)$, $X=X(\Si)$ is complete if and only if there exists a positive $\REF$--matrix $Q$ such that $Q=\G(V)$ and $\CQ=\langle Q\rangle\subset F^r_+$.
\end{proposition}

In the following, given a reduced $F$--matrix $V$, we will always assume the cone $\CQ\subseteq F^r_+$ and generated by the columns of a positive $\REF$ matrix $Q=\G(V)$.

\begin{definition}\label{def:GKZ&Mov} Let $\mathcal{S}_r$ be the family of all the $r$--dimensional subcones of $\CQ$ obtained as intersection of simplicial subcones of $\CQ$. Then define the \emph{secondary fan} (or \emph{GKZ decomposition}) of $V$ to be the set $\Ga=\Ga(V)$ of cones in $F^r_+$ such that
\begin{itemize}
  \item its subset of $r$--dimensional cones (the \emph{$r$--skeleton}) $\Ga(r)$ is composed by the minimal elements, with respect to the inclusion, of the family $\mathcal{S}_r$,
  \item its subset of $i$--dimensional cones (the \emph{$i$--skeleton}) $\Ga(i)$ is composed by all the $i$--dimensional faces of cones in $\Ga(r)$, for every $1\leq i\leq r-1$.
\end{itemize}
 A maximal cone $\g\in\Ga(r)$ is called a \emph{chamber} of the secondary fan $\Ga$. Finally define
\begin{equation}\label{Mov}
    \Mov(V):= \bigcap_{i=1}^{n+r}\left\langle Q^{\{i\}}\right\rangle\ ,
\end{equation}
where $\left\langle Q^{\{i\}}\right\rangle$ is the cone generated in $F^r_+$ by the columns of the submatrix $Q^{\{i\}}$ of $Q$ (see the list of notation \ref{ssez:lista}).
\end{definition}

\begin{theorem}\label{thm:GKZ}
If $V$ is a $F$--matrix then, for every $\Si\in\SF(V)$,
\begin{enumerate}
  \item $\CQ=\overline{\Eff}(X(\Si))$, the \emph{pseudo--effective cone of $X$}, which is the closure of the cone generated by effective Cartier divisor classes of $X$ \cite[Lemma~15.1.8]{CLS},
  \item $\Mov(V)=\overline{\Mov}(X(\Si))$, the closure of the cone generated by movable Cartier divisor classes of $X$ \cite[(15.1.5), (15.1.7), Thm.~15.1.10, Prop.~15.2.4]{CLS}.
  \item $\Ga(V)$ is the secondary fan (or GKZ decomposition) of $X(\Si)$ \cite[\S~15.2]{CLS}. In particular $\Ga$ is a fan and $|\Ga|=\CQ\subset F^r_+$.
\end{enumerate}
\end{theorem}

\begin{theorem}[\cite{CLS} Prop. 15.2.1]\label{thm:Gale sui coni} There is a one to one correspondence between the following sets
\begin{eqnarray*}
    \mathcal{A}_{\Ga}(V)&:=&\{\g\in\Ga(r)\ |\ \g\subset\Mov(V)\}\\
    \P\SF(V)&:=&\{\Si\in\SF(V)\ |\ X(\Si)\ \text{is projective}\}\ .
\end{eqnarray*}
\end{theorem}

For the following it is useful to understand the construction of such a correspondence. Namely (compare \cite{CLS} Prop. 15.2.1):
\begin{itemize}
  \item after \cite{Berchtold-Hausen}, given a chamber $\g\in\mathcal{A}_{\Ga}$ let us call \emph{the bunch of cones of $\g$} the collection of cones in $F^r_+$ given by
      \begin{equation*}
        \mathcal{B}(\g):=\left\{\left\langle Q_J\right\rangle\ |\ J\subset\{1,\ldots,n+r\}, |J|=r, \det\left(Q_J\right)\neq 0, \g\subset\left\langle Q_J\right\rangle\right\}
      \end{equation*}
      (see also \cite[p.~738]{CLS}),
  \item it turns out that $\bigcap_{\b\in\mathcal{B}(\g)}\b=\g$\,,
  \item  for any $\g\in\mathcal{A}_{\Ga}(V)$ there exists a unique fan $\Si_{\g}\in\P\SF(V)$ such that
  \begin{equation*}
    \Si_{\g}(n):=\left\{\left\langle V^J\right\rangle\ |\ \left\langle Q_J\right\rangle\in \mathcal{B}(\g)\right\}\,,
  \end{equation*}
  \item for any $\Si\in\P\SF(V)$ the collection of cones
  \begin{equation}\label{bunch}
    \mathcal{B}_{\Si}:=\left\{\left\langle Q^I\right\rangle\ |\ \left\langle V_I\right\rangle\in\Si(n)\right\}
  \end{equation}
  is the bunch of cones of the chamber $\g_{\Si}\in\mathcal{A}_{\Ga}$ given by $\g_{\Si}:=\bigcap_{\b\in\mathcal{B}_{\Si}}\b$.
\end{itemize}
Then the correspondence in Theorem \ref{thm:Gale sui coni} is realized by setting
\begin{equation*}
\begin{array}{ccc}
  \mathcal{A}_{\Ga}(V) & \longleftrightarrow & \P\SF(V) \\
  \g & \longmapsto & \Si_{\g} \\
  \g_{\Si} & \longmapsfrom & \Si
\end{array}
\end{equation*}

\begin{remark}\label{rem:cameraregolare} Notice that in the previous picture we get well established bijections
\begin{equation*}
  \forall\,\g\in\mathcal{A}_{\Ga}(V)\quad \begin{array}{ccc}
                                            \mathcal{B}(\g) & \longleftrightarrow & \Si_{\g}(n) \\
                                            \left\langle Q_J\right\rangle & \longmapsto & \left\langle V^J\right\rangle
                                          \end{array}
\end{equation*}
\begin{equation*}
  \forall\,\Si\in\P\SF(V)\quad \begin{array}{ccc}
                                            \Si(n) & \longleftrightarrow & \mathcal{B}_{\Si} \\
                                            \left\langle V_I\right\rangle & \longmapsto & \left\langle Q^I\right\rangle
                                          \end{array} \ .
\end{equation*}
A significant consequence of \cite[Cor.~3.3]{RT-LA&GD} is that these bijections preserve, possibly up to a constant integer, the determinants of generating submatrices, which is
 \begin{equation}\label{Gale_det}
    \d\cdot|\det\left(Q_J\right)|=|\det\left(V^J\right)|\quad\text{and}\quad|\det\left(V_I\right)|=\d\cdot|\det\left(Q^I\right)|\ ,
 \end{equation}
where $\d=1$ if and only if $V$ is a $CF$--matrix.
Therefore, in the following, when $V$ is a $CF$--matrix, a chamber $\g\in\mathcal{A}_{\Ga}(V)$ will be called \emph{non-singular} if the bunch of cones $\mathcal{B}(\g)$ is entirely composed by unimodular cones (as defined in list \ref{ssez:lista}) or, equivalently, if the associated fan $\Si_{\g}\in\P\SF(V)$ is non-singular.
\end{remark}

As a final result let us recall the following
\begin{proposition}[\cite{CLS} Thm. 15.1.10(c)]\label{prop:nef} If $V=(\v_1\ldots,\v_{n+r})$ is a $F$--matrix then, for every fan $\Si\in\P\SF(V)$, there is a na\-tu\-ral isomorphism $\Pic(X(\Si))\otimes\R\cong F^r_{\R}$ taking the cones
\begin{equation*}
    \Nef(X(\Si))\subseteq\overline{\Mov}(X(\Si))\subseteq\overline{\Eff}(X(\Si))
\end{equation*}
to the cones
\begin{equation*}
    \g_{\Si}\subseteq\Mov(V)\subseteq\CQ\,.
\end{equation*}
In particular, calling $d:\mathcal{W}_T(X(\Si))\to\Cl(X(\Si))$  the morphism giving to a torus invariant divisor $D$ its linear equivalence class $d(D)$, we get that:
\begin{enumerate}
  \item a $\Q$--Cartier divisor $D$ on $X(\Si)$ is a nef (ample) divisor if and only if $d(D)\in\g_{\Si}$ (\,$d(D)\in\Relint\left(\g_{\Si}\right)$, where $\Relint$ denotes the interior of the cone $\g_{\Si}$ in its linear span),
  \item $X(\Si)$ is $\Q$--Fano if and only if $$\sum_{j=1}^{n+r}d(D_j)\in\Relint\left(\g_{\Si}\right)$$
  where $D_j$ is the closure of the torus orbit of the ray $\langle\v_j\rangle$.
\end{enumerate}
 \end{proposition}

 \begin{example}[Example \ref{ex:1} continued]\label{ex:2} Let $X(\Si)$ be one of the two possible smooth and projective toric varieties defined in Example~\ref{ex:1}. \begin{figure}
\begin{center}
\includegraphics[width=7truecm]{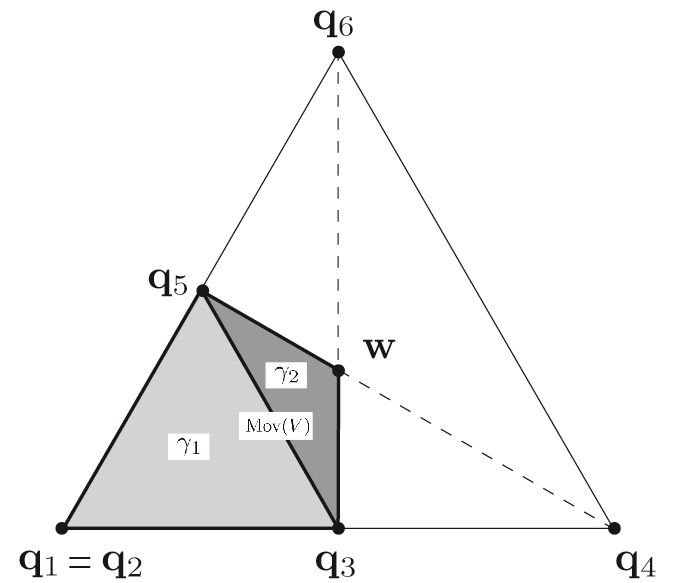}
\caption{\label{fig2}Ex.~\ref{ex:PTB}: the section of the cone $\Mov(V)$ and its chambers, inside the Gale dual cone $\mathcal{Q}=F^3_+$, as cut out by the plane $\sum_{i=1}^3x_i^2=1$.}
\end{center}
\end{figure}
One can visualize the pseudo-effective cone $\overline{\Eff}(X(\Si))$, that is the Gale dual cone $\gkz=\langle Q\rangle =F^3_+$, and the movable cone $\Mov(V)\subseteq\gkz$ by giving a picture of their section with the hyperplane $\{x_1+x_2+x_3=1\}\subseteq F^3_{\R}\cong\Cl(X)\otimes\R$, as in Fig.~\ref{fig2}.\\
$\P\SF(V)=\SF(V)$ can be dually described by the only two
chambers of $\Mov(V)$ represented in Fig.~\ref{fig2} and explicitly given by
\begin{equation*}
    \g_1=\left\langle\q_1=\q_2,\q_3,\q_5\right\rangle=\left\langle
            \begin{array}{ccc}
              1 & 1 & 1 \\
              0 & 1 & 0 \\
              0 & 0 & 1 \\
            \end{array}
          \right\rangle\quad,\quad\g_2=\left\langle\q_3,\w,\q_5\right\rangle=\left\langle
                                    \begin{array}{ccc}
                                      1 & 1 & 1 \\
                                      1 & 1 & 0 \\
                                      0 & 1 & 1 \\
                                    \end{array}
                                  \right\rangle
\end{equation*}
$$\Mov(V)=\langle\q_2,\q_3,\w,\q_5\rangle = \g_1+\g_2$$
where, as usual, $\q_1,\ldots,\q_6$ are the columns of $Q$ and $\w:=\q_3+\q_6=\q_4+\q_5$\,. Then
$\P\SF(V)=\SF(V)=\{\Si_1=\Si_{\g_1},\Si_2=\Si_{\g_2}\}$ and $\g_i=\Nef(X(\Si_i))$, $i=1,2$.
 \end{example}

 \subsection{Toric flips}\label{ssez:toricflip} In the present context, a \emph{toric flip} will be a torus--equivariant birational equivalence of projective $\Q$--factorial toric varieties which is an isomorphism in codimension 1. A toric flip is a composition of \emph{elementary flips} and a toric isomorphism  \cite[Thm.~15.3.14]{CLS}: given a reduced $F$--matrix $V$, an elementary flip is defined as the birational equivalence realized by passing, inside $\Mov(V)$, from a chamber to another one, just \emph{crossing a wall} \cite[(15.3.14)]{CLS}. \\
 E.g. in the previous examples \ref{ex:1} and \ref{ex:2}, the smooth projective toric varieties $X(\Si_1)$ and $X(\Si_2)$ are related by an elementary flip, obtained by crossing the wall determined by cutting $\gkz$ with the plane containing $\q_3$ and $\q_5$ (see Fig.~\ref{fig2}). Hence they are isomorphic in codimension 1.

\section{The Batyrev classification revised}\label{sez:Batyrev-rivisto}

 In the present section we are going to propose an alternative approach to Batyrev's results presented in \cite[\S~3, 4]{Batyrev91}, not depending on the smoothness hypothesis and holding for the case of a  $\Q$--factorial projective toric variety.

\subsection{Primitive relations and bordering chambers}\label{sssez:primitive2} Given a reduced $F$--matrix $V=(\v_1,\ldots,\v_{n+r})$ and a fan $\Si\in\SF(V)$, the datum of a collection of rays $\pc=\{\rho_1,\ldots,\rho_k\}\subseteq\Si(1)$ determines a subset $P=\{j_1,\ldots,j_k\}\subseteq\{1,\ldots,n+r\}$ such that
\begin{equation*}
    \pc=\{\rho_1,\ldots,\rho_k\}=\left\{\langle\v_{j_1}\rangle,\ldots,\langle\v_{j_k}\rangle\right\}
\end{equation*}
and a submatrix $V_P$ of $V$.

\subsubsection{Notation}\label{sssez:collezioni} By abuse of notation we will often write
\begin{equation*}
    \pc=\{\v_{j_1},\ldots,\v_{j_k}\}=\{V_P\}\,.
\end{equation*}
From the point of view of the Gale dual cone $\CQ=\langle Q\rangle$, where $Q=(\q_1,\ldots,\q_{n+r})=\G(V)$ is a reduced, positive, REF, $W$--matrix, the subset $P\subseteq\{1,\ldots,n+r\}$ determines the collection $\pc^*=\{\langle\q_{j_1}\rangle,\ldots,\langle\q_{j_k}\rangle\}\subseteq\Ga(1)$. By the same abuse of notation we will often write
\begin{equation*}
    \pc^*=\{\q_{j_1},\ldots,\q_{j_k}\}=\{Q_P\}\,.
\end{equation*}
The vector $\v_{\mathcal{P}}:=\sum_{i=1}^k \v_{j_i}$ lies in the relative interior of a cone $\s\in\Si$ and there is a unique relation
\begin{equation}\label{relazione primitiva 2}
    \v_{\mathcal{P}}-\sum_{\rho\in\s(1)}c_{\rho}\v_{\rho}=0\ \text{with $\left\langle\v_{\rho}\right\rangle=\rho\cap N$ and $c_{\rho}\in\Q\,,\,c_{\rho}> 0$}
\end{equation}
This fact allows us to define a rational vector $r(P)=r(\mathcal{P})=(b_1,\ldots,b_{n+r})\in\Q^{n+r}$, where $b_j$ is the coefficient of the column $\v_j$ of $V$ in (\ref{relazione primitiva 2}). Let $l$ be the least common denominator of $b_1,\ldots,b_{n+r}$. Then,
\begin{equation}\label{r_Z(P)}
    r_{\Z}(P)=r_{\Z}(\mathcal{P}):=l\,r(\mathcal{P})=(lb_1,\ldots,lb_{n+r})\in\mathcal{L}_r(Q)\subset\Z^{n+r}\,.
\end{equation}
Let us recall that a collection $\pc\subset\Si(1)$ is called \emph{primitive for $\Si$} if it is not contained in a single cone of $\Si$ but every proper subset of $\pc$ is (compare \cite[Def.~2.6]{Batyrev91}, \cite[Def.~1.1]{Cox-vRenesse}, \cite[Def.~5.1.5]{CLS}).
 If $\mathcal{P}$ is a primitive collection then it is determined by the positive entries in $r_{\Z}(\mathcal{P})$ (for the details see \cite[Lemma~1.8]{Cox-vRenesse}); this is no more the case if $\mathcal{P}$ is not a primitive collection.

 \noindent Consider the $\Q$-factorial complete toric variety $X=X(\Si)$ and recall the standard exact sequence on divisors
 \begin{equation}\label{div-sequence}
    \xymatrix{0 \ar[r] & M\ar[r]^-{div}_-{V^T}& \mathcal{W}_T(X)=\Z^{n+r}
\ar[r]^-{d} & \Cl(X) \ar[r]& 0}\ .
\end{equation}
where $\Weil(X)$ denotes the group of torus-invariant Weil divisors. Dualizing this sequence, one gets the following exact sequence of free abelian groups
\begin{equation}\label{HomZ-div-sequence}
                    \xymatrix{0 \ar[r] & A_1(X):=\Hom(\Cl(X),\Z) \ar[r]^-{d^{\vee}}_-{Q^T} & \Hom(\mathcal{W}_T(X),\Z)= \Z^{n+r}
 \ar[r]^-{div^{\vee}}_-{V} & N }
\end{equation}
Then (\ref{r_Z(P)}) gives that $r_{\Z}(P)\in\im(d^{\vee})$. Since $d^{\vee}$ is injective there exists a unique $\n_P\in A_1(X)$ such that
\begin{equation}\label{normalP}
    d^{\vee}(\n_P)=Q^T\cdot\n_P=r_{\Z}(P)
\end{equation}
which turns out to be the numerical equivalence class of the 1-cycle $r_{\Z}(P)$, whose intersection index with the torus--invariant Weil divisor $lD_j$ is given by the integer $lb_j$, for every $1\leq j\leq n+r$. In particular, given a primitive collection $\pc$ the associated primitive relation $r_{\Z}(P)$ is a numerically effective 1-cycle (nef) if and only if all the coefficients $lb_j$ in (\ref{r_Z(P)}) are non-negative: in this case $\pc$ will be called \emph{a numerically effective (nef) primitive collection}.

\begin{definition}\label{def:support} Given a collection $\pc=\{V_P\}$, for $P\subseteq\{1,\ldots,n+r\}$, its associated numerical class $\n_{P}\in N_1(X):=A_1(X)\otimes\R$, defined in (\ref{normalP}), determines a unique dual hyperplane
$$H_P\subseteq F^r_{\R}=\Cl(X)\otimes\R$$
which is called \emph{the support} of $\pc$, a \emph{positive half-space} $\mathcal{H}_P^+:=\{\mathbf{x}\in F^r_{\R}\,|\,\n_P\cdot\mathbf{x}\geq 0\}$ and a \emph{negative half-space} $\mathcal{H}_P^-:=\{\mathbf{x}\in F^r_{\R}\,|\,\n_P\cdot\mathbf{x}\leq 0\}$.
\end{definition}

Calling $\mathfrak{P}=\mathfrak{P}(\{1,\ldots,n+r\})$ the power set of $\{1,\ldots,n+r\}$, notation introduced in \ref{sssez:collezioni} allows us to think of the set of primitive collections of a fan $\Si\in\SF(V)$ as a suitable subset of $\mathfrak{P}$, namely
\begin{equation*}
    \PC(\Si)=\{P\in\mathfrak{P}\,|\,\{V_P\}\ \text{is a primitive collection}\}\,.
\end{equation*}

The following proposition gives some further characterization of a primitive collection.

\begin{proposition}\label{prop:primitive} Let $V$ be a reduced $F$--matrix, $Q=\G(V)$ be a Gale dual $\REF$, positive $W$--matrix, $\Si\in\P\SF(V)$ and $P\in\PC(\Si)$ such that $\pc=\{V_P\}$ is a primitive collection for $\Si$. Then $|\pc|=|P|\leq n+1$ and the following facts are equivalent:
\begin{enumerate}
  \item $\pc$ is a primitive collection for $\Si$, which is
  \begin{itemize}
    \item[($i.1$)] $\forall\,\s\in\Si(n)\quad\pc\nsubseteq\s(1)$,
    \item[($ii.1$)] $\forall\,\rho_i\in\pc\quad\exists\,\s\in\Si(n):\pc\backslash\{\rho_i\}\subseteq\s(1)$;
  \end{itemize}
  \item $V_P$ is a submatrix of $V$ such that
  \begin{itemize}
    \item[($i.2$)] $\forall\,J\subseteq\{1,\ldots,n+r\}:\langle V_J\rangle\in\Si(n)\quad\langle V_P\rangle\nsubseteq\langle V_J\rangle$,
    \item[($ii.2$)] $\forall\,i\in P\quad\exists\,J\subseteq\{1,\ldots,n+r\}:\langle V_J\rangle\in\Si(n)\ ,\ \langle V_{P\backslash\{i\}}\rangle\subseteq\langle V_J\rangle$;
  \end{itemize}
  \item $Q_P$ is a submatrix of $Q=\G(V)$ such that
  \begin{itemize}
    \item[($i.3$)] $\forall\,J\subseteq\{1,\ldots,n+r\}:\langle Q^J\rangle\in\mathcal{B}(\g_{\Si})\quad\langle Q^J\rangle\nsubseteq\langle Q^P\rangle$,
    \item[($ii.3$)] $\forall\,i\in P\quad\exists\,J\subseteq\{1,\ldots,n+r\}:\langle Q^J\rangle\in\mathcal{B}(\g_{\Si})\ ,\ \langle Q^J\rangle\subseteq\langle Q^{P\backslash\{i\}}\rangle$;
  \end{itemize}
  \item $Q_P$ is a submatrix of $Q=\G(V)$ such that
  \begin{itemize}
    \item[($i.4$)] $\g_{\Si}\nsubseteq\langle Q^P\rangle$,
    \item[($ii.4$)] $\forall\,i\in P\quad\g_{\Si}\subseteq\langle Q^{P\backslash\{i\}}\rangle$.
   \end{itemize}
\end{enumerate}
Moreover the previous conditions ($ii.1$), ($ii.2$), ($ii.3$), ($ii.4$) are equivalent to the following one:
\begin{itemize}
  \item[($ii$)] $\forall\,i\in P\quad\exists\, \mathcal{C}_{i,P}\in\mathcal{B}(\g_{\Si}):\ \mathcal{C}_{i,P}(1)\cap\pc^*=\{\langle\q_{i}\rangle\}$\,.
\end{itemize}
\end{proposition}

\begin{proof} First of all let us notice that if $|\pc|\geq n+2$ then condition ($ii.1$) can't be satisfied since every cone $\s\in\Si(n)$ is simplicial, implying that $|\s(1)|=n<n+1\leq |\pc|-1$. Then $|\pc|\leq n+1$ for a primitive collection.

The equivalence $(1)\Leftrightarrow(2)$ is clear. The equivalence $(2)\Leftrightarrow(3)$ follows by Gale duality and Theorem \ref{thm:Gale sui coni}. For the equivalence $(3)\Leftrightarrow(4)$:

\noindent $(i.3)\Rightarrow(i.4)$: $(i.4)$ is always true when $|P|=n+1$ because $\dim(\langle Q^P\rangle)\leq r-1$. Let us then assume $|P|\leq n$ and assume that $\g_{\Si}\subseteq\langle Q^P\rangle$. Then there certainly exists a simplicial subcone of $\langle Q^P\rangle$ containing $\g_{\Si}$, which is
\begin{equation}\label{no(i.3)}
    \exists\, J\subseteq\{1,\ldots,n+r\}:\langle Q^J\rangle\in\mathcal{B}(\g_{\Si})\ ,\ \langle Q^J\rangle\subseteq\langle Q^P\rangle
\end{equation}
contradicting $(i.3)$.

\noindent $(i.4)\Rightarrow(i.3)$: Assume (\ref{no(i.3)}). Then $\g_{\Si}\subseteq\langle Q^J\rangle\subseteq\langle Q^P\rangle$, contradicting $(i.4)$.

\noindent $(ii.3)\Rightarrow(ii.4)$: By $(ii.3)$, $\g_{\Si}\subseteq\langle Q^J\rangle\subseteq\langle Q^{P\backslash\{i\}}\rangle$, clearly giving $(ii.4)$.

\noindent $(ii.4)\Rightarrow(ii.3)$: Since $|P|-1\leq n$, assuming $\g_{\Si}\subseteq\langle Q^{P\backslash\{i\}}\rangle$ always give a simplicial subcone $\langle Q^J\rangle$ of $\langle Q^{P\backslash\{i\}}\rangle$ containing $\g_{\Si}$. Then $(ii.3)$ follows.

For the last part:

 \noindent$(ii.4)\Rightarrow(ii)$\,:\ if $|P|-1=n$ then define $\mathcal{C}_{i,P}:=\langle Q^{P\backslash\{i\}}\rangle$ which is a simplicial cone; if $|P|\leq n$ then $\langle Q^{P\backslash\{i\}}\rangle\supseteq\g_{\Si}$ and $|\langle Q^{P\backslash\{i\}}\rangle(1)|\geq r+1$; consider the simplicial star-subdivision of  $\langle Q^{P\backslash\{i\}}\rangle$ having center in the ray $\q_i\in\pc^*$; in this subdivision let $\mathcal{C}_{i,P}$ be the unique simplicial subcone containing $\g_{\Si}$, which exists by the definition of the secondary fan $\Ga$; being $\mathcal{C}_{i,P}\subseteq\langle Q^{P\backslash\{i\}}\rangle$ then $\mathcal{C}_{i,P}(1)\cap(\pc^*\backslash\{\langle\q_i\rangle\})=\emptyset$; but $\langle\q_i\rangle\in\mathcal{C}_{i,P}(1)$ by construction; then $\mathcal{C}_{i,P}(1)\cap\pc^*=\{\langle\q_i\rangle\})$;

 \noindent$(ii)\Rightarrow(ii.3)$\,:\ set $\langle Q^J\rangle:=\mathcal{C}_{i,P}$\,; then
  \begin{equation*}
    \langle Q^J\rangle(1)\cap(\pc^*\backslash\{\langle\q_i\rangle\})=\mathcal{C}_{i,P}(1)\cap(\pc^*\backslash\{\langle\q_i\rangle\})=\emptyset\,
    \Rightarrow\,\langle Q^J\rangle\subseteq\langle Q^{P\backslash\{i\}}\rangle
  \end{equation*}
\end{proof}

Let $\overline{\NE}(X)\subseteq N_1(X)$ be the Mori cone, generated by the numerical classes of effective curves. M.~Reid proved that $\overline{\NE}(X)$ is closed and polyhedral when $X$ is a $\Q$--factorial complete toric variety (\cite[Cor.~(1.7)]{Reid83}) and generated by classes of torus--invariant curves. When $X$ is smooth, C.~Casagrande ensures that the numerical class $\n_P$, of a primitive relation  $r_{\Z}(P)$, belong to $A_1(X)\cap\overline{\NE}(X)$ \cite[Lemma~1.4]{Casagrande}. D.~Cox and C.~von~Renesse generalize and improve this fact to toric varieties whose fan has convex support, showing that the Mori cone is generated by numerical classes of primitive relations \cite[Propositions.~1.9 and 1.10]{Cox-vRenesse}, namely
\begin{equation}\label{Mori}
    \overline{\NE}(X)=\sum_{P\in\PC(\Si)}\R_+ \n_P\,.
\end{equation}
In particular we get the following

\begin{proposition}[Lemma 1.4 in \cite{Casagrande}, Prop.~1.9 in \cite{Cox-vRenesse}]\label{prop:normal-inward} If $\pc=\{V_P\}$ is a primitive collection, for some $P\in\PC(\Si)$, then its numerical class $\n_P$ is positive against every nef divisor of $X(\Si)$, which is $\g_{\Si}\subseteq\mathcal{H}^+_P$.
\end{proposition}

Dualizing (\ref{Mori}) we get the following description of the closure of the \ka cone
\begin{equation}\label{Nef}
    \Nef(X)=\bigcap_{P\in\PC(\Si)}\mathcal{H}^+_P\,.
\end{equation}
Then Proposition \ref{prop:primitive} allows us to give the following alternative description of this cone:

\begin{corollary}\label{cor:NEF} Let $V$ be a reduced $F$--matrix, $Q=\G(V)$ be a $\REF$, positive $W$--matrix and $\Si\in\P\SF(V)$. Then
\begin{equation*}
    \Nef(X(\Si))=\bigcap_{i\in P\in\PC(\Si)}\left\langle Q^{P\backslash\{i\}}\right\rangle\,.
\end{equation*}
\end{corollary}

\begin{proof} By Proposition \ref{prop:nef}, $\Nef(X(\Si))=\g_{\Si}=\bigcap_{\b\in\mathcal{B}(\g_{\Si})}\b$; then clearly
\begin{equation*}
    \Nef(X(\Si))\subseteq\bigcap_{i\in P\in\PC(\Si)}\mathcal{C}_{i,P}\subseteq \bigcap_{i\in P\in\PC(\Si)}\left\langle Q^{P\backslash\{i\}}\right\rangle
\end{equation*}
where $\mathcal{C}_{i,P}\in\mathcal{B}(\g_{\Si})$ are the cones defined in condition ($ii$) of Proposition \ref{prop:primitive}. For the converse notice that
$$\forall\,P\in\PC(\Si)\quad\left\langle Q^{P\backslash\{i\}}\right\rangle=\left\langle Q^P\right\rangle\cup\left(\mathcal{H}_P^+\cap\left\langle Q^{P\backslash\{i\}}\right\rangle\right)\,.$$
But $\langle Q^P\rangle\subseteq\mathcal{H}_P^-$ and $\Si\in\P\SF(V)$ implies that $\dim(\Nef(X))=\dim(\g_{\Si})=r$. Then (\ref{Nef}) gives
\begin{equation*}
    \bigcap_{i\in P\in\PC(\Si)}\left\langle Q^{P\backslash\{i\}}\right\rangle=\bigcap_{i\in P\in\PC(\Si)}\left(\mathcal{H}_P^+\cap\left\langle Q^{P\backslash\{i\}}\right\rangle\right)\subseteq\bigcap_{P\in\PC(\Si)}\mathcal{H}_P^+= \Nef(X)\,.
\end{equation*}
\end{proof}

\begin{definition}[\emph{Bordering} collections and chambers]\label{def:bordering}\hfill
\begin{enumerate}
  \item Let $V$ be a reduced $F$--matrix and $Q=\G(V)$ a $\REF$, positive $W$--matrix. A collection $\pc=\{V_P\}$, for some $P\in\mathfrak{P}$, is called \emph{bordering} if its support $H_P$ cuts out a facet of the Gale dual cone $\gkz=\langle Q\rangle$.
  \item A chamber $\g\in\Ga(V)$ is called \emph{bordering} if $\dim(\g\cap\partial\gkz)\geq 1$. Notice that $\g\cap\partial\gkz$ is always composed by faces of $\g$: if it contains a facet of $\g$ then $\g$ is called \emph{maximally bordering} (\emph{maxbord} for short).  A hyperplane $H$ cutting a facet of $\gkz$ and such that $\dim(\g\cap H)\geq 1$ is called a \emph{bordering hyperplane of $\g$} and the bordering chamber $\g$ is also called \emph{bordering with respect to $H$}. A normal vector $\n$ to a bordering hyperplane $H$ is called \emph{inward} if $\n\cdot x\geq 0$ for every $x\in\g$.
\end{enumerate}

\end{definition}

\begin{remark}\label{rem:bordering->nef} Let us give some geometric interpretation of concepts introduced in the previous Definition~\ref{def:bordering}.
\begin{enumerate}
  \item $P\in\PC(\Si)$ gives a bordering primitive collection $\pc=\{V_P\}$ if and only $\pc$ is nef, which is that $r_{\Z}(\pc)$ is a numerically effective 1-cycle.
  \item Thinking of $\g\in\Ga(V)$ as the cone $\Nef(X(\Si_{\g}))\subseteq\Eff(X(\Si_{\g}))$, then the chamber $\g$ turns out to be bordering if and only if $X(\Si_{\g})$ admits non-trivial effective divisors which are nef but non-big \cite{FS}. Following \cite{Hu-Keel} and \cite{Casagrande13}, this is equivalent to require the existence of a \emph{rational contraction of fiber type} $f:X\dashrightarrow Y$ to a normal projective toric variety $Y$.
\end{enumerate}

\end{remark}

\begin{remark}\label{rem:conobordante}
Let $\g\in\Ga(V)$ be a bordering chamber and $H$ be a bordering hyperplane of $\g$. Then there exist al least $r-1$ columns of $Q=\G(V)$ belonging to $H$. Let $\mathcal{C}_H$ be the $(r-1)$--dimensional cone generated by all the columns of $Q$ belonging to $H$. Then
\begin{equation}\label{nelcono}
    \g\cap H\subset \mathcal{C}_H.
\end{equation}
In fact $\gkz=|\Ga(V)|$ and $\g\subset\gkz$, giving that $\g\cap H\subset \CQ \cap H=\mathcal{C}_H$.
\end{remark}

\begin{proposition}\label{prop:maxbord} Let $V$ be a reduced $F$--matrix, $Q=\G(V)$ a positive, $\REF$ $W$--matrix and $\g\in\Ga(r)\subseteq\gkz=\langle Q\rangle$. Then $\g$ is a maxbord chamber w.r.t. a hyperplane $H$ if and only if
\begin{equation*}
    \forall\,\b\in\mathcal{B}(\g)\quad\exists\q\in\gkz(1)\backslash\mathcal{C}_H(1):\b=\langle\q\rangle+\b\cap H
\end{equation*}
where $\mathcal{C}_H$ is the $(r-1)$--dimensional cone generated by all the columns of $Q$ belonging to $H$.
\end{proposition}

\begin{proof} If $\g$ is maxbord w.r.t. $H$ then $\dim(\g\cap H)=r-1$ implying that
\begin{equation*}
    \forall\,\b\in\mathcal{B}(\g)\quad\dim(\b\cap H)=r-1\,.
\end{equation*}
Since $\b$ is simplicial, this suffices to show that there exists a unique $\q\in\b(1)$ not belonging to $H$ and such that $\b=\langle\q\rangle+\b\cap H$.

The converse follows immediately by recalling that $\g=\bigcap_{\b\in\mathcal{B}(\g)}\b$ and we are assuming $\dim(\g)=r$.
\end{proof}

\begin{definition}\label{def:intbord} Let $V$ be a reduced $F$--matrix. A bordering chamber $\g\in\Ga(V)$, w.r.t. the hyperplane $H$,
is called \emph{internal bordering} (\emph{intbord} for short) w.r.t. $H$, if either $\g$ is maxbord w.r.t $H$ or there exists an hyperplane $H'$, cutting a facet of $\g$ and such that
\begin{itemize}
  \item[($i$)] $\g\cap H\subseteq \g\cap H'$
  \item[($ii$)] $\exists\,\q_1,\q_2\in H\cap\gkz(1) : (\n'\cdot\q_1)(\n'\cdot\q_2)< 0$
 \end{itemize}
where $\n'$ is the inward primitive normal vector of $H'$.
\end{definition}

\begin{remark}\label{rem:max-int-bordering} For Picard number $r\leq 2$, a chamber $\g\in \Ga(V)$ is bordering w.r.t. a hyperplane $H$ if and only if it is intbord w.r.t. $H$ if and only if it is maxbord w.r.t. $H$. For $r\geq 3$, this is no more the case but
\begin{equation*}
    \text{maxbord w.r.t.}\,H\ \Longrightarrow\ \text{intbord w.r.t.}\,H\ \Longrightarrow\ \text{bordering w.r.t.}\,H
\end{equation*}
and there exist chambers which are either bordering and not intbord or intbord and not maxbord  w.r.t $H$.
\end{remark}

The following result gives the existence of a bordering primitive collection, hence of a numerically effective primitive relation, for a $\Q$--factorial projective toric variety whose fan corresponds to an intbord chamber of the secondary fan $\Ga(V)$. This is one of the key results of the present paper, allowing us to improve and extend the Batyrev classification, explained in \cite{Batyrev91}, to the case of singular $\Q$--factorial projective toric varieties. In some sense, the following result is the analogue, in a singular setup, of Batyrev's result \cite[Prop.~3.2]{Batyrev91} (see the following Remark \ref{rem:intbord->smooth}).

\begin{theorem}\label{thm:intbord-pc} Let $V$ be a reduced $F$--matrix and assume $\g\in\mathcal{A}_{\Ga}(V)$ be a bordering chamber w.r.t. a hyperplane $H$. Then $\g$ is an intbord chamber w.r.t. $H$ if and only if the hyperplane $H$ is the support of a bordering primitive collection $\pc$, for the fan $\Si_{\g}\in\P\SF(V)$.
\end{theorem}

\begin{proof}
Let $H$ be a bordering hyperplane for $\g$. Let us assume, up to a permutation of columns of $Q=\G(V)$, that the first $s\geq r-1$ columns $\q_1,\ldots,\q_s$ are all the columns of $Q$ belonging to $H$. Setting $P=\{s+1,\ldots,n+r\}\in\mathfrak{P}$, consider the collection $\pc=\{V_P\}$. We want to show that $\pc$ is a primitive collection for $\Si_{\g}$. On the one hand, condition ($i.4$) in Proposition \ref{prop:primitive} is immediately satisfied since $\det(Q^P)=\det(Q_{\{1,\ldots,s\}})=0$, being $\q_1,\ldots,\q_s\in H$. On the other hand, to show that condition ($ii.4$), in Proposition \ref{prop:primitive}, is holding notice that, for every $i\in P$, condition ($ii$) in Definition \ref{def:intbord} ensures that
$$\g\subseteq \langle\q_i\rangle+\mathcal{C}_{H}=\left\langle Q^{P\backslash\{i\}}\right\rangle$$
where $\mathcal{C}_{H}=H\cap\gkz$, as defined in Remark \ref{rem:conobordante}.

For the converse, let us assume $\pc=\{V_P\}$ be a bordering primitive collection w.r.t. the hyperplane $H$. By (\ref{Nef}), $H$ turns out to cut out a face of $\Nef(X)$. If $H$ cut out a facet of $\Nef(X)=\g$ then $\g$ turns out to be maxbord w.r.t. $H$, hence intbord. Let us then assume that $$\dim(H\cap\Nef(X))\leq r-2\,.$$
This means that the numerical class $\n_P$ is not extremal in the decomposition (\ref{Mori}) of the Mori cone. Then there exist $l\geq 2$ extremal classes $\n_1,\ldots,\n_{l}\in\partial\overline{\NE}(X)$ such that $\n_P=\sum_{k=2}^{l}\mu_k \n_k$, for some $\mu_k> 0$. Let $H_k\subseteq F^r_{\R}$ be the dual hyperplane to $\n_k$, for $1\leq k\leq l$, which by construction cuts out a facet of $\g$. Since we are assuming $\g$ be bordering w.r.t. $H$, then $\g\cap H=\g\cap \left(\bigcap_{k=2}^l H_k\right)$. Hence, by Definition \ref{def:intbord}, if $\g$ would not be intbord w.r.t. $H$ then
\begin{equation*}
    \forall\,k=1,\ldots,l\ \text{either}\ \forall\,j\not\in P\quad\n_k\cdot\q_j\leq 0\ \text{or}\ \forall\,j\not\in P\quad\n_k\cdot\q_j\geq 0\,.
\end{equation*}
In particular, since $H_k\neq H$, there exists $j\not\in P$ such that $\n_k\cdot\q_{j}\neq 0$. But $\q_j\in H$, giving
\begin{eqnarray*}
    0=\n_P\cdot\q_{j}=\sum_{k=2}^{l}\mu_k \n_k\cdot\q_{j}&\Longrightarrow&\exists\,1\leq k_0\leq l : \n_{k_0}\cdot\q_{j_0}<0\\
                                                         &\Longrightarrow&\forall\,j\not\in P\ \n_{k_0}\cdot\q_j\leq 0\,.
\end{eqnarray*}
By construction there exists $i\in P$ such that  $\q_i\in H_{k_0}\cap\pc$. Then $\langle Q^{P\backslash\{i\}}\rangle=\langle\q_i\rangle+\mathcal{C}_{H}$. On the one hand $\g\subseteq\langle Q^{P\backslash\{i\}}\rangle$, by condition $(ii.4)$ of Proposition \ref{prop:primitive}. Therefore there exists $\mathbf{x}\in\g\subseteq\langle Q^{P\backslash\{i\}}\rangle$ such that $\n_{k_0}\cdot \mathbf{x}>0$, being $\n_{k_0}$ the primitive inward normal vector to the facet $H_{k_0}$ of $\g$. On the other hand
\begin{equation*}
    \forall\,\mathbf{x}\in\langle Q^{P\backslash\{i\}}\rangle=\langle\q_i\rangle+\mathcal{C}_{H}\quad \n_{k_0}\cdot\mathbf{x} = \l_i\n_{k_0}\cdot\q_i+\sum_{j\not\in P}\l_j\n_{k_0}\cdot\q_j=\sum_{j\not\in P}\l_j\n_{k_0}\cdot\q_j\leq 0
\end{equation*}
giving a contradiction. Then $\g$ has to be intbord w.r.t. $H$.
\end{proof}

\subsubsection{Notation}\label{ssez:Hi} Calling $x_1,\ldots,x_r$ the coordinates of $F^r_{\R}=\R^r$, in the following  $H_i$ denotes the coordinate hyperplane $x_i=0$; in particular $H_r:=\{x_r=0\}$.

\begin{corollary}\label{cor:intbord+reg} Let $V$ be a reduced $F$--matrix and $\g\in\mathcal{A}_{\Ga}(V)$ be an intbord and non--singular chamber. Then the as\-so\-cia\-ted fan $\Si_{\g}\in\P\SF(V)$ admits a nef primitive collection $\pc$ whose primitive relation (\ref{r_Z(P)}) has all the non--zero coefficients equal to 1.
\end{corollary}

\begin{proof} This is immediate after \cite[Prop.~3.2]{Batyrev91}. Alternatively, the previous Theorem \ref{thm:intbord-pc} gives a bordering, hence nef, primitive collection $\pc=\{V_P\}\subset\Si_{\g}(1)$ with $P=\{s+1,\ldots,n+r\}$. By the following Lemma \ref{lm:dibase}, one can always assume the bordering hyperplane $H_P$ be given by $H_r:=\{x_r=0\}$, recalling notation \ref{ssez:Hi}, and $Q$ be a positive, $\REF$, $W$--matrix. By (\ref{normalP}), this means that
\begin{equation*}
    r_{\Z}(P)=Q^T\cdot\n_P=Q^T\cdot\left(
                                     \begin{array}{c}
                                       0 \\
                                       \vdots \\
                                       0 \\
                                       1 \\
                                     \end{array}
                                   \right)=\left(0,\ldots,0,q_{r,s+1},\ldots,q_{r,n+r}\right)
\end{equation*}
i.e. the bottom row of $Q$ gives the primitive relation $r_{\Z}(P)$.
Then condition $(ii)$ of Proposition \ref{prop:primitive} gives that, for every $s+1\leq i\leq n+r$, the column $\q_i$ of the weight matrix $Q$ is always a generator of the simplicial cone $\mathcal{C}_{i,P}\in\mathcal{B}(\g)$, whose determinant is necessarily a multiple of the entry $q_{r,i}$ of $Q$. The non-singularity of $\g$ then imposes $q_{r,i}=1$, for all $i\in P$.
\end{proof}

\begin{example}[Examples \ref{ex:1} and \ref{ex:2} continued]\label{ex:3} Consider the two isomorphic in codimension 1, smooth and projective toric varieties $X(\Si_1),X(\Si_2)$, defined in Example \ref{ex:1}. Their chambers (i.e. Nef cones), $\g_1,\g_2$, respectively, are described in Example \ref{ex:2} and Fig.~\ref{fig2}. From the latter it is evident that both the chambers are intbord w.r.t. both the hyperplanes $H_2$ and $H_3$, under notation \ref{ssez:Hi}, and moreover $\g_1$ is maxbord w.r.t. these hyperplanes. Thm.~\ref{thm:intbord-pc} then gives that $H_2$ and $H_3$ are supporting two collections, $\pc_2=\{\v_3,\v_4\}$ and $\pc_3=\{\v_5,\v_6\}$, respectively, which are primitive and nef for both the fans $\Si_1$ and $\Si_2$.

\end{example}

\begin{lemma}\label{lm:dibase} Let $H$ be a hyperplane cutting a facet of $\gkz$. Then there exist $\a\in\GL_r(\Z)$ and a permutation matrix $\b\in\GL_{n+r}(\Z)$ such that $\a\,Q\,\b$ is in $\REF$ and $H$ is sent to the hyperplane $H_r$.
\end{lemma}

\begin{proof} Since $H$ cuts a facet of $\gkz$, up to a permutation of columns, one can assume that the first $s\geq r-1$ columns $\q_1,\ldots,\q_s$ are all the columns of $Q=\G(V)$ belonging to $H$. Consider $\a'\in\GL_r(\Z)$ and a permutation matrix $\b'\in\GL_{s}(\Z)$ such that
$\a'\,Q_{\{1,\ldots,s\}}\,\b'$ is $\REF$. Since there cannot exist $r$ linearly independent vectors among $\q_1,\ldots,\q_s$, the last $r$--th row of $\a'\,Q_{\{1,\ldots,s\}}\,\b'$ has to be $\mathbf{0}$, meaning that $H$ has been sent to $H_r$. In particular the primitive inward normal vector of $H$ has been transformed to $(0,\ldots,0,\pm 1)$. Therefore
\begin{equation*}
     \alpha' \,Q\,\begin{pmatrix} \beta' & \mathbf{0}^T\\ \mathbf{0}& \mathbf{I}_{n+r-s}\end{pmatrix} = \overbrace{\left(\begin{array}{c}
                   \\
                 \REF \\
                 \\
               \end{array}\right.}^{s}\overbrace{\left.\begin{array}{c}
                  \vdots \\
                 q'_{r,s+1}\cdots q'_{r,n+r}\\
               \end{array}\right)}^{n+r-s}
\end{equation*}
with $q'_{r,s+1},\ldots, q'_{r,n+r}$ either strictly positive or strictly negative integer entries, depending on the sign of $(0,\ldots,0,\pm 1)$. The proof then ends up, after a possible change of sign of the bottom $r$-th row, by  summing up suitable multiples of this latter row to the upper ones and by reordering the last $n+r-s$ columns to finally get a $\REF$ matrix.
\end{proof}

\begin{remark}\label{rem:intbord->smooth} The previous Corollary \ref{cor:intbord+reg} would give an alternative proof of Ba\-ty\-rev's result \cite[Prop.~3.2]{Batyrev91} if it would be possible proving that \emph{a non--singular chamber is a bordering chamber}. Moreover Theorem \ref{thm:intbord-pc} would give a generalization of this Batyrev's result to a singular setup. Actually this is the case for Picard number $r\leq2$. In fact for $r=1$ every chamber is maxbord. For $r=2$, a non--singular chamber $\g$ is maxbord, hence bordering. For $r=3$, in \cite{RT2} we prove that a non--singular chamber is bordering by assuming the existence of a nef primitive collection, hence assuming \cite[Prop.~3.2]{Batyrev91}: this fact gives strong geometric consequences on smooth projective toric varieties of rank $r\leq 3$. Unfortunately for $r\geq 4$ non--singular chambers which are not bordering may exist: in fact, recalling Remark \ref{rem:bordering->nef}(2), in \cite{FS} Fujino and Sato exhibited examples of smooth projective toric varieties with $r\geq 5$ whose non-trivial nef divisors are big. We improved this result in \cite[\S~4.3, 4.4]{RT2} to the case of Picard number $r=4$.

For further comments, evidences and details, we refer the interested reader to the above mentioned \cite{RT2}.
\end{remark}

\subsection{Toric bundles and covers}\label{ssez:t-bundles&covers}

The present subsection is dedicated to in\-tro\-du\-cing the main objects useful for the Batyrev--type classification of the next \S~\ref{ssez:maxbord&WPTB}.

\subsubsection{Weighted Projective Toric Bundles (WPTB)}\label{sssez:WPTB} We will adopt an obvious ge\-ne\-ra\-lization of notation and terminology given in \cite{CLS} \S~7.3 for a projective toric bundle (PTB).

\noindent Let $X'(\Si')$ be a $n'$--dimensional $\Q$--factorial complete toric variety, of rank $r'$, and consider $s+1$ Cartier divisors $E_0,\ldots,E_s$ and the associated locally free sheaf $\mathcal{E}=\bigoplus_{k=0}^s\mathcal{O}_{X'}(E_k)$ of rank $s+1$. Let $W=(w_0,\ldots,w_s)$ be a reduced $1\times(s+1)$ $W$--matrix and consider the $W$--weighted symmetric algebra $S^W(\mathcal{E})$: if $\mathcal{E}$ is locally free then $S^W(\mathcal{E})$ is locally free, too. The bundle $\P^W(\mathcal{E})\rightarrow X'$, defined by setting
\begin{equation*}
    \P^W(\mathcal{E}):=\mathbf{Proj}\left(S^W(\mathcal{E})\right)=\mathbf{Proj}\left(S^W\left(\bigoplus_{k=0}^s\mathcal{O}_{X'}(E_k)\right)\right)
\end{equation*}
is called \emph{the ($W$--)weighted projective toric bundle (WPTB) associated with $\mathcal{E}$}. Its fibers look like the $s$--dimensional weighted projective space $\P(w_0,\dots,w_s)$ and it turns out to be a $\Q$--factorial complete toric variety whose fan is described as follows.

\noindent Let $\Si_W\subset N_{W,\R}\cong\R^s$ be a fan of $\P(W)$. Then its 1--skeleton $\Si_W(1)$ is composed by $s+1$ rays whose primitive generators are $s+1$ integer vectors $\e_0,\e_1,\ldots,\e_s\in\Z^s$ such that
$$\sum_{k=0}^s w_k\e_k=0\quad\text{and}\quad\forall\,0\leq i\leq s\ |\det(\e_0,\ldots,\widehat{\e}_{i},\ldots,\e_s)|=w_i\,.$$
(see e.g. \cite[Thm.~3]{RT-WPS}).
The fan $\Si_W$ is then composed by the cones
\begin{equation}\label{coniW}
    F_i:=\langle\e_0,\e_1\ldots,\widehat{\e}_i,\ldots,\e_s\rangle\subset\R^s ,\quad 1\leq i\leq s\,,
\end{equation}
and all their faces.
Consider now the fan defining $X'$, given by $\Si'\subset N_{X',\R}\cong\R^{n'}$. Let $V'=(\v'_1,\ldots,\v'_{n'+r'})$ be a $n'\times(n'+r')$ fan matrix of $X'$. Let $D'_j$ be the torus invariant Weil divisor associated with the ray $\langle\v'_j\rangle\in\Si(1)$. Then
\begin{equation*}
    \forall\,k=0,\ldots,s\quad E_k=\sum_{j=1}^{n'+r'}a_{kj}D'_j\ .
\end{equation*}
For a cone $\s\in\Si'$ define the \emph{fibred} cone
\begin{equation}\label{cono fibrato}
    \s_i:=\left\langle \left\{\left(\begin{array}{c}
                           \v'_j \\
                           \mathbf{0}_{s,1}
                         \end{array}\right)
     + \sum_{k=0}^s a_{kj}\left(\begin{array}{c
     }
                            \mathbf{0}_{n',1} \\
                             \e_k
                          \end{array}\right)
     \ |\ \langle\v'_j\rangle\in\s(1)\right\}\right\rangle + F_i \subset N_{X',\R}\times N_{W,\R}\cong\R^{n'+s}
\end{equation}

\begin{proposition}\label{prop:fan fibrato}
The set of fibred cones (\ref{cono fibrato}) and all their faces give rise to a fan $\Si_{W,\mathcal{E}}\subset N_{X',\R}\times N_{W,\R}$ whose toric variety is the $W$--weighted projective toric bundle $\P^W(\mathcal{E})$.
\end{proposition}

The fibred cone (\ref{cono fibrato}) is the analogue of the cone (7.3.3) in \cite{CLS} giving the fan of a projective toric bundle. The proof of the previous proposition is then the same of \cite[Prop.~7.3.3]{CLS}.

Let $V$ be a fan matrix of $\P^W(\mathcal{E})$: setting $r=r'+1$ and $n=n'+s$, by (\ref{cono fibrato}), $V$ can be chosen to be the following $n\times(n+r)$ matrix
\begin{equation}\label{Vfibrata}
    V=\overbrace{\left(\begin{array}{ccc}
                 &V'& \\
                 \sum_{k=0}^sa_{k,1}e_{1k}&\cdots&\sum_{k=0}^sa_{k,n'+r'}e_{1k}\\
                 &\cdots&\\
                 \sum_{k=0}^sa_{k,1}e_{sk}&\cdots&\sum_{k=0}^sa_{k,n'+r'}e_{sk}\\
               \end{array}\right.}^{n'+r'}\overbrace{\left.\begin{array}{c}
               \mathbf{0}
                   \\
                   \\
                 \e_0\cdots \e_s\\
                 \\
               \end{array}\right)}^{s+1}
\end{equation}
By Gale duality, a weight matrix of $\P^W(\mathcal{E})$ is then given by the following $r\times(n+r)$ matrix
\begin{equation}\label{Qfibrata}
    Q=\G(V)=\overbrace{\left(\begin{array}{c}
                 Q' \\
                 0\cdots0\\
               \end{array}\right.}^{n'+r'}\overbrace{\left.\begin{array}{c}
               Q'' \\
                 w_0\cdots w_s\\
               \end{array}\right)}^{s+1}
\end{equation}
where $Q'=\G(V')$ and the $r'\times(s+1)$ matrix $Q''$ is defined by observing that
\begin{equation*}
    Q'\,V''^T +Q''\,\left(
                  \begin{array}{c}
                    \e_0^T \\
                    \vdots \\
                    \e_s^T \\
                  \end{array}
                \right) = 0
\end{equation*}
where $V''$ is the $s\times(n'+r')$ matrix whose $(i,j)$--entry is given by $\sum_{k=0}^sa_{k,j}e_{ik}$. Therefore $Q''=(b_{h,k+1})$ with
\begin{equation*}
    b_{h,k+1}=-\sum_{j=1}^{n'+r'}q'_{hj}a_{kj}\quad\text{for $0\leq k\leq s$\,,}
\end{equation*}
where $(q'_{h,j})=Q'$. Recalling the morphism $d':\mathcal{W}_T(X'(\Si'))\to\Cl(X'(\Si'))$ introduced in Proposition \ref{prop:nef}, this means that the $(k+1)$--st column of $Q''$ is given by
\begin{equation}\label{colonneQ''}
    \mathbf{b}_{k+1}=-d'(E_k)\quad\text{for $0\leq k\leq s$\,,}
\end{equation}
where $d'(E_k)$ is the class of $E_k$ in $\Cl(X')$.

\subsubsection{Toric covers}\label{sssez:Tcovers} Let us recall that, given two lattices $N$ and  $\widetilde{N}$ with two fans $\Si\subset N_{\R}$ and $\widetilde{\Si}\subset \widetilde{N}_{\R}$, a $\Z$--linear map $\overline{f}:N\to \widetilde{N}$ is called \emph{compatible} with the given fans if
$$ \forall\,\s\in\Si\quad \exists\,\widetilde{\s}\in\widetilde{\Si}:\overline{f}_{\R}(\s)\subseteq\widetilde{\s} $$
where $\overline{f}_{\R}:N_R\to\widetilde{N}_R$ is the natural $\R$--linear extension of $\overline{f}$ (see \cite[Def.~3.3.1]{CLS}).

For the following notion of \emph{toric cover} we refer the interested reader to \cite[\S~3]{AP}.

\begin{definition}\label{def:toricover} A \emph{toric cover} $f:X(\Si)\to \widetilde{X}(\widetilde{\Si})$ is a finite morphism of toric varieties inducing a $\Z$--linear map $\overline{f}:N\to\widetilde{N}$, compatible with $\Si$ and $\widetilde{\Si}$, such that:
\begin{enumerate}
  \item $ \overline{f}(N)\subseteq\widetilde{N}$ is a subgroup of finite index, so that $\overline{f}(N)\otimes\R=\widetilde{N}\otimes \R$,
  \item $\overline{f}_{\R}(\Si)=\widetilde{\Si}$.
\end{enumerate}
\end{definition}

\begin{lemma}[\cite{AP} Lemma 3.3]\label{lm:AP} A toric cover $f:X(\Si)\to \widetilde{X}(\widetilde{\Si})$ has the following properties:
\begin{enumerate}
  \item $f$ is an \emph{abelian cover} with Galois group $G\cong \widetilde{N}/\overline{f}(N)$,
  \item $f$ is ramified only along the torus invariant divisors $D_{\rho}$, with multiplicities $d_{\rho}\geq 1$ defined by the condition that the integral generator of $\overline{f}(N)\cap\langle\v_{\rho}\rangle$ is $d_{\rho}\v_{\rho}$, for every ray $\rho=\langle\v_{\rho}\rangle\in\widetilde{\Si}(1)$.
\end{enumerate}
\end{lemma}

\subsubsection{Weighted Projective Toric weak Bundles (WPTwB)}\label{sssez:WPTwB} Let us first of all notice that Proposition \ref{prop:fan fibrato} holds regardless of whether divisors $E_k=\sum_{j=1}^{n'+r'}a_{kj}D'_j$, for $0\leq k\leq s$, are truly Cartier divisors or instead, more generally, Weil divisors. Therefore the following natural question arises: which kind of geometric structures supports the toric variety associated with the simplicial complete fan given by Prop.~\ref{prop:fan fibrato} in the case of Weil non--Cartier divisors $E_k$'s?

The answer gives a nice account of both the previous subsections \ref{sssez:WPTB} and \ref{sssez:Tcovers}.

Let us recall that, given a Weil divisor $D$ on a $\Q$--factorial variety, the \emph{Cartier index} of $D$ is the least positive integer $c(D)\in\N$ such that $c(D)\,D$ is a Cartier divisor.

\begin{proposition}\label{prop:WPTwB} Let $V'$ be a $n'\times(n'+r')$ $CF$--matrix and $\Si'\in\SF(V')$. Consider the set $\Si$ of fibred cones (\ref{cono fibrato}) and all their faces and assume that
\begin{equation*}
    \forall\,0\leq k\leq s\quad E_k=\sum_{j=1}^{n'+r'}a_{kj}D'_j\in\Weil(X')\quad \text{whose Cartier index is}\quad l_k:=c(E_k).
\end{equation*}
Then $\Si$ is a simplicial and complete fan whose associated toric variety $X(\Si)$ is a toric cover of the WPTB $\P^{W'}\left(\mathcal{E}\right)$, where $W'$ is the reduced weight vector of $(l_0w_0,\ldots,l_sw_s)$ and
$\mathcal{E}=\bigoplus_{k=0}^s\mathcal{O}_{X'}\left(\eta_k E_k\right)$,
with $\eta_k$ defined by setting
\begin{eqnarray*}
% \nonumber to remove numbering (before each equation)
  \l &:=& \gcd(l_0w_0,\ldots,l_sw_s) = \gcd(l_0,\ldots,l_s)\quad\text{\emph{(since $\gcd(w_0,\ldots,w_s)=1$)}}\\
  d_k &:=& \gcd\left(\frac{l_0w_0}{\l},\ldots,\widehat{\frac{l_{k}w_{k}}{\l}},\ldots,\frac{l_sw_s}{\l}\right)\\
  &=& \gcd\left(\frac{l_0}{\l},\ldots,\widehat{\frac{l_{k}}{\l}},\ldots,\frac{l_s}{\l}\right)\quad\text{\emph{(since $W$ is reduced)}} \\
  a_k &:=& \lcm\left(d_0,\ldots,\widehat{d}_{k},\ldots,d_{s}\right) = \frac{\prod_{i=0}^s d_i}{d_k}\quad\text{\emph{(by \cite[Prop.~3(2)]{RT-WPS})}}\\
  a   &:=& \lcm(a_0,\ldots,a_s)=\prod_{k=0}^s d_k\quad\text{\emph{(by \cite[Prop.~3(5)]{RT-WPS})}}\\
  \eta_k &:=& l_k a/a_k =l_k d_k
\end{eqnarray*}
In particular the toric cover $X(\Si)\to\P^{W'}(\mathcal{E})$ is an abelian covering admitting a Galois group $G$ of order
$$|G|=\frac{\prod_{k=0}^s l_k}{\l}$$
and ramified along the torus invariant divisors $D_{n'+r'+1+k}$, with multiplicity $\eta_k$, for every $0\leq k\leq s$.  In the following $X(\Si)$ is called a \emph{weighted projective toric weak bundle (WPTwB)}: it is a PWS.
\end{proposition}

\begin{proof} As for the proof of Prop.~\ref{prop:fan fibrato}, the fact that $\Si$ is a simplicial and complete fan, follows by the same argument proving \cite[Prop.~7.3.3]{CLS}.

\noindent Given Cartier indexes $l_k=c(E_k)\geq 1$, consider the diagonal matrix
$$\Lambda':=\left(
                                                          \begin{array}{cc}
                                                            I_{n'+r'} &  \mathbf{0}_{n'+r',s+1}\\
                                                            \mathbf{0}_{s+1,n'+r'} & \diag(l_0,\ldots,l_s) \\
                                                          \end{array}
                                                        \right)\in\GL_{n+r}(\Q)\cap \mathbf{M}_{n+r}(\Z)\,.$$
Then, recalling (\ref{Qfibrata}) and (\ref{colonneQ''}), one gets
\begin{equation*}
    Q\cdot \Lambda' =\overbrace{\left(\begin{array}{c}
                 Q' \\
                 0\cdots0\\
               \end{array}\right.}^{n'+r'}\overbrace{\left.\begin{array}{ccc}
               -d(l_0E_0)&\cdots&-d(l_sE_s) \\
                 l_0w_0&\cdots&l_sw_s\\
               \end{array}\right)}^{s+1}
\end{equation*}
If the weight vector $(l_0w_0,\ldots,l_sw_s)$ is reduced then set $W'=(l_0w_0,\ldots,l_sw_s)$ and $Q\cdot\Lambda'$ turns out to be a weight matrix of $\P^{W'}\left(\bigoplus_{k=0}^s\mathcal{O}_{X'}\left(l_k E_k\right)\right)$.

\noindent If $(l_0w_0,\ldots,l_sw_s)$ is not reduced then define $\l$, $d_k$, $a_k$ and $a$ as in the statement and consider the matrices
\begin{eqnarray*}
% \nonumber to remove numbering (before each equation)
  \Delta &:=& \diag\left(1,\ldots,1,\frac{1}{\l a}\right)\in\GL_r(\Q)  \\
  \Lambda'' &:=& \left(\begin{array}{cc}
                                                            I_{n'+r'} &  \mathbf{0}_{n'+r',s+1}\\
                                                            \mathbf{0}_{s+1,n'+r'} & \diag\left(\frac{a}{a_0},\ldots,\frac{a}{a_s}\right) \\
                                                          \end{array}
                                                        \right)\in\GL_{n+r}(\Q)\cap \mathbf{M}_{n+r}(\Z)\\
  \Lambda  &:=&\Lambda'\cdot\Lambda'' =  \left(\begin{array}{cc}
                                                            I_{n'+r'} &  \mathbf{0}_{n'+r',s+1}\\
                                                            \mathbf{0}_{s+1,n'+r'} & \diag\left(\frac{l_0a}{a_0},\ldots,\frac{l_sa}{a_s}\right) \\
                                                          \end{array}
                                                        \right)\\
 \widetilde{Q}&:=&\Delta\cdot Q\cdot\Lambda = \left(\begin{array}{cccc}
                 Q' &-d(\frac{l_0a}{a_0}E_0)&\cdots&-d(\frac{l_sa}{a_s}E_s)\\
                 0\cdots0 & \frac{l_0w_0}{\l a_0}&\cdots&\frac{l_sw_s}{\l a_s}\\
               \end{array}\right)\\
              & =& \left(\begin{array}{cccc}
                 Q' &-d(\eta_0E_0)&\cdots&-d(\eta_sE_s)\\
                 0\cdots0 & w'_0&\cdots&w'_s\\
               \end{array}\right)
\end{eqnarray*}
where $W'=(w'_0,\ldots,w'_s)$ is the reduced weight vector of $W$ and $\eta_k=l_ka/a_k$. Then $\widetilde{Q}$ turns out to be a weight matrix of $\P^{W'}\left(\bigoplus_{k=0}^s\mathcal{O}_{X'}\left(\eta_k E_k\right)\right)$.

\noindent Recalling $(\ref{Vfibrata})$, the fan matrix $V$ is a $CF$--matrix since $V'$ is a $CF$--matrix. Then $X(\Si)$ is a PWS and $\Cl(X)$ is a free $\Z$--module. Then the dualized divisors' exact sequence (\ref{HomZ-div-sequence}) is exact on the right, too, hence giving the following short exact sequence of free abelian groups
\begin{equation}\label{HomZ-div-sequence_exact}
                    \xymatrix{0 \ar[r] & \Hom(\Cl(X),\Z) \ar[r]^-{d^{\vee}} & \Hom(\mathcal{W}_T(X),\Z)
 \ar[r]^-{div^{\vee}} & N \ar[r]&  0}
\end{equation}
Fixing once for all a basis of $M\cong\Z^n$, the basis $\{D_j\}_{j=0}^{n+r}$ of $\Weil(X)\cong\Z^{n+r}$, a basis of $F^r=\Cl(X)\cong\Z^r$ and their dual bases, then $Q^T$ and $V$ turns out to be representative matrices of morphisms $d^{\vee}$ and $div^{\vee}$, respectively. Then we get the following commutative diagram of exact sequences
\begin{equation}\label{covering-diagram-1}
    \xymatrix{&0\ar[d]&0\ar[d]&0\ar[d]&\\
    0\ar[r]&\Hom(\Cl(X),\Z)\cong\Z^r\ar[d]_-{d^{\vee}}^-{Q^T}\ar[r]^-{(\Delta^{-1})^T}&\Z^r\ar[d]^{\widetilde{Q}^T}\ar[r]&\Z/(\l a)\Z\ar[d]\ar[r]&0\\
    0\ar[r]&\Hom(\mathcal{W}_T(X),\Z)\cong\Z^{n+r}\ar[r]^-{\Lambda^T}\ar[d]_-{div^{\vee}}^-{V}&\Z^{n+r}\ar[r]\ar[d]^{\widetilde{V}}&
    \left(\bigoplus_{k=0}^s\Z/\eta_k\Z\right) \ar[r]\ar[d]&0\\
    0\ar[r]&N\cong\Z^n\ar[r]^-{\Phi^T}_-{\overline{f}}\ar[d]&N_{X'}\times N_{W'}\cong\Z^n\ar[r]\ar[d]&G\ar[r]\ar[d]&0\\
    &0&0&0&}
\end{equation}
where the matrix $\Phi$ is obtained as follows:
\begin{itemize}
  \item $V$ is a $CF$--matrix if and only if  $H:=\HNF(V^T)=\left(
                                                                 \begin{array}{c}
                                                                   I_n \\
                                                                   \mathbf{0}_{r,n} \\
                                                                 \end{array}
                                                               \right)$ \cite[Thm.~2.1(4)]{RT-QUOT},
  \item let $U\in \GL_{n+r}(\Z)$ such that $U\cdot V^T = H$,
  \item then the upper $n$ rows of $U$ give $^nU\cdot V^T = I_n$ (recall notation in list \ref{ssez:lista})
\end{itemize}
Therefore
\begin{equation*}
    V^T\cdot\Phi=\Lambda\cdot \widetilde{V}^T\ \Longrightarrow\ \Phi=\ ^nU\cdot  \Lambda\cdot \widetilde{V}^T\,.
\end{equation*}
From diagram (\ref{covering-diagram-1}), the $\Z$-linear morphism $\overline{f}:N\to N_{X'}\times N_{W'}$, represented by $\Phi$, is clearly injective, giving rise to the toric cover we were looking for, whose Galois group is given by $G$ in the same diagram. The exactness of the vertical sequence on the right implies that
$$|G|=\frac{\prod_{k=0}^s\eta_k}{\l a} = \frac{\prod_{k=0}^s l_k d_k}{\l a} = \frac{\prod_{k=0}^s l_k}{\l}\,,$$
while the ramification is given by the matrix $\Lambda^T=\Lambda$.
\end{proof}

\begin{remark}\label{rem:fan_fibrato} The hypothesis that $V'$ is a $CF$-matrix is essential in proving Proposition \ref{prop:WPTwB}. In fact, recalling (\ref{Vfibrata}) if $V'$ is a $F$ non--$CF$--matrix then $V$ is a $F$ non--$CF$--matrix, too. Then the dual exact sequence (\ref{HomZ-div-sequence}) is not exact on the right, meaning that the morphism $\overline{f}$ in diagram (\ref{covering-diagram-1}) may not exist.

\noindent On the other hand the set $\Si$ of fibred cones (\ref{cono fibrato}) and all their faces still turns out to be a simplicial complete fan, due to the same argument proving \cite[Prop.~7.3.3]{CLS}. It remains then open the geometric interpretation of this case, for which the reader is referred to \S~\ref{sez:Qfproj} and in particular to Theorem~\ref{thm:quot-maxbord} and Remark~\ref{rem:fan_fibrato_2}.
\end{remark}

\subsection{Maximally bordering chambers and WPTB}\label{ssez:maxbord&WPTB}
The present subsection is devoted to generalize Batyrev's results given in \cite[\S~4]{Batyrev91}, by dropping the smoothness hypothesis.
Let us first of all notice the following useful fact:

\begin{lemma}\label{lm:cover&bord} Let $V$ and $\widetilde{V}$ be $n\times(n+r)$ reduced $F$--matrices such that $f:X(\Si)\to \widetilde{X}\left(\widetilde{\Si}\right)$ is a toric cover for some $\Si\in\SF(V)$ and $\widetilde{\Si}\in\SF\left(\widetilde{V}\right)$. Then $\g_{\Si}\in\Ga(V)$ is a maxbord chamber if and only if $\widetilde{\g}_{\widetilde{\Si}}\in\Ga\left(\widetilde{V}\right)$ is a maxbord chamber.
\end{lemma}

\begin{proof} Recalling the Definition \ref{def:toricover} of a toric cover, the induced $\Z$--linear morphism $\overline{f}:N\to \widetilde{N}$, which is compatible with the fans $\Si$ and $\widetilde{\Si}$, gives actually an equality of fans $\overline{f}_{\R}(\Si)=\widetilde{\Si}$. Then for every cone $\s\in\Si$ define $\widetilde{\s}:=\overline{f}_{\R}(\s)\in\widetilde{\Si}$.
By Theorem \ref{thm:intbord-pc}, there exists a bordering primitive collection $\pc=\{V_P\}$ of $\Si$, for some $P\in\mathfrak{P}$, whose support hyperplane $H_P$ is the maximally bordering hyperplane of $\g$. Then, by Proposition \ref{prop:maxbord},
\begin{eqnarray*}
    \forall\,\langle V_I\rangle\in\Si(n)\ \exists!\,i\in P&:&\langle Q^I\rangle=\langle\q_{i}\rangle+\langle Q^I\rangle\cap H_P\\
                                            &\Longleftrightarrow& \langle V_I\rangle= \langle V_{P\backslash\{i\}}\rangle+\langle V_{I\backslash P}\rangle
\end{eqnarray*}
On the other hand $\widetilde{\pc}:=\left\{\overline{f}_{\R}(\v_i)\,|\,i\in P\right\}=\left\{\widetilde{V}_{\widetilde{P}}\right\}$ turns out to be a primitive collection for $\widetilde{\Si}=\overline{f}_{\R}(\Si)$. Up to a permutation on columns of $\widetilde{V}$ we can assume $\widetilde{P}=P$. Then
\begin{eqnarray*}
    \forall\,\left\langle \widetilde{V}_{I}\right\rangle\in\widetilde{\Si}(n)\ \exists!\,i\in P&:&\left\langle\widetilde{V}_{I}\right\rangle =
    \overline{f}_{\R}\left(\langle V_I\rangle\right)=\overline{f}_{\R}\left(\langle V_{P\backslash\{i\}}\rangle\right)+\overline{f}_{\R}\left(\langle V_{I\backslash P}\rangle\right)\\
    &&\hskip0.9truecm=\left\langle\widetilde{V}_{P\backslash\{i\}}\right\rangle+
    \left\langle\widetilde{V}_{P\backslash I}\right\rangle\\
    &\Longleftrightarrow&\left\langle \widetilde{Q}^{I}\right\rangle=\left\langle\widetilde{\q}_{i}\right\rangle+\left(\left\langle \widetilde{Q}^{I}\right\rangle\cap \widetilde{H}_{P}\right)
\end{eqnarray*}
which is enough, by Proposition \ref{prop:maxbord}, to show that $\widetilde{\g}_{\widetilde{\Si}}$ is maxbord w.r.t. the support $\widetilde{H}_{P}$ of $\widetilde{\pc}$.

The converse can be proved in the same way by observing that $\Si=\overline{f}_{\R}^{-1}\left(\widetilde{\Si}\right)$.
\end{proof}

We are now in a position to state and prove the following generalization of \cite[Prop.~4.1]{Batyrev91}.

\begin{theorem}\label{thm:maxbord} Given a reduced $n\times(n+r)$ $CF$--matrix $V$, with $r\geq 2$, a chamber $\g\in\mathcal{A}_{\Gamma}(V)$ is maximally bordering if and only if the associated PWS $X(\Si_{\g})$ is a toric cover of a weighted projective toric bundle $\P^W(\mathcal{E})$.
\end{theorem}

\begin{proof} Recalling Definition \ref{def:bordering} and Lemma \ref{lm:dibase} we can assume $\g$ be a maxbord chamber w.r.t. $H_r$, hence giving $\dim(\g\cap H_r)=r-1$. Since $\g$ is maxbord, it is intbord and Theorem \ref{thm:intbord-pc} implies that, after suitable transformations, the reduced $W$--matrix $Q=\G(V)$ can be set in $\REF$ with the bottom $r$-th row giving a primitive relation for $\Si_{\g}$,  $\pc=\{\v_{n+r-s},\ldots,\v_{n+r}\}$, on the last $s+1$ columns of $V$ (here we are exchanging each other the roles of $s$ and $n+r-s$ with respect to the proof of Thm. \ref{thm:intbord-pc}). Then $Q$ looks like (\ref{Qfibrata}) where $Q'$ is $(r-1)\times(n+r-s-1)$ matrix in $\REF$.

\noindent First of all let us notice that $Q'$ can be thought of as a $W$--matrix of a $n'=(n-s)$--dimensional variety, with the exception of condition b in Definition \ref{def:Wmatrix}. In fact, the $\REF$ form of $Q$ and the fact that $Q$ is a $W$--matrix imply immediately conditions a,c and d of Definition \ref{def:Wmatrix} for $Q'$.

\noindent For what concerns condition e, let us observe that $\mathcal{L}_r(Q')$ cannot contain any vector of the form $(0,\ldots,0,q,0,\ldots,0)$. Otherwise, if the non--trivial entry $q$ is in the $i$--th position then the $i$--th column $\q'_i$ of $Q'$ cannot be in $\mathcal{L}_c(Q'^{\{i\}})$. Therefore $\dim\langle Q'^{\{i\}}\rangle\leq r-2$. On the other hand, by the $\REF$ of $Q$ and (\ref{Mov}), one gets
\begin{equation}\label{intersezione_i}
\g\cap H_r\subseteq \Mov(V)\cap H_r\subseteq \langle Q^{\{i\}}\rangle\cap H_r = \langle Q'^{\{i\}}\rangle
\end{equation}
where the last equality on the right comes from the fact that $\g$ is maxbord w.r.t. $H_r$, meaning that $H_r$ cuts a facet of $\Mov(V)$, hence a facet of $\langle Q^{\{i\}}\rangle$. Clearly (\ref{intersezione_i}) contradicts the maxbord hypothesis $\dim(\g\cap H_r)=r-1$.

\noindent The same argument applies to guarantee condition f of Definition \ref{def:Wmatrix} for $Q'$. In fact $\mathcal{L}_r(Q')$ cannot contain any vector of the form $(0,\ldots,0,a,0,\ldots,0,b,0,\ldots,0)$ with $ab<0$. Otherwise, if the non--trivial entries $a,b$ are in the $i$--th, $j$--th positions, respectively, then the $i$--th and the $j$--th columns $\q'_i,\q'_j$ of $Q'$ cannot be in $\mathcal{L}_c(Q'^{\{i,j\}})$. Therefore $\dim\langle Q'^{\{i,j\}}\rangle\leq r-2$. Moreover one also gets that
\begin{equation}\label{HPassurda}
    \forall\,\mu,\l\quad\mu\l>0\ \Longrightarrow\ \mu\q_i-\l\q_j\not\in\mathcal{L}_c(Q'^{\{i,j\}})
\end{equation}
because $\mu a-\l b\neq 0$. On the other hand, by the $\REF$ of $Q$ and (\ref{Mov}), one gets
\begin{equation}\label{intersezione_ij}
    \g\cap H_r\subseteq \Mov(V)\cap H_r\subseteq \langle Q^{\{i\}}\rangle\cap \langle Q^{\{j\}}\rangle\cap H_r = \langle Q'^{\{i\}}\rangle\cap \langle Q'^{\{j\}}\rangle
\end{equation}
where the last equality on the right comes from the fact that $H_r$ cuts a facet of both $\langle Q^{\{i\}}\rangle$ and $\langle Q^{\{j\}}\rangle$. Notice that if one proves that
\begin{equation}\label{intersezione_ij2}
    \langle Q'^{\{i\}}\rangle\cap \langle Q'^{\{j\}}\rangle=\langle Q'^{\{i,j\}}\rangle
\end{equation}
then (\ref{intersezione_ij}) turns out to contradicting the maxbord hypothesis $\dim(\g\cap H_r)=r-1$. Since clearly $\langle Q'^{\{i\}}\rangle\cap \langle Q'^{\{j\}}\rangle\supseteq\langle Q'^{\{i,j\}}\rangle$, to show (\ref{intersezione_ij2}) we need to prove that: if $\mathbf{x}\in\langle Q'^{\{i\}}\rangle\cap \langle Q'^{\{j\}}\rangle$ then $\mathbf{x}\in\langle Q'^{\{i,j\}}\rangle$. For this purpose consider the linear combinations with nonnegative coefficients
$$\mathbf{x}=\sum_{k\neq i,j}\l_k\q_k+\l_j\q_j=\sum_{k\neq i,j}\mu_k\q_k+\mu_i\q_i\in \langle Q'^{\{i\}}\rangle\cap \langle Q'^{\{j\}}\rangle\,.$$
This gives
$$\mu_i\q_i-\l_j\q_j=\sum_{k\neq i,j}(\l_k-\mu_k)\q_k$$
contradicting (\ref{HPassurda}) unless $\mu_i=\l_j=0$, which is $\mathbf{x}\in\langle Q'^{\{i,j\}}\rangle$.

We have now to consider three possible cases: (a) $Q'$ is a reduced $W$--matrix, (b) $Q'$ is a non--reduced $W$--matrix, (c) $Q'$ is not a $W$--matrix in the sense that $\mathcal{L}_r(Q')$ has cotorsion in $\Z^{n'+r'}$, with $n'=n-s$ and $r'=r-1$.

(a) Assume now that $Q'$ is a reduced $W$--matrix. Since $\g$ is maxbord then $\g':=\g\cap H_r$ is $(r-1)$--dimensional. Let us show that $\g'$ is actually a chamber contained in $\Mov(V')$, with $V'=\G(Q')$. For this purpose notice that (\ref{Mov}) and the fact that $H_r$ cuts out a facet of every $\langle Q^{\{i\}}\rangle$ give
\begin{equation*}
    \Mov(V')=\bigcap_{i=1}^{n+r-s-1}\left\langle Q'^{\{i\}}\right\rangle=H_r\cap\bigcap_{i=1}^{n+r}\left\langle Q^{\{i\}}\right\rangle
    = H_r\cap\Mov(V)
\end{equation*}
Hence $\g'= H_r\cap\g\subset H_r\cap\Mov(V)=\Mov(V')$.
Finally observe that Prop.~\ref{prop:maxbord} gives that every simplicial cone $\langle Q_I\rangle$ in the bunch of cones $\mathcal{B}(\g)$ has to contain the $(r-1)$--dimensional chamber $\g'=\g\cap H_r$, hence cutting out a simplicial cone $\langle Q_{I}\rangle\cap H_r=:\langle Q'_{I'}\rangle\in\mathcal{B}(\g')$ and admitting a unique ray generated by a column $\q_i$ of $Q$ not belonging to $H_r$ i.e.
\begin{equation*}
    \q_i\in\pc^*\quad,\quad I=I'\cup\{i\}\quad,\quad\langle Q_I\rangle=\langle Q'_{I'}\rangle+ \langle\q_i\rangle\,,
\end{equation*}
with $i\geq n+r-s$. Let us now consider Gale dualities w.r.t. $W$--matrices $Q$, $Q'$ and $W$, giving the Gale dual cones $\G(\langle Q_I\rangle)=\langle V^I\rangle$, $\G(\langle Q'_{I'}\rangle)=\langle (V')^{I'}\rangle$ and $\G(\langle\q_i\rangle)=F_i$, respectively, where the latter is precisely the cone $F_i$ defined in (\ref{coniW}). Then $\langle V^I\rangle=\langle (V')^{I'}\rangle+F_i$. Notice that, on the one hand, the set of cones $\langle (V')^{I'}\rangle$ and all their faces define a $(n-s)$--dimensional fan $\Si'_{\g'}$ and, on the other hand, the cones $F_i$, jointly with all their faces, give the fan $\Si_W$ of $\P(W)$. This suffices to show that $\Si_{\g}$ is \emph{split by $\Si'_{\g'}$ and $\Si_W$}, in the sense of \cite[Def.~3.3.18]{CLS}. Therefore \cite[Thm.~3.3.19]{CLS} gives a locally trivial fibre bundle
$    X(\Si_{\g})\twoheadrightarrow X'(\Si'_{\g'})$
whose fibers are all isomorphic to $\P(W)$.

\noindent It remains to prove that such a fiber bundle is actually a WPTwB, as defined in \ref{sssez:WPTwB}, hence a toric cover of a WPTB associated with some locally free sheaf $\mathcal{E}$. For this purpose add suitable negative multiples of the bottom row of $Q$ to the previous ones until one gets no positive entries in the $(i,j)$--positions with $1\leq i\leq r-1$ and $n+r-s\leq j\leq n+r$. These entries give the matrix $Q''$ in (\ref{Qfibrata}) whose columns give, up to a sign, the linear equivalence classes of some Weil divisors $E_0,\ldots,E_s$, as in (\ref{colonneQ''}). Consequently the Gale dual $\langle V^I\rangle$, of every cone $\langle Q_I\rangle\in\mathcal{B}(\g)$, turns out to be a fibred cone (\ref{cono fibrato}). Recalling \cite[Prop.~3.12(1)]{RT-LA&GD}, $V'=\G(Q')$ is a $CF$--matrix. This suffices to show that $X(\Si_{\g})$ is a WPTwB, by Proposition \ref{prop:WPTwB}. Then $X(\Si_{\g})$ is a toric cover of $\P^{W'}\left(\bigoplus_{k=0}^s\mathcal{O}_{X'(\Si'_{\g'})}(\eta_kE_k)\right)$, where $W'$ and $\eta_k$ are defined as in Proposition \ref{prop:WPTwB}. Notice that, by (\ref{Vfibrata}), $V'$ is a $CF$--matrix if and only if $V$ is a $CF$--matrix: then this case (a) can occur only if $V$ is a $CF$--matrix, as assumed in the statement.

(b) Assume now that $Q'$ is a non-reduced $W$--matrix (hence $\mathcal{L}_r(Q')$ has not cotorsion in $\Z^{n'+r'}$). Then $V'=\G(Q')$ admits some non--primitive column. Without loss of generality, up to a permutation on columns and an iteration of the following argument, assume that the unique non primitive column of $V'$ is the first one: namely $\v_1=d\w_1$, with $\w_1$ primitive. Consider the reduced weight matrix $Q'^{\text{red}}=\G(V'^{\text{red}})$ (see the list of notation \ref{ssez:lista}). The construction described in \cite[Thm.~3.15(3)]{RT-LA&GD} then gives
\begin{equation*}
    Q'^{\text{red}}=\diag(1,\ldots,1,1/d)\,\a_1\,Q'\,\diag(d,1,\ldots,1)
\end{equation*}
for a suitable $\a_1\in\GL_{r-1}(\Z)$. Define
\begin{eqnarray}\label{matrici}
\nonumber
    A&=&\left(
                    \begin{array}{cc}
                      \diag(1,\ldots,1,1/d)\,\a_1 & \mathbf{0}_{r-1,1} \\
                      \mathbf{0}_{1,r-1} & 1/d \\
                    \end{array}
                  \right)\in\GL_r(\Q)\\
    B&=& \left(
                                 \begin{array}{cc}
                                   \diag(d,1,\ldots,1) & \mathbf{0}_{n+r-s-1,s+1} \\
                                   \mathbf{0}_{s+1,n+r-s-1} & d\,I_{s+1} \\
                                 \end{array}
                               \right)\in\GL_{n+r}(\Q)\cap \mathbf{M}_{n+r}(\Z)\\
                               \nonumber
    \widetilde{Q}&=&A\,Q\,B\in \mathbf{M}(r,n+r;\Z)
\end{eqnarray}
Notice that the bottom $r$--th row of $\widetilde{Q}$ coincides with the bottom $r$--th row of $Q$. Then $\widetilde{Q}$ is reduced for the $\REF$ of $Q$ and the fact that $Q'^{\text{red}}$ is reduced. Define
$$\widetilde{V}=\left(\widetilde{\v}_1,\ldots,\widetilde{\v}_{n+r}\right):=\G(\widetilde{Q})\,.$$
Consider the dual divisors' sequence (\ref{HomZ-div-sequence}) for $X(\Si_\g)$ and notice that it turns out to be exact on the right, since $V$ is a $CF$--matrix, hence giving the short exact sequence (\ref{HomZ-div-sequence_exact}). Fixing the bases of the $\Z$-modules appearing in (\ref{HomZ-div-sequence_exact}), the matrices defined in (\ref{matrici}) define $\Z$--linear morphisms giving the following commutative diagram
\begin{equation}\label{covering-diagram}
    \xymatrix{&0\ar[d]&0\ar[d]&0\ar[d]&\\
    0\ar[r]&\Hom(\Cl(X))\cong\Z^r\ar[d]_-{d^{\vee}}^-{Q^T}\ar[r]^-{(A^{-1})^T}&\Z^r\ar[d]^{\widetilde{Q}^T}\ar[r]&\left(\Z/d\Z\right)^{\oplus 2}\ar[d]\ar[r]&0\\
    0\ar[r]&\Hom(\mathcal{W}_T(X))\cong\Z^{n+r}\ar[r]^-{B^T}\ar[d]_-{div^{\vee}}^-{V}&\Z^{n+r}\ar[r]\ar[d]^{\widetilde{V}}&\left(\Z/d\Z\right)^{\oplus s+2} \ar[r]\ar[d]&0\\
    0\ar[r]&N\cong\Z^n\ar[r]^-{C^T}\ar[d]&\Z^n\ar[r]\ar[d]&\left(\Z/d\Z\right)^{\oplus s}\ar[r]\ar[d]&0\\
    &0&0&0&}
\end{equation}
where $C^T$ is the representative matrix of an injective $\Z$--linear map $\overline{f}:\Z^n\to\Z^n$ easily defined by diagram chasing. Then
\begin{equation*}
    C^T\,V = \widetilde{V}\,B^T= \left(d\widetilde{\v}_1,\widetilde{\v}_2,\ldots,\widetilde{\v}_{n+r-s-1},d\widetilde{\v}_{n+r-s},\ldots,d\widetilde{\v}_{n+r}\right)
\end{equation*}
and $\overline{f}(N)$ is a subgroup of finite index $d^s$ of $\Z^n=\widetilde{N}$, as deduced by the vertical sequence on the right. Moreover, by Lemma~\ref{lm:cover&bord}, $\widetilde{\Si}:=\overline{f}_{\R}(\Si_{\g})$ is a fan in $\P\mathcal{SF}(\widetilde{V})$ defining a PWS $\widetilde{X}(\widetilde{\Si})$ whose weight matrix is $\widetilde{Q}$ and whose chamber $\widetilde{\g}_{\widetilde{\Si}}\in\Ga(\widetilde{V})$ is maxbord. Then, by the previous part (a), $\widetilde{X}$ is a WPTwB, which is a toric cover of a WPTB. Moreover $\overline{f}$ induces a toric cover $f:X\to\widetilde{X}$, hence $X$ is a toric cover of a WPTB.

\noindent Let us notice that in the present situation one can say something more than Lemma~\ref{lm:cover&bord}: in fact, the matrix $A^{-1}$ represents an injective $\Z$--linear map $g:\Cl(\widetilde{X})\to\Cl(X)$ which is compatible with the secondary fans $\Ga(\widetilde{V})$ and $\Ga(V)$ and gives $g_{\R}\left(\Ga(\widetilde{V})\right)=\Ga(V)$.

(c) Finally let us now assume that $\mathcal{L}_r(Q')$ has cotorsion in $\Z^{n'+r'}$. Without loss of generality, up to a permutation on rows and an iteration of the following argument, assume that $\a Q'$, for some $\a\in\GL_{r'}(\Z)$, admits a unique row giving cotorsion, namely the bottom $r'$--th row of $Q'$. Let $d>1$ be the greater common divisor of all entries in that row, i.e. $d=\gcd(q_{r',1},\ldots,q_{r',n'+r'})$. Recall the matrix $A$ given in (\ref{matrici}) and define
\begin{eqnarray*}
    A'&=&A\,\left(
                                 \begin{array}{cc}
                                   \a & \mathbf{0}_{r',1} \\
                                   \mathbf{0}_{1,r'} & 1 \\
                                 \end{array}
                               \right)\in\GL_{r}(\Q)\\
    B&=& \left(
                                 \begin{array}{cc}
                                   I_{n'+r'} & \mathbf{0}_{n'+r',s+1} \\
                                   \mathbf{0}_{s+1,n'+r'} & d\,I_{s+1} \\
                                 \end{array}
                               \right)\in\GL_{n+r}(\Q)\cap \mathbf{M}_{n+r}(\Z)\\
    \widetilde{Q}&=&A'\,Q\,B\in \mathbf{M}(r,n+r;\Z)
\end{eqnarray*}
Clearly the bottom $r$--th row of $\widetilde{Q}$ coincides with the bottom $r$--th row of $Q$. Moreover $\widetilde{Q}$ turns out to be a $W$--matrix for the $\REF$ of $Q$ and the fact that $\widetilde{Q}'=\diag(1,\ldots,1,1/d)\,\a_1\,\a\,Q'$ is now a $W$--matrix. From now on we can go on as in part (b), showing that $X$ is a toric cover of a suitable WPTB.

For the converse, let us assume that $X(\Si_{\g})$ is a toric cover of a WPTB  $\P^W(\mathcal{E})$. This means that there exists a $\Z$--linear morphism $\overline{f}:N\to \widetilde{N}$ such that $\widetilde{\Si}=\overline{f}_{\R}(\Si_{\g})$ is the fan of $\widetilde{X}=\P^W(\mathcal{E})$.
Therefore $\widetilde{\Si}$ is composed by fibred cones (\ref{cono fibrato}) implying that it is split by a $n'$--dimensional fan $\widetilde{\Si}'$ and a $s$--dimensional fan $\widetilde{\Si}_W$. The equality of fans $\widetilde{\Si}=\overline{f}_{\R}(\Si_{\g})$ then imposes an analogous splitting for the fan $\Si_\g$. Gale duality then gives that every cone in the bunch $\mathcal{B}(\g)$ is the sum of a 1--dimensional cone and a $(r-1)$--dimensional cone belonging to a fixed facet of the Gale dual cone ${\gkz}$. Prop.~\ref{prop:maxbord} then enables us to conclude that $\g$ is a maxbord chamber.
\end{proof}

The geometric picture described by the previous Theorem \ref{thm:maxbord} dramatically simplifies in the case of smooth projective toric varieties: in this context, the following result is equivalent to \cite[Prop.~4.1]{Batyrev91}.

\begin{corollary}\label{cor:smooth&maxbord} Given a reduced $n\times(n+r)$ $CF$--matrix $V$, with $r\geq 2$, a chamber $\g\in\mathcal{A}_{\Gamma}(V)$ is maximally bordering and non-singular if and only if the associated PWS $X(\Si_{\g})$ is a projective toric bundle $\P(\mathcal{E})\to X'$ over a smooth PWS $X'(\Si')$.
\end{corollary}

\begin{proof} Since maxbord implies intbord, Theorem \ref{thm:intbord-pc} and Corollary \ref{cor:intbord+reg} give a numerically effective primitive collection for $\Si_{\g}$ whose primitive relation has all the non--zero coefficient equal to 1. By Lemma~\ref{lm:dibase}, such a primitive relation can be considered as the bottom row of the $\REF$ positive weight matrix $Q$. Hence, Theorem~\ref{thm:maxbord} gives that $X$ is a toric cover of a projective toric bundle  $\P(\mathcal{E})$, since $W=(1,\ldots,1)$. In particular the covering map $f:X\to\P(\mathcal{E})$ is the identity if and only if the $r'\times (n'+r')$ matrix $Q'$, obtained as in (\ref{Qfibrata}), is a reduced $W$-matrix. The fact that $\g$ is a non-singular chamber implies that $Q'$ is necessarily a reduced $W$--matrix: in other words cases (b) and (c) in the proof of Theorem~\ref{thm:maxbord} cannot occur. In fact:
\begin{itemize}
  \item for (c) notice that Prop.~\ref{prop:maxbord} gives that every cone $\langle Q_I\rangle\in\mathcal{B}(\g)$ admits a unique generator $\q_i$, with $n'+r'+1\leq i\leq n+r$, such that $\langle Q_I\rangle=\langle\q_i\rangle+\langle Q'_{I\backslash\{i\}}\rangle$; by the REF of $Q$, $|\det(Q_I)|=q_{r,i}|\det(Q'_{I\backslash\{i\}})|$; since $\mathcal{L}_r(Q')$ has cotorsion in $\Z^{n'+r'}$, every $r$--minor of $Q'$ can't be unimodular, giving $|\det(Q_I)|>1$, against the non-singularity of $\g$;
  \item for (b) notice that if $Q'$ is a non--reduced $W$--matrix, then there exists a column index $h$ such that $1\leq h\leq n'+r'$ and $\mathcal{L}_r(Q'^{\{h\}})$ has cotorsion in $\Z^{n'+r'-1}$; in the bunch $\mathcal{B}(\g)$ there certainly exists a cone $\s^{\{h\}}$ not admitting the column $\q_h$ as a generator; Prop.~\ref{prop:maxbord} then gives that $\s^{\{h\}}$ admits a unique generator $\q_i$ with $n'+r'+1\leq i\leq n+r$; as above, the $\REF$ of $Q$ then gives $|\det(\s^{\{h\}})|>1$, against the non-singularity of $\g$.
\end{itemize}
Then $X$ is a PTB $\P(\mathcal{E})\to X'$. The smoothness of $X'$ follows by the smoothness of $X$.

The converse is obvious, since a PTB $X(\Si_{\g})=(\P(\mathcal{E})\to X')$ over a smooth PWS $X'$ is clearly smooth, giving the non-singularity of $\g$. Moreover $\g$ is maxbord by Theorem \ref{thm:maxbord}.
\end{proof}

The previous results allows us to give the following characterization of a PWS which is a toric flip (in the sense of \S~\ref{ssez:toricflip}) of a toric cover of a PWS, by means of a particular condition on the weight matrix: see the following Example \ref{ex:noWPTB} for a PWS not satisfying such a condition and then not realizing this kind of a birational equivalence.

\begin{theorem}\label{thm:birWPTB} Let $V$ be a $CF$--matrix and consider $X(\Si)$, with $\Si\in\P\SF(V)$. Then $X$ is a toric flip of a toric cover $\widetilde{X}\twoheadrightarrow\P^W(\mathcal{E})$ of a WPTB if and only if $\Mov(V)$ is maxbord w.r.t. an hyperplane $H\subseteq F^r_{\R}$ i.e., up to an application of Lemma~\ref{lm:dibase} sending $H$ to $H_r=\{x_r=0\}$, there exists a positive, REF, $W$--matrix $Q=\G(V)$ looking as in (\ref{Qfibrata}) and such that
 \begin{enumerate}
   \item either the left--upper submatrix $Q'$ is a reduced $W$--matrix: in this case $X$ is a toric flip of a WPTwB;
   \item or the left--upper submatrix $Q'$ is either a non--reduced $W$--matrix or satisfies all the conditions of Definition \ref{def:Wmatrix} but condition b: in this case $X$ is a toric flip of a toric cover of a WPTwB.
 \end{enumerate}
 Moreover, if $X$ is smooth case (2) cannot occur and $X$ turns out to be a toric flip of a PTB if and only if the left--upper submatrix $Q'$ is a reduced $W$--matrix.
 \end{theorem}

\begin{proof} Let $X(\Si)$ be a toric flip of a toric cover $\widetilde{X}(\widetilde{\Si})$ of a WPTB $\P^W(\mathcal{E})$. This means that we can assume $\widetilde{\Si}\in\P\SF(V)$ and $\widetilde{\g}:=\g_{\widetilde{\Si}}$ be a maxbord chamber, w.r.t. an hyperplane $H$, of $\Ga(V)$. This is enough to show that $\Mov(V)$ is maxbord w.r.t. $H$. Let us apply Lemma~\ref{lm:dibase}: then we can assume $Q=\G(V)$ be a positive, REF, $W$--matrix looking as in (\ref{Qfibrata}) and $H=H_r$ be the supporting hyperplane of a bordering primitive collection for $\widetilde{\Si}$. Proceeding as in the proof of Theorem~\ref{thm:maxbord}, the upper left submatrix $Q'$ turns out to satisfy conditions (1) and (2) of the statement.

Conversely, let us assume $\Mov(V)$ be maxbord w.r.t. an hyperplane $H$. This means that there exists a maxbord chamber $\widetilde{\g}\subseteq\Mov(V)$ w.r.t. the hyperplane $H$. Proceeding as in the proof of Theorem~\ref{thm:maxbord} we can assume $H=H_r$ and $Q=\G(V)$ be a positive, REF, $W$--matrix looking as in (\ref{Qfibrata}). In particular the upper left submatrix $Q'$ turns out to satisfy conditions (1) and (2) in the statement. Setting $\widetilde{\Si}:=\Si_{\widetilde{\g}}\in\P\SF(V)$, Theorem~\ref{thm:maxbord} ensures that $\widetilde{X}(\widetilde{\Si})$ is either a WPTwB, when $Q'$ satisfies condition (1), or a toric cover of a WPTwB, when $Q'$ satisfies condition (2). Clearly $X$ is a toric flip of $\widetilde{X}$.

The last part of the statement, regarding the smooth case, follows by Corollary \ref{cor:smooth&maxbord}.
\end{proof}

\subsection{The geometric meaning of a maximally bordering chamber}\label{ssez:geometrico}
Recalling Proposition~\ref{prop:nef} a maxbord chamber $\g$ w.r.t. a hyperplane $H$ gives a fan $\Si=\Si_{\g}$ such that the hyperplane $H$ cuts out a common facet of $\Nef(X(\Si))$ and $\overline{\Eff}(X(\Si))$: dually we are fixing an extremal ray of the Mori cone $\overline{\NE}(X_{\Si})$. By \cite[Prop.~15.4.1, Lemma~15.4.2(b,c) and Prop.~15.4.5(a)]{CLS}, contracting such an extremal ray gives rise to a fibering morphism of $\Q$--factorial complete toric varieties $\phi:X(\Si)\to X_0(\Si_0)$ whose fibers are connected and isomorphic to a finite abelian quotient of a WPS (also called a \emph{fake WPS}) whose dimension is given by $s$, where $s+1$ is the cardinality $|\pc_H|$ of the primitive collection supported by $H$.

On the other hand if $X$ is a PWS then Theorem~\ref{thm:maxbord} exhibits $X$ as a toric cover of a WPTB $\P^W(\mathcal{E})$ so giving
\begin{equation*}
  \xymatrix{X\ar[rr]^-f_-{\text{toric cover}}&&\P^W(\mathcal{E})\ar[rr]^-{\varphi}_-{\text{WPTB}}&& X'}
\end{equation*}
Putting all together this means that the fibering morphism $\phi$ gives the morphism with connected fibers of the Stein factorization of $\varphi\circ f$, which is
\begin{equation}\label{Stein}
  \xymatrix{X\ar[d]^-{\phi}\ar[r]^-f&\P^W(\mathcal{E})\ar[d]^-{\varphi}\\
            X_0\ar[r]^-{f_0}_-{\text{finite}}&X'}
\end{equation}
Let us underline that, by Theorem~\ref{thm:maxbord}, the right hand side of diagram (\ref{Stein}) allows one to completely determine (starting from a fan $CF$--matrix $V$ and, by Gale duality, a REF positive $W$-matrix $Q=\G(V)$) the toric cover $f$, the WPS giving the fibers of $\P^W(\mathcal{E})$ and the basis $X'$, in terms of a collection of matrices giving diagram (\ref{covering-diagram}).

Moreover Corollary~\ref{cor:smooth&maxbord} shows that when $X$ is smooth both the finite morphisms $f$ and $f_0$ in the commutative diagram (\ref{Stein}) are trivial, meaning that in the smooth case $\phi=\varphi$.

Let us finally notice that \cite[Prop.\,1.11]{Hu-Keel} and considerations following Prop.\,2.5 in \cite{Casagrande13} suggest that a similar construction may probably be proposed in the more general setup of Mori Dream Spaces and their ambient toric varieties.

\subsection{Maximally bordering chambers and splitting fans}\label{ssez:maxbord&splitting}
In \cite[\S~4]{Batyrev91} V.~Batyrev relates the fibred structure of smooth complete toric varieties with some intersection properties of their primitive collections. In particular, restricting our attention to the subclass of projective varieties, the previous Corollary \ref{cor:smooth&maxbord}, compared with \cite[Prop.~4.1]{Batyrev91}, gives that

\begin{itemize}
\item[]\emph{Given a reduced $n\times(n+r)$ $CF$--matrix $V$, with $r\geq 2$, a non-singular chamber $\g\in\mathcal{A}_{\Gamma}(V)$ is maximally bordering if and only if there exists a primitive collection $\mathcal{P}$ for $\Si_{\g}$ such that:}
      \begin{itemize}
        \item[$(i)$] \emph{the corresponding primitive relation $r(\mathcal{P})$ is numerically effective,}
        \item[$(ii)$] \emph{$\mathcal{P}\cap\mathcal{P}'=\emptyset$ for any primitive collection $\mathcal{P}'$ for $\Si_{\g}$ such that $\mathcal{P}'\neq\mathcal{P}$.}
      \end{itemize}
\end{itemize}
Therefore Theorem \ref{thm:maxbord} is clearly the extension of \cite[Prop.~4.1]{Batyrev91} to the case of $\Q$--factorial projective toric varieties.

Moreover the just given characterization of maxbord chambers in terms of pri\-mi\-ti\-ve collections can be obtained by dropping the smoothness hypothesis too, i.e.

\begin{proposition}\label{prop:maxbord vs primitive} Given a reduced $n\times(n+r)$ $F$--matrix $V$, with $r\geq 2$, a chamber $\g\in\mathcal{A}_{\Gamma}(V)$ is maximally bordering if and only if there exists a primitive collection $\mathcal{P}$ for $\Si_{\g}$ such that:
      \begin{itemize}
        \item[$(i)$] the corresponding primitive relation $r_{\Z}(\mathcal{P})$ is numerically effective,
        \item[$(ii)$] $\mathcal{P}\cap\mathcal{P}'=\emptyset$ for any primitive collection $\mathcal{P}'$ for $\Si_{\g}$ such that $\mathcal{P}'\neq\mathcal{P}$.
      \end{itemize}
\end{proposition}

\begin{proof} If $\g$ is maxbord then the existence of a nef primitive collection $$\pc=\{\langle\v_{s+1}\rangle,\ldots,\langle\v_{n+r}\rangle\}$$
is guaranteed by  Theorem \ref{thm:intbord-pc}: in particular we can assume that $\g$ is maxbord w.r.t. the hyperplane $H_r$ and, saying $\pc^*=\{\langle\q_{s+1}\rangle,\ldots,\langle\q_{n+r}\rangle\}$, Prop.~\ref{prop:maxbord} guarantees that $|\s(1)\cap\pc^*|=1$, for every cone $\s\in\mathcal{B}(\g)$.  Let $\pc'$ be a further primitive collection, for $\Si_{\g}$, supported on a hyperplane $H'$ and let $\n'$ be the numerical class of $\pc'$, which is the inward primitive normal vector to $H'$. If there would exist a $\q_i\in\pc\cap\pc'$ then condition ($ii$) of Proposition \ref{prop:primitive} gives a cone $\mathcal{C}_{i,\pc'}\in\mathcal{B}(\g)$ such that $\mathcal{C}_{i,\pc'}(1)\cap\pc'^*=\{\langle\q_i\rangle\}$. On the other hand $|\mathcal{C}_{i,\pc'}(1)\cap\pc^*|=1$ implying that $\mathcal{C}_{i,\pc'}(1)\cap\pc^*=\{\langle\q_i\rangle\}$. This is enough to show that $\mathcal{C}_{i,\pc'}\cap H'=\mathcal{C}_{i,\pc'}\cap H_r$ hence giving $H'=H_r$ and therefore $\pc'=\pc$.

Conversely let $\pc=\{V_P\}$ be a primitive collection satisfying conditions ($i$) and ($ii$) and supported by a hyperplane $H_P$. Then ($i$) ensures that $\pc$ is bordering and Lemma \ref{lm:dibase} allows us to assume that $H_P=H_r$ and $P=\{i\,|\,s+1\leq i\leq n+r\}$.
\begin{claim} Let $\q_i\in\pc^*$. If $\g$ is not maxbord w.r.t. $H_r$ then there exists a hyperplane $H'\neq H_r$, cutting a facet of $\g$, whose inward normal vector $\n'$ gives
$$\n'\cdot\q_i>0\,.$$
\end{claim}
Then the collection supported by $H'$, i.e.
$$\pc'^*:=\{\q_j\,|\,\q_j\ \text{is a column of $Q$ with $\n'\cdot\q_j>0$} \}$$
turns out to be a primitive collection such that $\pc'\neq\pc$ and $\pc'\cap\pc\supseteq\{\langle\q_i\rangle\}\neq\emptyset$, giving a contradiction.

\noindent To prove the Claim let us consider all the hyperplanes $H^{(1)},\ldots,H^{(l)}$ cutting a facet of $\g$. Since $\g$ is not maxbord w.r.t. $H_r$, none of them coincides with $H_r$. Let $\n_j$ be the primitive inward normal vector to $H^{(j)}$. If $\n_j\cdot\q_i\leq 0$, for every $1\leq j\leq l$, then $-\q_i\in\g$ giving a contradiction since $-\q_i$ has negative entries. Then there should exist $\n_j$ such that $\n_j\cdot\q_i>0$.
\end{proof}

Keeping in mind the previous Proposition \ref{prop:maxbord vs primitive} one can then give the following generalization of  \cite[Thm.~4.3]{Batyrev91}.

\begin{proposition}\label{prop:bimaxbord} Given a reduced $n\times(n+r)$ $CF$--matrix $V$, with $r\geq 2$, let $\g\in\mathcal{A}_{\Gamma}(V)$ be a maximally bordering chamber w.r.t. two distinct hyperplanes $H$ and $H'$. Then $X(\Si_{\g})$ is a toric cover of a WPTB over a PWS $X'(\Si')$ which is still a toric cover of a WPTB.
\end{proposition}

\begin{proof} The proof is given by an iterated application of Theorem \ref{thm:maxbord}. To start the iteration one has to prove that the chamber $\g'=\g\cap H$, as defined in the proof of the Thm.~\ref{thm:maxbord}, part (a), possibly up to a toric cover if we are in cases (b) or (c), is still a maxbord chamber in $\Mov(V')$ w.r.t. $H\cap H'$, where $V'$ is the fan matrix of the basis of the first, possibly trivial, toric cover. This fact follows by observing that every cone $\s\in\mathcal{B}(\g)$ has the following properties:
\begin{enumerate}
  \item $|\s(1)\cap \pc^*|=1$\,,
  \item $|\s(1)\cap \pc'^*|=1$\,,
  \item $\pc^*\cap\pc'^*=\emptyset$\,,
\end{enumerate}
where $\pc$ and $\pc'$ are nef primitive collections associated with $H$ and $H'$ respectively. Then (1) and (2) follow by
the maxbord hypothesis with respect to both these hyperplanes and (3) follows immediately by Prop.~\ref{prop:maxbord vs primitive}. Therefore every cone $\s\in\mathcal{B}(\g)$ admits the decomposition
$\s=\langle\pp\rangle+\langle\pp'\rangle+\langle\q_1,\ldots,\q_{r-2}\rangle$, with $\pp\in\pc^*$, $\pp'\in\pc'^*$ and $\langle\q_1,\ldots,\q_{r-2}\rangle\subset H\cap H'$. This suffices to show that $\g'$ is maxbord w.r.t. $H\cap H'$: in fact all the cones of $\mathcal{B}(\g')$ comes from a cone in $\mathcal{B}(\g)$, since the latter is always the Gale dual of a fibred cone, in the sense of (\ref{cono fibrato}).
\end{proof}

We are now in a position of understanding, in the projective case, the concept of a \emph{splitting fan}, as given in Definition 4.2. in \cite{Batyrev91}, in terms of the geometry of the associated chamber. Let us then set the following crucial definition.

\begin{definition}\label{def:totally maxbord} Let $V$ be a reduced $n\times(n+r)$ $F$--matrix. A chamber $\g\in\Gamma(V)$ is called \emph{totally maxbord} if it is maxbord with respect to $r-1$ distinct hyperplanes. Moreover, for $1\leq l\leq r-1$, the chamber $\g$ is called \emph{$l$--recursively maxbord} if there exists a sequence of $l$ distinct hyperplanes $H^{(1)},\ldots,H^{(l)}$ such that $\g$ is maxbord w.r.t. $H^{(1)}$ and, for every $0\leq i\leq l-1$, $\g^{(i)}:=\g\cap\bigcap_{j\leq i}H^{(j)}$ is maxbord w.r.t. $H^{(i+1)}\cap\bigcap_{j\leq i}H^{(j)}$, possibly up to a finite sequence of toric covers. When $l=r-1$ we simply say that $\g$ is \emph{recursively maxbord}.

\noindent Notice that $1$--recursively maxbord means simply maxbord. In particular
\begin{itemize}
  \item if $r=2$ then maxbord\, $\Leftrightarrow$\, recursively maxbord\, $\Leftrightarrow$\, totally maxbord.
\end{itemize}

\end{definition}

By an easy induction, the previous Proposition \ref{prop:bimaxbord} shows that \emph{a totally maxbord chamber is a recursively maxbord chamber}. Let us underline that the converse is false, as the following Example \ref{ex:nototmaxbord} shows.

We can then state the following generalization of Cor.~4.4 in \cite{Batyrev91}.

\begin{theorem}\label{thm:recmaxbord} A PWS $X(\Si)$ is produced from a toric cover of a WPS by a sequence of toric covers of WPTB's if and only if the corresponding chamber $\g_{\Si}$ is recursively maxbord.
\end{theorem}

The proof is an easy iteration of Theorem \ref{thm:maxbord}.

Recalling that a Batyrev's splitting fan is a non--singular fan whose primitive collections are all disjoint pair by pair, by \cite[Cor.~4.4]{Batyrev91} and the previous Theorem \ref{thm:recmaxbord} a splitting fan turns out to be completely equivalent to a fan associated with a non--singular recursively maxbord chamber. Analogously to what was done in Proposition \ref{prop:maxbord vs primitive} we can try to drop the smoothness hypothesis obtaining the following:

\begin{proposition}\label{prop:rec-maxbord vs splitting} Given a reduced $n\times(n+r)$ $F$--matrix $V$, if $\g\in\Gamma(V)$ is a $(r-2)$--recursively maxbord chamber then any two different primitive collection for $\Si_{\g}$ have no common elements.
\end{proposition}

\begin{remark}\label{rem:r-2} In the statement of Proposition~\ref{prop:rec-maxbord vs splitting} $\g$ is supposed to be a $(r-2)$--recursively maxbord chamber and not necessarily a $(r-1)$--recursively maxbord one: this fact may result surprising since, after the previous Theorem~\ref{thm:recmaxbord}, a splitting fan is equivalent to a fan associated with a non--singular recursively maxbord chamber. The following Example~\ref{ex:noconverse} clarifies the situation, describing a case which cannot occur in the smooth case. Actually: \emph{a non--singular $(r-2)$--recursively maxbord chamber turns out to be necessarily a $(r-1)$--recursively maxbord one. }

\noindent This fact is a consequence of \cite[Thm.~4.3]{Batyrev91} and Thm.~\ref{thm:recmaxbord}. Notice that the starting step of the induction proving \cite[Thm.~4.3]{Batyrev91} does not more hold in the singular case even for projective varieties, which is: there exist projective $\Q$--factorial toric va\-rie\-ties not admitting any nef primitive collection although their primitive collections are disjoint pair by pair, as Example~\ref{ex:noconverse} shows.
\end{remark}

\begin{example}\label{ex:noconverse} Consider the 2--dimensional PWS of rank 2 whose weight and fan matrices are, respectively, given by
\begin{equation*}
    Q=\left(
        \begin{array}{cccc}
          1 & 2 & 1 & 0 \\
          0 & 1 & 1 & 1 \\
        \end{array}
      \right)\quad\Rightarrow\quad V=\G(Q)=\left(
                                             \begin{array}{cccc}
                                               1 & 0 & -1 & 1 \\
                                               0 & 1 & -2 & 1 \\
                                             \end{array}
                                           \right)
\end{equation*}
Then $\Mov(V)=\langle\q_2,\q_3\rangle=\left\langle\begin{array}{cc}
                            2 & 1 \\
                            1 & 1
                          \end{array}
\right\rangle \subset \gkz=F^2_+$, and there is a unique chamber $\g=\Mov(V)$, giving a unique fan $\Si_{\g}$. This fan admits only two disjoint primitive collections given by $\pc_1=\{\v_1,\v_2\}$ and $\pc_2=\{\v_3,\v_4\}$ whose primitive relations are given, respectively, by
\begin{equation*}
    \v_1+\v_2=\v_4\quad,\quad\v_3+2\v_4=\v_1\,.
\end{equation*}
Hence $\g$ does not admit any nef primitive collection.

\noindent In particular notice that $\pc_1\cap\pc_2=\emptyset$ but $\g$ is not even a bordering chamber.
\end{example}

\begin{proof}[Proof of Proposition \ref{prop:rec-maxbord vs splitting}] Let us first of all observe that if $r=2$ then any chamber always admits only two distinct and disjoint primitive collections.

\noindent Assume now $r\geq 3$ and $\g$ be a $(r-2)$--recursively maxbord chamber. Then there exists a hyperplane $H$ such that $\g$ is maxbord w.r.t. $H$ and, possibly up to a toric cover, $\g':=\g\cap H\in\mathcal{A}_{\Gamma}(V')$, where $V'$ is a reduced fan matrix of the base of the mentioned toric cover. Let us assume, for ease, that such a toric cover is trivial, hence $V'=\G(Q')$ where $Q'$ is the left--upper $(r-1)\times(n+r-s-1)$ submatrix of $Q=\G(V)$ and it is a $W$--matrix, as in part (a) of the proof of Thm.~\ref{thm:maxbord}: such an assumption does not cause any loss of generality since, after a toric cover, the general case is reduced precisely to this situation, as in cases (b) and (c) of the proof of Thm.~\ref{thm:maxbord}. Let $\pc$ be the nef primitive collection associated with $H$. As a first step we want to show that:
\begin{itemize}
\item[$(i)$] \emph{a primitive collection $\pc^{(1)}\neq\pc$ for the fan $\Si$ is still a primitive collection for $\Si'=\Si_{\g'}'$.}                               \end{itemize}
Let us first of all observe that, by the maxbord hypothesis and Prop.~\ref{prop:maxbord vs primitive}, $\pc^{(1)}\cap\pc=\emptyset$, meaning that $\pc^{(1)*}\subset H$. To prove $(i)$ notice that all the rays contained in $\pc^{(1)}$ cannot be contained in a unique cone of the fan $\Si'$ which is that, dually, there cannot exist a cone $\s'\in\mathcal{B}(\g')$ such that $\s'(1)\cap\pc^{(1)*}=\emptyset$. In fact, by the maxbord hypothesis w.r.t. $H$, there exists $\pp\not\in H$ such that
$\s=\langle\pp\rangle+\s'\in\mathcal{B}(\g)$: if $\pc^{(1)*}\cap\s'(1)=\emptyset$ then $\pc^{(1)*}\cap\s(1)=\emptyset$, since $\pc^{(1)*}\subset H$ and $\pp\not\in H$; this gives a contradiction with the assumption that $\pc^{(1)}$ is a primitive collection for $\Si$.

\noindent On the other hand if $\pp^{(1)}\in\pc^{(1)*}$ then there exists a cone $\s\in\mathcal{B}(\g)$ such that $\s(1)\cap\pc^{(1)*}=\{\pp^{(1)}\}$. Again the maxbord hypothesis for $\g$ w.r.t. $H$ gives the existence of $\pp\not\in H$ and $\s'\in\mathcal{B}(\g')$ such that
$\s=\langle\pp\rangle+\s'$. Consequently
$$\s'(1)\cap\pc^{(1)*}=\s(1)\cap\pc^{(1)*}=\{\pp^{(1)}\}\,,$$
since $\pc^{(1)}\subset H$ while $\pp\not\in H$. This suffices to prove $(i)$.

As a second step, observe now that if $\pc^{(1)}\neq\pc^{(2)}$ are two distinct primitive collections for $\Si$ with $\pc^{(1)}\cap\pc^{(2)}\neq\emptyset$ then they give two distinct primitive collections for $\Si'$ admitting common elements.
End up now by induction.
\end{proof}

Let us now focus on the number of primitive collections. The following result gives a relation between the mi\-ni\-mal number of primitive collections and the mi\-ni\-mal number of facets of a chamber $\g$.

\begin{proposition}\label{prop:split-vs-simplicial} Let $V$ be a reduced $n\times(n+r)$ $F$--matrix and $\g\in\Gamma(V)$ a $(r-2)$--recursively maxbord chamber. Then $\g$ is a simplicial cone if and only if the associated fan $\Si_{\g}$ admits precisely $r$ primitive collections. In particular if $\g$ is simplicial and
\begin{enumerate}
  \item either $r\geq 3$
  \item or $\g$ is recursively maxbord
\end{enumerate}
 then at least one primitive collection is numerically effective.
\end{proposition}
\begin{proof}
Let us start by proving the only if condition: in fact after Thm. 1.4 in \cite{Cox-vRenesse} one knows that every facet of the chamber  $\g$ generates a primitive collection $\pc$: e.g. by thinking $\pc^*$ as all the columns $\q$ of $Q=\G(V)$ such that $\n\cdot\q> 0$, where $\n$ is the inward primitive normal vector $\n$ to the considered facet. Then $\g$ admits the minimal number $r$ of facets, meaning that it is necessarily simplicial.

For the converse, let us first of all notice that, when $r=2$, every chamber $\g$ is simplicial and admits precisely 2 primitive collections. For the second part of the statement notice that, in this case, one of these two primitive collection is nef if and only if $\g$ is bordering, hence maxbord.

Let us now assume $r\geq 3$ and $\g$ be a simplicial $(r-2)$--recursively maxbord chamber. Let
$H^{(1)}$ be a hyperplane w.r.t. $\g$ is maxbord: then $H^{(1)}$ cuts out a facet of $\g$ and, possibly up to a toric cover, $\g':=\g\cap H^{(1)}\in\mathcal{A}_{\Gamma}(V')$, where $V'$ is a reduced fan matrix of the base of this toric cover. Let $H^{(2)},\ldots,H^{(r)}$ be the further $r-1$ hyperplanes cutting out the remaining facets of $\g$. For every $1\leq i\leq r$, $H^{(i)}$ is the support of a primitive collection $\pc^{(i)}$ defined by setting $\pc^{(i)*}=\{\q\in\gkz(1)\,|\,\n_i\cdot\q>0 \}$, where $\n_i$ is the primitive inward normal vector to $H^{(i)}$. By Prop.~\ref{prop:rec-maxbord vs splitting}, $i\neq j\Rightarrow \pc^{(i)}\cap\pc^{(j)}=\emptyset$ and the first step ($i$) in the proof of this Proposition ensures that $\pc^{(2)},\ldots,\pc^{(r)}$ are $r-1$ distinct primitive collections for $\Si'=\Si'_{\g'}$. Assume now by induction that $\Si'$ admits precisely $r-1$ primitive collections, meaning that $\pc^{(2)},\ldots,\pc^{(r)}$ give all the primitive collections for $\Si'$. By the same argument, if $\pc\neq\pc^{(1)}$ is a primitive collection of $\Si$ then it is a primitive collection of $\Si'$, which is, by induction, that $\pc=\pc^{(i)}$ for some $2\leq i\leq r$. This suffices to show that $\Si$ admits precisely $r$ primitive collections given by the facets of $\g$.
Since $\g$ is maxbord w.r.t. $H^{(1)}$ then $\pc^{(1)}$ is nef.
\end{proof}

Let us now assume that $r\leq 3$. The previous Proposition \ref{prop:split-vs-simplicial} allows us to prove the following result extending to the singular case an analogous result proven by V.~Batyrev in sections 5 and 6 of \cite{Batyrev91} under smoothness hypothesis (see the following Remark \ref{rem:batyrev}).

\begin{theorem}\label{thm:3pc} Let $V$ be a reduced $n\times(n+r)$ $F$--matrix with $r\leq 3$. If $\g\in\Gamma(V)$ is a maxbord chamber then $\g$ is simplicial and the associated fan $\Si_{\g}$ admits precisely $r$ primitive collections and at least one of them is numerically effective.
\end{theorem}

\begin{proof} Let $\g$ be a maxbord chamber w.r.t. the hyperplane $H$. Notice that
\begin{itemize}
  \item \emph{if $r\leq 3$ then a maxbord chamber is a simplicial cone.}
\end{itemize}
In fact, if $r\leq 2$ there is nothing to prove since every chamber is simplicial. Let us assume $r=3$. Therefore, every cone $\s\in\mathcal{B}(\g)$ can be written as follows
\begin{equation}\label{decomposizione 3cono}
    \s=\langle\pp\rangle + \s'\,,\quad\text{with}\ \pp\not\in H\,,\ \s'\subset H
\end{equation}
Since $\g=\bigcap_{\s\in\mathcal{B}(\g)}\s$, (\ref{decomposizione 3cono}) shows that $\g$ admits a unique ray outside of the hyperplane $H$, generated by some $\pp\not\in H$. Hence $\g=\pp+\s'$ for some $\s'\subset H$ which is necessarily simplicial since $\dim(\s')=2$. Then Proposition \ref{prop:split-vs-simplicial} concludes the proof. In particular the primitive collection associated with the hyperplane $H$ is nef.
\end{proof}

\begin{remark}\label{rem:batyrev}
Recalling Remark \ref{rem:r-2} the previous Theorem \ref{thm:3pc} generalizes a result already proved by Batyrev under the further hypothesis that $\g$ is non--singular, meaning that $\Si_{\g}$ is a splitting fan in the sense of Def.~4.2 in \cite{Batyrev91} (see Propositions~5.2--5, Theorem 5.7 and Theorem 6.6 in \cite{Batyrev91}). Actually Batyrev proved also the converse result, which is: a non--singular fan admitting precisely 3 primitive collections is necessarily a splitting fan, which is, by Corollary 4.4 in \cite{Batyrev91} and Theorem \ref{thm:recmaxbord}, that $\g$ is a non--singular recursively maxbord chamber.
Notice that this latter fact is false in the singular case, as the following Example \ref{ex:WPTB(c)} shows.
\end{remark}

\subsection{Maximally bordering chambers and contractible primitive relations}\label{ssez:contraibile}

Let us conclude the present subsection by giving a partial generalization of results by H.~Sato and C.~Casagrande. Let us first of all recall the following definition.

\begin{definition}[\cite{Casagrande}, Def.\ 2.3] A class $\kappa\in A_1(X)\cap\overline{\NE}(X)$ is called \emph{contractible} if $\kappa$ is a generator of the semigroup $A_1(X)\cap\Q_{\geq 0}\kappa$ and there exist
\begin{itemize}
  \item some irreducible curve having numerical class in $\Q_{\geq 0}\kappa$,
  \item a toric variety $X_{\kappa}$,
  \item an equivariant morphism $\varphi_{\kappa}:X\longrightarrow X_{\kappa}$ with connected fibers,
\end{itemize}
such that for every irreducible curve $C\subset X$
\begin{equation*}
    \varphi_{\kappa}(C)=\{pt\}\ \Longleftrightarrow\ [C]\in \Q_{\geq 0}\kappa\,.
\end{equation*}
\end{definition}

Let us now assume that $X$ is smooth: then $r_{\Z}(\mathcal{P})=r(\mathcal{P})$. Corollary 2.4 and Prop.~3.4 in \cite{Casagrande} jointly with Prop.~\ref{prop:maxbord vs primitive} above, allows one to conclude that

\begin{proposition}\label{prop:contraibile} Let $X(\Si)$ be a smooth projective toric variety. Then the following facts are equivalent:
\begin{enumerate}
  \item there exists a numerically effective primitive relation $\kappa=r(\mathcal{P})$ for $\Si$ which is contractible,
  \item there exists a nef primitive collection $\mathcal{P}$ for $\Si$ such that for every primitive collection $\pc'\neq\pc$, for $\Si$, then $\pc'\cap\pc=\emptyset$,
  \item $\g_{\Si}$ is a maxbord chamber.
\end{enumerate}
 In particular $X_{\kappa}$ is smooth of dimension $n-s$ and rank $r-1$ and $\varphi_{\kappa}$ is a toric $\P^{s}$--bundle.
\end{proposition}

\begin{proof}
$(1)\Leftrightarrow (2)$. This is precisely \cite[Prop.~3.4]{Casagrande}. In particular the contractible primitive relation $\kappa$ in part (1) is the primitive relation $\kappa=r(\pc)$ of a primitive collection $\pc$ like in part (2) and viceversa.

$(2)\Leftrightarrow (3)$. This is Prop.~\ref{prop:maxbord vs primitive}. In particular $\g_{\Si}$ turns out to be maxbord w.r.t. the the hyperplane $H_{\pc}$ supporting the primitive collection $\pc$ as in part (2) and viceversa.

Finally \cite[Cor.~2.4]{Casagrande} gives the last part of the statement.
\end{proof}

Theorem \ref{thm:maxbord} and Proposition \ref{prop:maxbord vs primitive} allows us to extend the previous result to a PWS, although the situation turns out to be more intricate. Let us first of all notice that if $f:X(\Si)\to\widetilde{X}\left(\widetilde{\Si}\right)$ is a toric cover then Lemma \ref{lm:cover&bord} guarantees that $\g=\g_{\Si}$ is a maxbord chamber if and only if $\widetilde{\g}=\widetilde{\g}_{\widetilde{\Si}}$ is a maxbord chamber. Let $H$ and $\widetilde{H}$ be bordering hyperplanes of $\g$ and $\widetilde{\g}$, respectively, and $\pc$ and $\widetilde{\pc}$ be the collections supported by $H$ and $\widetilde{H}$, respectively, and such that $\widetilde{\pc}=\{f_{\R}(\v)\,|\,\v\in\pc\}$: in this case we will write $\widetilde{\pc}=f(\pc)$. Prop.~\ref{prop:maxbord vs primitive} implies that $\pc$ and $\widetilde{\pc}=f(\pc)$ turns out to be nef primitive collection for $\Si$ and $\widetilde{\Si}$, respectively, both satisfying property (2) in Prop.~\ref{prop:contraibile}. Consider the associated numerically effective primitive relations $\kappa=r_{\Z}(\pc)$ and $\widetilde{\kappa}=r_{\Z}\left(\widetilde{\pc}\right)$: also in this case we will write $\widetilde{\kappa}=f(\kappa)$.

\begin{definition}\label{def:pseudocontraibile} Given a  toric cover $f:X(\Si)\to\widetilde{X}\left(\widetilde{\Si}\right)$, a numerically effective primitive relation $\kappa=r_{\Z}(\pc)\in A_1(X)\cap\overline{\NE}(X)$ for $\Si$ is called \emph{pseudo--contractible} if $\widetilde{\kappa}=f(\kappa)$ is a contractible class in $A_1\left(\widetilde{X}\right)\cap\overline{\NE}\left(\widetilde{X}\right)$. Then there exist a toric variety $X_{\widetilde{\kappa}}$ and the following commutative diagram of equivariant morphisms
\begin{equation}\label{pseudocontraibile}
    \xymatrix{X\ar[r]^-f\ar[rd]_{\varphi_{\k}}&\widetilde{X}\ar[d]^-{\varphi_{\widetilde{\k}}}\\
                                           & X_{\widetilde{\k}}}
\end{equation}
such that $\varphi_{\widetilde{\k}}$ has connected fibers and for every irreducible curve $C\subset X$
\begin{equation*}
    \varphi_{\k}(C)=\{pt\}\ \Longleftrightarrow\ [f(C)]\in \Q_{\geq 0}\widetilde{\kappa}\,.
\end{equation*}
\end{definition}

\begin{theorem}\label{thm:contraibile} Let $V$ be a reduced $n\times(n+r)$ $CF$--matrix  and $\Si\in\P\SF(V)$. Assume that there exists a primitive collection $\pc$ for $\Si$ whose primitive relation $\kappa=r_{\Z}(\pc)$ is numerically effective. By applying Lemma \ref{lm:dibase}, assume that $H_r=\{x_r=0\}\subseteq F^r_{\R}$ is the supporting hyperplane of $\pc$, meaning that there exists a $W$--matrix $Q=\G(V)$ in positive $\REF$ and looking as in (\ref{Qfibrata}), which is
 $$Q=\overbrace{\left(\begin{array}{c}
                 Q' \\
                 0\cdots0\\
               \end{array}\right.}^{n'+r'}\overbrace{\left.\begin{array}{c}
               Q'' \\
                 w_0\cdots w_s\\
               \end{array}\right)}^{s+1}$$
hence giving $r_{\Z}(\pc)=(w_0,\ldots,w_s)$. Then, setting $V'=\G(Q')$:
\begin{enumerate}
  \item $\kappa$ is contractible if and only if  the chamber $\g_{\Si}$ is maxbord w.r.t. the hyperplane $H_r$ and
  \begin{itemize}
    \item[$(1.i)$] $Q'$ is a $(r-1)\times(n+r-s-1)$ reduced $W$--matrix,
    \item[$(1.ii)$] the columns of $Q''$ are classes of $s+1$ Cartier divisors of $X'(\Si')$ where $\Si'\in\P\SF(V')$ is the fan associated with the chamber $\g'=\g\cap H_r$,
  \end{itemize}
    In particular the contraction $\varphi_{\kappa}:X(\Si)\longrightarrow X_{\kappa}=X'(\Si')$ exhibits $X$ as WPTB whose fibers are isomorphic to the $s$--dimensional WPS $\P(w_0,\ldots,w_s)$.
  \item $\kappa$ is pseudo--contractible if and only if $\g_{\Si}$ is maxbord w.r.t. the hyperplane $H_r$ and either $(1.i)$ holds and
  \begin{itemize}
    \item[$(2.ii)$] there exists a column of $Q''$ giving the class of a Weil non--Cartier divisor of $X'(\Si')$ where $\Si'\in\P\SF(V')$ is the fan associated with the chamber $\g'=\g\cap H_r$,
  \end{itemize}
  or
  \begin{itemize}
    \item[$(2.i)$] either $Q'$ is a $(r-1)\times(n+r-s-1)$ non--reduced $W$--matrix or $Q'$ satisfies all the conditions of Definition \ref{def:Wmatrix} but condition b.
  \end{itemize}
  In particular, in the former case $\varphi_{\k}:X(\Si)\to X'(\Si')$ exhibits $X$ as a WPTwB and in any case $X$ is a toric cover of a WPTB.
\end{enumerate}
In any case, either \emph{(1)} or \emph{(2)} occurs if and only if one of the following equivalent conditions happens:
\begin{itemize}
  \item[(I)] $\g_{\Si}$ is a maxbord chamber,
  \item[(II)] for every primitive collection $\pc'\neq\pc$, for $\Si$, then $\pc'\cap\pc=\emptyset$.
\end{itemize}
\end{theorem}

\begin{proof}
 Case (1) is an application of Thm.~\ref{thm:maxbord} in the easiest situation in which $Q$ is a $W$--matrix of a WPTB. Then techniques proving \cite[Prop.~3.4]{Casagrande} and \cite[Thm.1.10]{Sato} here applies as in the smooth case. Proving case (2) reduces to case (1) after we consider the toric covers described in \ref{sssez:WPTwB}, to settle the case $(2.ii)$, and in parts (b) and (c) of the proof of Thm.~\ref{thm:maxbord}, to settle the remaining case $(2.i)$. Finally equivalence of conditions (I) and (II) follows immediately by Prop.\ref{prop:maxbord vs primitive}.
\end{proof}

\begin{remark}\label{rem:geometrico} Let us here recall and apply what has been observed in \S~\ref{ssez:geometrico}. The fact that the chamber $\g_{\Si}$ is maxbord w.r.t. the hyperplane $H$ actually implies that the primitive relation $\k=r_{\Z}(\pc_H)$, supported by $H$, is a nef contractible class of $A_1(X)\cap\overline{\NE}(X)$, whose contraction gives a fibering morphism $\phi:X(\Si)\to X_0(\Si_0)$. Hence:
\begin{itemize}
  \item[] \emph{a pseudo--contractible class is actually a contractible class.}
\end{itemize}
More precisely the fibering morphism $\phi$ turns out to be the morphism with connected fibers of the Stein factorization of the composition $\varphi_{\k}=\varphi_{\widetilde{\k}}\circ f$ in the commutative diagram (\ref{pseudocontraibile}), hence giving the following commutative diagram
\begin{equation}\label{Stein-contraibile}
    \xymatrix{X\ar[d]^-{\phi}\ar[rr]^-f_-{\text{toric cover}}\ar[rrd]^{\varphi_{\k}}&&\widetilde{X}\ar[d]^-{\varphi_{\widetilde{\k}}}\\
                            X_0\ar[rr]^-{f_0}_-{\text{finite morphism}}&& X_{\widetilde{\k}}}
\end{equation}
The geometric description of the Stein factorization $f_0\circ\phi$ of $\varphi_{\k}$ on the left hand side of diagram (\ref{Stein-contraibile}) is well known and has its roots in Reid's paper \cite{Reid83} (see also \cite[Lemma~15.4.2 and Prop.~15.4.5]{CLS} and references therein): nevertheless it simply says that the fibers of $\phi$ are given by a fake WPS.

By Theorem~\ref{thm:contraibile}, the factorization $\varphi_{\widetilde{\k}}\circ f$ of $\varphi_{\k}$ on right hand side of diagram (\ref{Stein-contraibile}) is completely described, starting from a fan matrix $V$ of $X$, in terms of all the matrices representing morphisms giving diagram (\ref{covering-diagram}).
\end{remark}

\subsection{Examples}\label{ssez:esempi}
In this section we give some applications of techniques just illustrated. Let us start with the case of smooth projective toric varieties presented along with the introduction of definitions and constructions given above, to make a comparison between Batyrev's techniques described in \cite{Batyrev91} and techniques here presented.

\begin{example}[Examples \ref{ex:1}, \ref{ex:2} and \ref{ex:3} continued]\label{ex:PTB} Consider the smooth and projective toric variety $X(\Si)$, of dimension and rank equal to 3, with $\Si\in\P\SF(V)=\SF(V)=\{\Si_1,\Si_2\}$ defined in the Example \ref{ex:1}. Recall Fig.~\ref{fig2}, to visualizing visualize the Gale dual cone $\gkz=\langle Q\rangle =F^3_+$ and $\Mov(V)\subseteq\gkz$ and the two chambers $\g_1,\g_2$, explicitly presented in Example \ref{ex:2}.

\noindent In Example \ref{ex:3} we observed that both the chambers are intbord w.r.t. both the hyperplanes $H_2$ and $H_3$ and moreover $\g_1$ is maxbord w.r.t. these hyperplanes: hence it is totally maxbord. Hyperplanes $H_2$ and $H_3$ are supporting collections, $\pc_2=\{\v_3,\v_4\}$ and $\pc_3=\{\v_5,\v_6\}$, respectively, which are primitive and nef for both the fans $\Si_1$ and $\Si_2$.  Batyrev's results \cite[Prop.~4.1, Thm.~4.3, Cor.~4.4]{Batyrev91} or equivalently Corollary \ref{cor:smooth&maxbord} and Theorem \ref{thm:recmaxbord} above, ensure that $X(\Si_1)$ is a PTB over a smooth toric surface of rank 2.

\noindent The weight matrix $Q$ makes this fact quite explicit: the last row of $Q$ gives the class $\k:=r(\pc_3)\in A_1(X(\Si_i))\cap\overline{\NE}(X(\Si_i))$ which is a numerically effective primitive relation for both the fans $\Si_1,\Si_2$. The class $\k$ turns out to be contractible when hypotheses of \cite[Prop.~3.4]{Casagrande} are verified: more easily $\k$ turns out to be contractible by condition (3) in Prop.~\ref{prop:contraibile} above. The contraction morphism $\phi_{\k}$ is now explicitly described by the weight matrix $Q$, following Corollary \ref{cor:smooth&maxbord}: namely
$$\varphi_{\k}:X(\Si_1)\longrightarrow Y\left(\Si'\right)$$
exhibits $X(\Si_1)$ as a PTB, whose fibers are isomorphic to $\P^1$, over the toric surface $Y$, whose weight matrix $Q'$ is obtained from $Q$ by deleting the third row and the 5-th and 6-th columns, and whose fan $\Si'$ is the unique simplicial fan in $\SF\left(\G\left(Q'\right)\right)$,  which is
\begin{equation*}
    Q'=\left(
                    \begin{array}{cccc}
                      1 & 1 & 1 & 0 \\
                      0 & 0 & 1 & 1 \\
                    \end{array}
                  \right)\ ,\  \g'=\g_{\Si'}=\left\langle\begin{array}{cc}
                                                                      1 & 1 \\
                                                                      0 & 1
                                                                    \end{array}
                  \right\rangle\quad \Longrightarrow\quad Y\cong\P\left(\mathcal{O}_{\P^1}\oplus\mathcal{O}_{\P^1}(1)\right)\,.
\end{equation*}
By subtracting the third row and the second rows from the first one in $Q$, and recalling the role of matrix $Q''$ in (\ref{Qfibrata}), one gets the following weight matrix of $X\left(\Si_1\right)$
\begin{equation*}
Q\sim \widetilde{Q}=\left(
  \begin{array}{cccccc}
    1 & 1 & 0 & -1 & 0 & -1 \\
    0 & 0 & 1 & 1 & 0 & 0 \\
    0 & 0 & 0 & 0 & 1 & 1 \\
  \end{array}
\right)\quad\Longrightarrow\quad X(\Si_1)\cong\P\left(\mathcal{O}_Y\oplus\mathcal{O}_Y(h)\right)
\end{equation*}
where $Y=\P(\mathcal{O}_{\P^1}\oplus\mathcal{O}_{\P^1}(1))$ and $h$ is the generator of $\Pic(Y)$ given by the pull--back of the Picard generator $\mathcal{O}_{\P^1}(1)$ of the base $\P^1$. Then we get a recursive PTB structure given by
$$\xymatrix{X(\Si_1)\cong\P\left(\mathcal{O}_Y\oplus\mathcal{O}_Y(h)\right)\ar@{>>}[r]& Y\cong \P\left(\mathcal{O}_{\P^1}\oplus\mathcal{O}_{\P^1}(1)\right)\ar@{>>}[r] &\P^1}\,.$$
Such a fibration of $X(\Si_1)$ is not unique since the contraction of the other numerically effective primitive relation $\k':=r(\pc_2)$ gives precisely the same description of $X(\Si_1)$. This fact can be immediately deduced from the weight matrix $Q$, whose se\-cond row gives the class $\k'$: in fact by exchanging the second and the third row and reordering the columns to still get a REF matrix, one still obtains the $W$--matrix $Q$.

$X(\Si_2)$ is now obtained by $\P\left(\mathcal{O}_Y\oplus\mathcal{O}_Y(h)\right)$ after the elementary flip determined by crossing the internal wall $\langle\q_3,\q_5\rangle$ of $\Mov(V)$. Therefore the indetermination loci of the elementary flip $\xymatrix{X(\Si_1)\ar@{<-->}[r]& X(\Si_2)}$ are given by invariant 1-cycles $C_1:=O(\langle\v_4,\v_6\rangle)\subset X(\Si_1)$ and $C_2:=O(\langle\v_1,\v_2\rangle)\subset X(\Si_2)$\,. In fact in this way the primitive collections $\pc=\{\v_1,\v_2\}$ and $\pc'=\{\v_4,\v_6\}$, supported by the hyperplane cutting the internal wall, and their foci are exchanged each other. This means that $\pc$ is a primitive collection for $\Si_1$ and $\pc'$ is a primitive collection for $\Si_2$. Notice that $\pc\cap\mathcal{P}_3=\emptyset$, but $\pc'\cap\mathcal{P}_3=\{\v_6\}\neq\emptyset$ and $\pc'\setminus\pc_3=\{\v_4\}$, meaning that $X(\Si_2)$ is not a toric $\P^1$--bundle over $Y$ and $\k$ is not a contractible class for $X(\Si_2)$ (by Prop.~\ref{prop:contraibile} and \cite[Prop.~3.4]{Casagrande}, respectively).

Let us finally notice that $\Si_1$ is a splitting fan, by Prop.\ref{prop:rec-maxbord vs splitting}. This is clearly not the case for the fan $\Si_2$: moreover $\Si_1$ turns out to admit 3 primitive collections and $\Si_2$ to admit 5 primitive collections, according with \cite[\S~5]{Batyrev91}.
\end{example}

The following is still an example of smooth projective toric varieties coming from chambers which are recursively maxbord but not maxbord.

\begin{example}\label{ex:nototmaxbord} By adding the further column $\left(
                                                \begin{array}{c}
                                                  1 \\
                                                  1 \\
                                                  1 \\
                                                \end{array}
                                              \right)
$ in the weight matrix of the previous Example \ref{ex:PTB}, one gets the following reduced weight and fan matrices
\begin{equation*}
Q=\left(
                             \begin{array}{ccccccc}
                               1 & 1 & 1 & 0 & 1 & 1 & 0 \\
                               0 & 0 & 1 & 1 & 1 & 0 & 0 \\
                               0 & 0 & 0 & 0 & 1 & 1 & 1 \\
                             \end{array}
                           \right)\ \Rightarrow\ V=\left(
        \begin{array}{ccccccc}
          1 & 0 & 0 & 0 & 0 & -1 & 1 \\
          0 & 1 & 0 & 0 & 0 & -1 & 1 \\
          0 & 0 & 1 & 0 & -1 & 0 & 1 \\
          0 & 0 & 0 & 1 & -1 & 1 & 0
        \end{array}
      \right)=\G(Q)
\end{equation*}
The new weight column introduces a further subdivision in $\gkz=F^3_+$ along the hyperplane $H: x_2-x_3=0$, through $\q_1=\left(
                                                                                                             \begin{array}{c}
                                                                                                               1 \\
                                                                                                               0 \\
                                                                                                               0 \\
                                                                                                             \end{array}
                                                                                                           \right)=\q_2$
and $\q_5=\left(
              \begin{array}{c}
                1 \\
                1 \\
                1 \\
              \end{array}
            \right)$, leaving unchanged $\Mov(V)$ and giving
            $\P\SF(V)=\SF(V)=\{\Si_1,\Si_2,\Si_3,\Si_4\}$.
See Fig.~\ref{fig3}. The four simplicial complete fans $\Si_i$ are described by the following chambers
\begin{eqnarray*}
    \g_1=\left\langle\q_1=\q_2,\w,\q_6\right\rangle=\left\langle
            \begin{array}{ccc}
              1 & 2 & 1 \\
              0 & 1 & 0 \\
              0 & 1 & 1 \\
            \end{array}
          \right\rangle&,&\g_2=\left\langle\q_1=\q_2,\q_3,\w\right\rangle=\left\langle
                                    \begin{array}{ccc}
                                      1 & 1 & 2 \\
                                      0 & 1 & 1 \\
                                      0 & 0 & 1 \\
                                    \end{array}
                                  \right\rangle \\
                                  \g_3=\left\langle\q_3,\w,\q_5\right\rangle=\left\langle
            \begin{array}{ccc}
              1 & 2 & 1 \\
              1 & 1 & 1 \\
              0 & 1 & 1 \\
            \end{array}
          \right\rangle&,&\g_4=\left\langle\w,\q_5,\q_6\right\rangle=\left\langle
                                    \begin{array}{ccc}
                                      2 & 1 & 1 \\
                                      1 & 1 & 0 \\
                                      1 & 1 & 1 \\
                                    \end{array}
                                  \right\rangle
\end{eqnarray*}
\begin{figure}
\begin{center}
\includegraphics[width=7.5truecm]{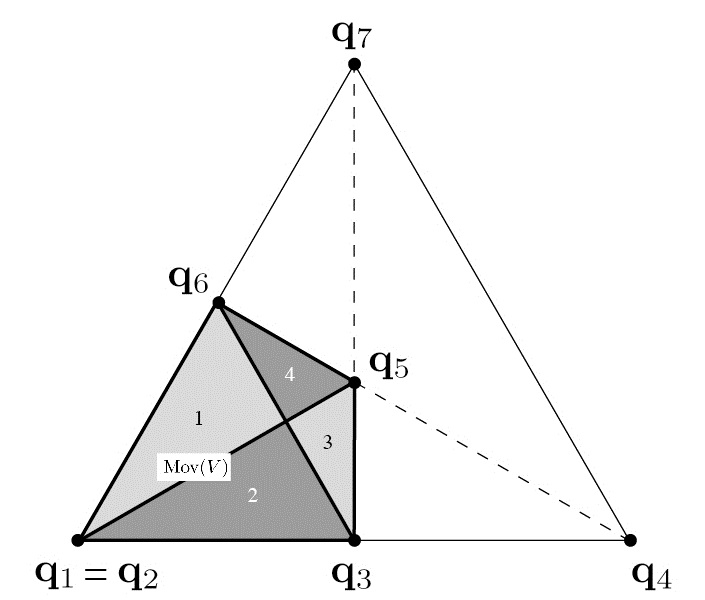}
\caption{\label{fig3}Ex.~\ref{ex:nototmaxbord}: the section of the cone $\Mov(V)$ and its four chambers, inside the Gale dual cone $\mathcal{Q}=F^3_+$, as cut out by the plane $\sum_{i=1}^3x_i^2=1$.}
\end{center}
\end{figure}

respectively, where $\q_1,\ldots,\q_7$ are the columns of $Q$ and $\w:=\q_3+\q_6=\q_1+\q_5=\q_2+\q_5$\,. Let us observe that $X(\Si_i)$ is smooth for every $i=1,\ldots,4$.

\noindent In particular both $\g_1$ and $\g_2$ are recursively maxbord chambers, $\g_1$ w.r.t. the sequence of hyperplanes $H_2, H$ and $\g_2$ w.r.t. the sequence of hyperplanes $H_3, H$: notice that none of them is totally maxbord.

Let us firstly describe the sequence of PTB's describing $X(\Si_2)$ and given by the recursively maxbord structure of $\g_2$. The last row of $Q$ gives the numerically effective primitive relation $\k=r(\pc_2)$ of the primitive collection $\pc_2=\{\v_5,\v_6,\v_7\}$, for $\Si_2$, associated with the maxbord hyperplane $H_3$. The left--upper $2\times 4$ submatrix $Q'$ w.r.t. the primitive collection $\pc_2^*$ is the same as in the previous Example \ref{ex:PTB}. Therefore the contraction of $\k$ gives the morphism $\varphi_{\k}:X(\Si_2)\longrightarrow Y\left(\Si'\right)$, where $Y$ has weight matrix $Q'$ and $\Si'$ is the unique simplicial fan in $\SF\left(\G\left(Q'\right)\right)$, associated with the chamber
$$\g_2'=\g_2\cap H_3=\left\langle\begin{array}{cc}
                                                                      1 & 1 \\
                                                                      0 & 1
                                                                    \end{array}
                  \right\rangle\,.$$
Notice that $H_3\cap H$ gives the line generated by $\q_1$, hence the hyperplane $H_2'$ of $\gkz'$, w.r.t. $\g_2'$ is clearly maxbord.
Hence, as above, $Y\cong\P\left(\mathcal{O}_{\P^1}\oplus\mathcal{O}_{\P^1}(1)\right)$ and the contraction of the primitive relation $\pc_2'$, associated with $\g_2'$, gives the structural projection $\P\left(\mathcal{O}_{\P^1}\oplus\mathcal{O}_{\P^1}(1)\right)\twoheadrightarrow\P^1$. On the other hand, by firstly subtracting the second row to the first one and then subtracting the third row to the previous ones, in $Q$, one gets the following weight matrix of $X\left(\Si_i\right)$
\begin{equation*}
Q\sim \widetilde{Q}=\left(
  \begin{array}{ccccccc}
    1 & 1 & 0 & -1 & -1 & 0 & -1 \\
    0 & 0 & 1 & 1 & 0 & -1 & -1 \\
    0 & 0 & 0 & 0 & 1 & 1 & 1 \\
  \end{array}
\right)
\end{equation*}
which gives
\begin{equation*}
X(\Si_2)\cong\P\left(\mathcal{O}_Y(h)\oplus\mathcal{O}_Y(f)\oplus\mathcal{O}_Y(f+h)\right)
\end{equation*}
where $f,h$ are the generators of $\Pic(Y)$ given by the fibre and the pull--back of the Picard generator of the base $\P^1$, respectively. Therefore $X(\Si_2)$ is obtained from $\P^1$ by the following sequence of PTB's:
$$ \xymatrix@1{X(\Si_2)\cong \P\left(\mathcal{O}_Y(h)\oplus\mathcal{O}_Y(f)\oplus\mathcal{O}_Y(f+h)\right)\ar@{>>}[r]& Y\cong \P\left(\mathcal{O}_{\P^1}\oplus\mathcal{O}_{\P^1}(1)\right)\ar@{>>}[r]&\P^1 }\,.$$
For what concerns $X(\Si_1)$, by exchanging the second and the third rows in $Q$ and reordering the columns one gets still the same weight matrix,  but now the last row gives the primitive relation $\k'=r(\pc_1)$ of the primitive collection $\pc_1=\{\v_3,\v_4,\v_5\}$ for $\Si_1$. By the previous analysis we still get
$$ \xymatrix@1{X(\Si_1)\cong \P\left(\mathcal{O}_{Y'}\oplus\mathcal{O}_{Y'}(f')\oplus\mathcal{O}_{Y'}(f'+h')\right)\ar@{>>}[r]& Y'\cong \P\left(\mathcal{O}_{\P^1}\oplus\mathcal{O}_{\P^1}(1)\right)\ar@{>>}[r]&\P^1 }\,.$$
The elementary flip $\xymatrix{X(\Si_1)\ar@{<-->}[r]& X(\Si_2)}$ is then obtained by crossing the internal wall of $\Mov(V)$ cut out by the hyperplane $H$ and the indetermination loci are described by the foci of the primitive collections supported by $H$ w.r.t. the two fans $\Si_1$ and $\Si_2$, namely $\pc_1'=\{\v_6,\v_7\}$ for $\Si_1$ and $\pc_2'=\{\v_3,\v_4\}$ for $\Si_2$, whose foci are given by the cones $\langle\v_3,\v_4\rangle$ and $\langle\v_6,\v_7\rangle$, respectively. The indetermination loci are then given by $C_1=O(\langle\v_3,\v_4\rangle)\subset X(\Si_1)$ and $C_2=O(\langle\v_6,\v_7\rangle)\subset X(\Si_2)$.

Let us finally observe that both the chambers $\g_1$ and $\g_2$ are simplicial, accordingly with  Proposition \ref{prop:split-vs-simplicial}: moreover they both admit three primitive collections, namely:
\begin{eqnarray*}
% \nonumber to remove numbering (before each equation)
  \Si_1 &\rightsquigarrow& \pc_1=\{\v_3,\v_4,\v_5\}, \pc_1'= \{\v_6,\v_7\}, \pc''=\{\v_1,\v_2\}\\
  \Si_2 &\rightsquigarrow& \pc_2=\{\v_5,\v_6,\v_7\}, \pc_2'= \{\v_3,\v_4\}, \pc''=\{\v_1,\v_2\}
\end{eqnarray*}
In particular both $\Si_1$ and $\Si_2$ are splitting fans while this clearly does not hold for $\Si_3$ and $\Si_4$.
\end{example}

The following examples will deal with singular $\Q$--factorial PWS then applying techniques introduced in this paper. Let us start by giving an example not satisfying hypothesis of Theorem \ref{thm:birWPTB}.

\begin{example}\label{ex:noWPTB} Consider the 2--dimensional PWS $X(\Si)$, of rank 3, whose reduced fan matrix is given by
\begin{equation*}
    V=\left(
        \begin{array}{ccccc}
          1 & 0 & -1 & 1 & -1 \\
          0 & 1 & -2 & 1 & -1 \\
        \end{array}
      \right)\ \Longrightarrow\ Q=\left(
                             \begin{array}{ccccc}
          1 & 1 & 0 & 0 & 1 \\
          0 & 1 & 1 & 1 & 0 \\
          0 & 0 & 0 & 1 & 1 \\
        \end{array}
                           \right)=\G(V)\,.
\end{equation*}
Then $\P\SF(V)=\SF(V)=\{\Si\}$ and the unique simplicial complete fan $\Si$ is as\-so\-cia\-ted with the unique chamber of $\Mov(V)\subset\mathcal{Q}$
\begin{equation*}
    \g=\Mov(V)=\left\langle\q_2,\w_1,\w_2,\w_3\right\rangle=\left\langle
            \begin{array}{cccc}
              1 & 1 & 2 & 1 \\
              1 & 2 & 1 & 1 \\
              0 & 1 & 1 & 1\\
            \end{array}
          \right\rangle
\end{equation*}
where $\q_1,\ldots,\q_5$ are the columns of $Q$ and
\begin{equation*}
    \w_1:=\q_2+\q_4\quad,\quad\w_2=\q_2+\q_5\quad,\quad\w_3=\q_1+\q_4=\q_3+\q_5
\end{equation*}
(see Fig.~\ref{fig4}).
\begin{figure}
\begin{center}
\includegraphics[width=8truecm]{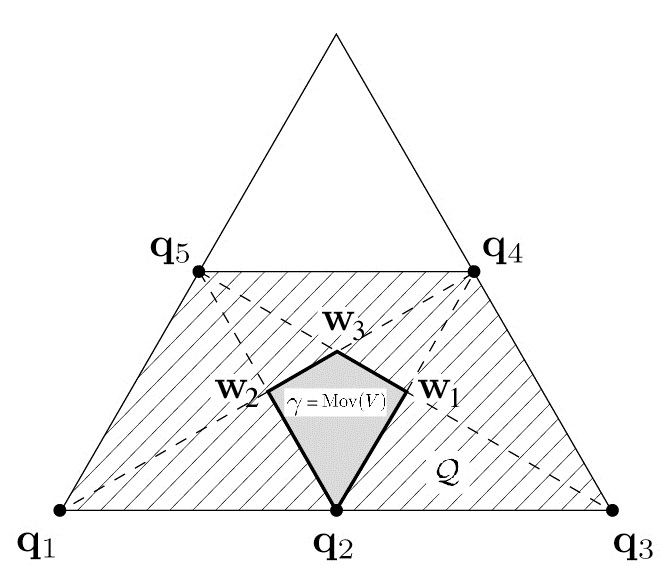}
\caption{\label{fig4}Ex.~\ref{ex:noWPTB}: the section of the cone $\g=\Mov(V)$ inside the Gale dual cone $\mathcal{Q}\subset F^3_+$, as cut out by the plane $\sum_{i=1}^3x_i^2=1$.}
\end{center}
\end{figure}
Notice that the cone $\s=\left\langle\begin{array}{cc}
                                       1 & -1 \\
                                       0 & -2
                                     \end{array}
\right\rangle$ is a maximal cone of $\Si$, hence $X(\Si)$ turns out to admit a singular point. The last row of $Q$ gives the numerical effective primitive relation $r_{\Z}(P)\in A_1(X(\Si))\cap\overline{\NE}(X(\Si))$ of the primitive collection $\pc=\{\v_4,\v_5\}$. Moreover $\g$ is not a maxbord chamber and, being the unique chamber in $\Mov(V)$, the PWS $X$ cannot be birational and isomorphic in codimension 1 to a toric cover of a WPTB, by Theorem \ref{thm:maxbord}. Notice that, by Theorem \ref{thm:birWPTB}, this fact is equivalent to assert that the left--upper $2\times 3$ submatrix $Q'$ of $Q$, w.r.t. the primitive collection $\pc$, has to violate one of the conditions in Def.~\ref{def:Wmatrix} different from condition (b). In fact $Q'=\left(
                                                                                                             \begin{array}{ccc}
                                                                                                               1 & 1 & 0 \\
                                                                                                               0 & 1 & 1 \\
                                                                                                             \end{array}
                                                                                                           \right)
$ giving that $(1,0,-1)\in\mathcal{L}_r(Q)$, contradicting the condition (f) in the Def.~\ref{def:Wmatrix}.
\end{example}

\begin{example}\label{ex:WPTB(b)} The present example is devoted to give an account of the case (b) in the proof of Thm.~\ref{thm:maxbord}. Consider a 4--dimensional PWS of rank 3 given by the following reduced fan and weight matrices
\begin{equation*}
    V=\left(
        \begin{array}{ccccccc}
          1 & 0 & 0 & -1 & 0 & 2 & -4 \\
          0 & 1 & 0 & -1 & 0 & 2 & -4 \\
          0 & 0 & 1 & -1 & 0 & 1 & -2 \\
          0 & 0 & 0 & 0 & 1 & -1 & 1 \\
        \end{array}
      \right)\Rightarrow Q=\left(
                             \begin{array}{ccccccc}
                               1 & 1 & 1 & 1 & 0 & 0 & 0\\
                               0 & 0 & 1 & 2 & 1 & 1 & 0\\
                               0 & 0 & 0 & 0 & 1 & 2 & 1\\
                             \end{array}
                           \right)=\G(V)
\end{equation*}
The first interest of this example is in the fact that
$$|\P\SF(V)|=8<10=|\SF(V)|$$
meaning that $V$ carries two distinct fans of $\Q$--factorial complete toric varieties of rank 3 which are not projective. In particular $X(\Si)$ is singular for every fan $\Si\in\SF(V)$.
\begin{figure}
\begin{center}
\includegraphics[width=7.5truecm]{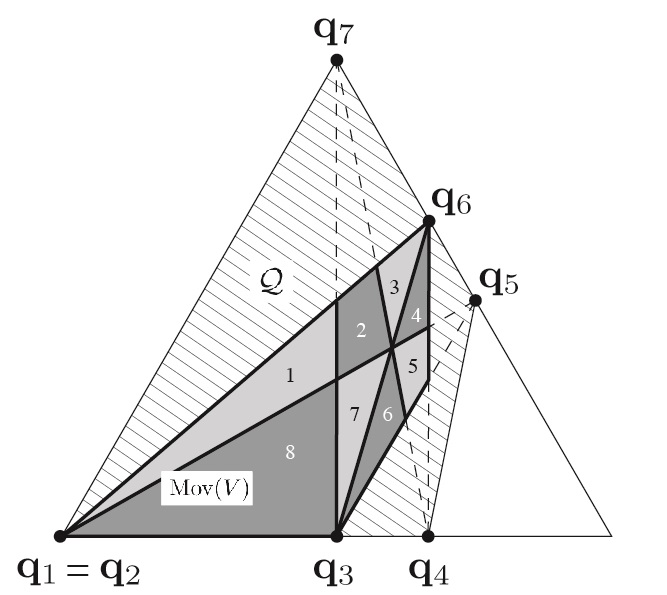}
\caption{\label{fig5}Ex.~\ref{ex:WPTB(b)}: this is the section cut out by the plane $\sum x_i^2=1$ of the cone $\Mov(V)$, with its eight chambers, inside the Gale dual cone $\mathcal{Q}\subset F^3_+$.}
\end{center}
\end{figure}
\noindent Figure \ref{fig5} describes the Gale dual cone $\gkz=\langle Q\rangle$ and $\Mov(V)$ with its eight chambers $\g_i\,,i=1,\ldots,8$. Following the notation introduced Example \ref{ex:PTB}, the associated fans $\Si_i:=\Si_{\g_i}$ are given by all the faces of their maximal cones. The two non projective fans are generated by the following list of maximal cones:
\begin{eqnarray*}
\Si_9(4)&=&\{\{2, 3, 4, 7\}, \{2, 3, 4, 6\}, \{2, 3, 5, 7\}, \{2, 3, 5, 6\}, \{1, 3, 4, 7\}, \{1, 2, 4, 7\},\\
 &&\{1, 3, 4, 6\}, \{1, 2, 4, 6\}, \{1, 3, 5, 7\}, \{1, 2, 5, 7\}, \{1, 3, 5, 6\}, \{1, 2, 5, 6\}\}\\
 \Si_{10}(4)&=&\{\{2, 4, 5, 7\}, \{2, 4, 6, 7\}, \{2, 3, 5, 7\}, \{2, 3, 6, 7\}, \{1, 4, 5, 7\}, \{1, 4, 6, 7\},\\
 &&\{1, 2, 4, 5\}, \{1, 2, 4, 6\}, \{1, 3, 5, 7\}, \{1, 2, 3, 5\}, \{1, 2, 3, 6\}, \{1, 3, 6, 7\}\}\,.
 \end{eqnarray*}
The intersection of the cones in the associated bunches of cones inside $\gkz$, gives, in both the cases, the 1-dimensional cone
$$\langle\w_1\rangle\quad \text{with}\quad\w_1=\q_3+\q_6=\q_1+2\q_5=\q_2+2\q_5=\q_4+2\q_7=\left(\begin{array}{c}
1 \\
2 \\
2 \\
\end{array}
\right)$$
which is the primitive generator of the common ray to the six chambers $\g_i,2\leq i\leq 7$ (see Fig.~\ref{fig5}).
Among the 8 distinct chambers of $\Mov(V)$, giving the projective fans, the unique maxbord chamber is given by
  \begin{equation*}
  \g_8=\langle\q_1=\q_2,\q_3,\w\rangle,\quad\text{with}\quad \w=\q_3+\q_7=\q_1+\q_5=\q_2+\q_5=\left(
                                                     \begin{array}{c}
                                                       1 \\
                                                       1 \\
                                                       1 \\
                                                     \end{array}
                                                   \right)
 \end{equation*}
  which actually is also a recursively (no totally) maxbord chamber w.r.t. the sequence of hyperplanes $H_3,H$, where $H$ is the hyperplane $x_2-x_3=0$, through $\q_1$ (or, equivalently, $\q_2$) and $\q_5$. The maxbord hyperplane $H_3$ supports the nef primitive collection $\pc=\{\v_5,\v_6,\v_7\}$ whose numerically effective primitive relation $\k=r_{\Z}(\pc)$ is given by the last row of $Q$. Notice that the left--upper $2\times 4$ submatrix $Q'=\left(
                            \begin{array}{cccc}
                              1 & 1 & 1 & 1 \\
                              0 & 0 & 1 & 2 \\
                            \end{array}
                          \right)
$ of $Q$ is a non--reduced $W$--matrix: in fact $\mathcal{L}_r\left(Q^{\{3\}}\right)\subset\Z^3$ has cotorsion, admitting the generator $(0,0,2)$. We are then in case (b) in the proof of Thm.~\ref{thm:maxbord}. Using the same notation therein, we then get
\begin{eqnarray}\label{A,B}
  A&=&\diag(1,1/2,1/2)\in\GL_3(\Q) \\
  \nonumber
  B&=&\diag(1,1,2,1,2,2,2)\in\GL_7(\Q)\cap\mathbf{M}_7(\Z)
\end{eqnarray}
which give
\begin{eqnarray*}
  \widetilde{Q}=AQB&=&\left(
                             \begin{array}{ccccccc}
                               1 & 1 & 2 & 1 & 0 & 0 & 0\\
                               0 & 0 & 1 & 1 & 1 & 1 & 0\\
                               0 & 0 & 0 & 0 & 1 & 2 & 1\\
                             \end{array}
                           \right)\\
  \widetilde{V}=\G\left(\widetilde{Q}\right)&=&\left(
        \begin{array}{ccccccc}
          1 & 0 & 0 & -1 & 0 & 1 & -2 \\
          0 & 1 & 0 & -1 & 0 & 1 & -2 \\
          0 & 0 & 1 & -2 & 0 & 1 & -2 \\
          0 & 0 & 0 & 0 & 1 & -1 & 1 \\
        \end{array}
      \right)\\
      \end{eqnarray*}
Recall that $V$ is a $CF$--matrix then $H:=\HNF(V^T)=\left(
                                                                 \begin{array}{c}
                                                                   I_4 \\
                                                                   \mathbf{0}_{3,4} \\
                                                                 \end{array}
                                                               \right)$ \cite[Thm.~2.1(4)]{RT-QUOT}. Let $U\in \GL_{7}(\Z)$ such that $U\cdot V^T = H$. Then the upper $4$ rows of $U$ give $^4U\cdot V^T = I_4$ (recall notation in list \ref{ssez:lista}).
Therefore
\begin{equation*}
    V^T\cdot C=B\cdot \widetilde{V}^T\ \Longrightarrow\ C=\ ^4U\cdot  B\cdot \widetilde{V}^T=\left(
            \begin{array}{cccc}
              1 & 0 & 0 & 0 \\
              0 & 1 & 0 & 0 \\
              0 & 0 & 2 & 0 \\
              0 & 0 & 0 & 2 \\
            \end{array}
          \right)\,.
\end{equation*}
Since $\det C=4$, calling $\overline{f}:N\to\widetilde{N}$ the map represented by $C^T$, one gets that $\overline{f}(N)$ turns out to be a subgroup of index 4 in $\widetilde{N}$. The fan $\widetilde{\Si}=\overline{f}_{\R}(\Si_8)$ is then the fan defined by the simplicial chamber
\begin{equation}\label{gammatilde}
    \widetilde{\g}:=\left\langle\widetilde{\q}_1=\widetilde{\q}_2,\widetilde{\q}_3,\widetilde{\w}\right\rangle=\left\langle\begin{array}{ccc}
                                                 1 & 2 & 2 \\
                                                 0 & 1 & 1 \\
                                                 0 & 0 & 1
                                               \end{array}
    \right\rangle\subset\Mov\left(\widetilde{V}\right)
\end{equation}
where $\widetilde{\q}_i$ are the columns of $\widetilde{Q}$ and $\widetilde{\w}=\widetilde{\q}_3+\widetilde{\q}_7=\left(
                                                                            \begin{array}{c}
                                                                              2 \\
                                                                              1 \\
                                                                              1 \\
                                                                            \end{array}
                                                                          \right)
$. It is still a recursively maxbord chamber w.r.t. the same sequence of hyperplanes $H_3,H$. The maxbord hyperplane $H_3$ supports the nef primitive collection $\widetilde{\pc}=\{\widetilde{\v}_5,\widetilde{\v}_6,\widetilde{\v}_7\}$ whose numerically effective primitive relation $\widetilde{\k}=r_{\Z}(\widetilde{\pc})$ is given by the last row of $\widetilde{Q}$. Now that the left--upper $2\times 4$ submatrix $\widetilde{Q}'=\left(
                            \begin{array}{cccc}
                              1 & 1 & 2 & 1 \\
                              0 & 0 & 1 & 1 \\
                            \end{array}
                          \right)
$ of $\widetilde{Q}$ turns out to be a reduced $W$--matrix: then we are now in case (a) of the proof of Thm.~\ref{thm:maxbord} and $\widetilde{X}(\widetilde{\Si})$ turns out to be either a WPTwB or a WPTB.

Let us first describe the toric cover $f:X\to \widetilde{X}$ induced by the homomorphism $\overline{f}:N\cong\Z^n\to \widetilde{N}\cong\Z^n$, represented by the transposed matrix $C^T$. Recalling the Cox's description of $X$ and $\widetilde{X}$ as geometric quotients, the covering $f$ is then completely described by taking the non-trivial entries of the diagonal matrix $B^T$ as exponents of the Cox ring variables of $X$ to obtain the Cox ring variables of $\widetilde{X}$. One then gets that
 \begin{itemize}
   \item[] \emph{$f$ is a double covering ramified along the torus-invariant divisors $D_3$, $D_5$, $D_6$ and $D_7$.}
 \end{itemize}
On the other hand, to describe the structure of $\widetilde{X}$, notice that
\begin{equation*}
    \widetilde{Q}'\sim \left(
                         \begin{array}{cccc}
                           1 & 1 & 0 & -1 \\
                           0 & 0 & 1 & 1 \\
                         \end{array}
                       \right)
\end{equation*}
Hence, going on as in the previous Example \ref{ex:PTB} and applying Theorem \ref{thm:contraibile}, one gets that $\widetilde{X}$ is a WPTB over the PTB $Y\cong\P\left(\mathcal{O}_{\P^1}\oplus\mathcal{O}_{\P^1}(1)\right)$, whose weights are given by $W=(1,2,1)$. The fibration is given by the contraction $\varphi_{\widetilde{\k}}$ of the class $\widetilde{\k}$. After suitable operations on the rows of $\widetilde{Q}$, one gets the following weight matrix of $\widetilde{X}$
\begin{equation*}
\widetilde{Q}\sim \left(
  \begin{array}{ccccccc}
    1 & 1 & 0 & -1 & -2 & -2 & 0 \\
    0 & 0 & 1 & 1 & 0 & -1 & -1 \\
    0 & 0 & 0 & 0 & 1 & 2 & 1 \\
  \end{array}
\right)
\end{equation*}
hence giving
\begin{equation*}
\widetilde{X}\left(\widetilde{\Si}\right)\cong\P^{(1,2,1)}\left(\mathcal{O}_Y(2h)\oplus\mathcal{O}_Y(f+2h)\oplus\mathcal{O}_Y(f)\right)
\end{equation*}
where $f,h$ are the generators of $\Pic(Y)$ given by the fibre and the pull--back of the Picard generator $\mathcal{O}_{\P^1}(1)$ of the base $\P^1$, respectively.

In conclusion: \emph{$X(\Si_8)$ is obtained from $\P^1$ by means of the following sequence of toric covers and WPTB's}
\begin{equation*}
    \xymatrix{X\ar[r]^-{2:1}_-f&\P^{(1,2,1)}\left(\mathcal{O}_Y(2h)\oplus\mathcal{O}_Y(f+2h)\oplus\mathcal{O}_Y(f)\right)\stackrel{\varphi_{\widetilde{\k}}}{\twoheadrightarrow}
    Y\cong\P\left(\mathcal{O}_{\P^1}\oplus\mathcal{O}_{\P^1}(1)\right)\twoheadrightarrow\P^1}
\end{equation*}
Let us finally observe that, accordingly with Proposition \ref{prop:split-vs-simplicial}, the fan $\Si$ admits only the following primitive collections
\begin{equation}\label{splitting}
    \pc=\{\v_5,\v_6,\v_7\}\quad,\quad\pc'=\{\v_3,\v_4\}\quad,\quad\pc''=\{\v_1,\v_2\},.
\end{equation}
\end{example}

\begin{example}\label{ex:WPTB(c)} The present example is devoted to give an account of the case (c) in the proof of Thm.~\ref{thm:maxbord}. Moreover this example is obtained by the previous Ex.~\ref{ex:WPTB(b)} by \emph{breaking the symmetry} around the ray $\left\langle\begin{array}{c}
                                                                              1 \\
                                                                              2 \\
                                                                              2
                                                                            \end{array}
\right\rangle\in\Ga(1)$ of the secondary fan $\Ga$. This is enough to get a simplicial and complete fan $\Si\in\SF(V)$ such that $\g_{\Si}=\Nef(X(\Si))=0$.

\noindent Consider a 4--dimensional PWS of rank 3 given by the following reduced fan and weight matrices
\begin{equation*}
    V=\left(
        \begin{array}{ccccccc}
          1 & 0 & -1 & 0 & 0 & 6 & -12 \\
          0 & 1 & -1 & 0 & 0 & 4 & -8 \\
          0 & 0 & 0 & 1 & 0 & -2 & 4 \\
          0 & 0 & 0 & 0 & 1 & -1 & 1 \\
        \end{array}
      \right)\Rightarrow Q=\left(
                             \begin{array}{ccccccc}
                               1 & 1 & 1 & 0 & 0 & 0 & 0\\
                               0 & 2 & 6 & 2 & 1 & 1 & 0\\
                               0 & 0 & 0 & 0 & 1 & 2 & 1\\
                             \end{array}
                           \right)=\G(V)
\end{equation*}
\begin{figure}
\begin{center}
\includegraphics[width=11truecm]{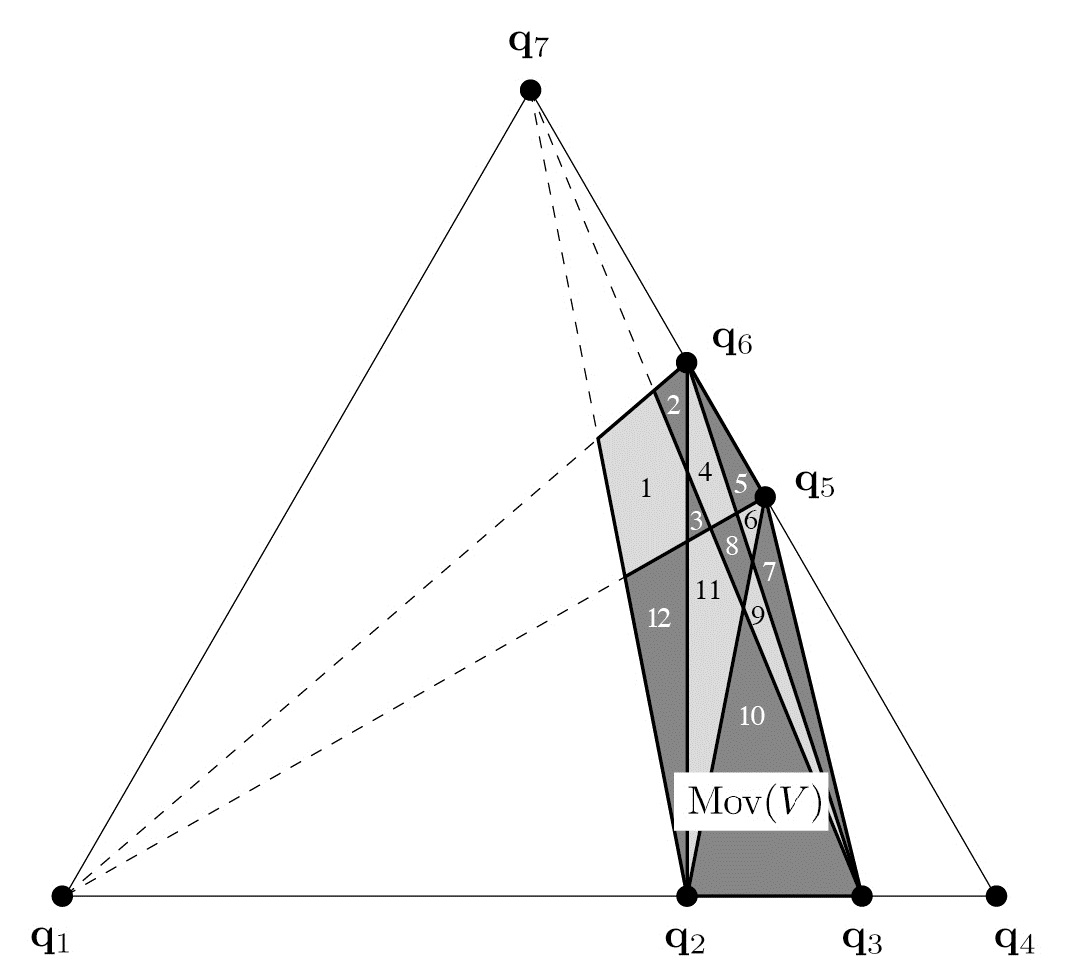}
\caption{\label{fig6}Ex.~\ref{ex:WPTB(c)}: the section of the cone $\Mov(V)$ and its twelve chambers, inside the Gale dual cone $\mathcal{Q}=F^3_+$, as cut out by the plane $\sum_{i=1}^3x_i^2=1$.}
\end{center}
\end{figure}
Then:
\begin{itemize}
  \item $|\P\SF(V)|=12<13=|\SF(V)|$, meaning that $V$ carries one fan of a $\Q$--factorial complete toric varieties of rank 3 which is not projective; explicitly this is generated as the faces of maximal cones in
      \begin{eqnarray*}
      % \nonumber to remove numbering (before each equation)
        &\Si_{13}(4) = \{\{3, 4, 5, 7\}, \{3, 4, 6, 7\}, \{2, 3, 5, 7\}, \{2, 3, 5, 6\}, \{2, 3, 6, 7\}, \{2, 4, 5, 7\},&\\
        &\{2, 4, 6, 7\}, \{1, 3, 4, 5\}, \{1, 3, 5, 6\}, \{1, 3, 4, 6\}, \{1, 2, 4, 5\}, \{1, 2, 5, 6\}, \{1, 2, 4, 6\}\}&
      \end{eqnarray*}
      The intersection of the cones in the associated bunch of cones inside $\gkz$, gives the trivial cone $\langle 0\rangle$.
  \item $X(\Si)$ is singular for every fan $\Si\in\SF(V)$,
  \item among the 12 distinct chambers of $\Mov(V)$, giving the projective fans, there are two maxbord chambers which are both no recursively maxbord chambers: these are given by the simplicial chambers (see Fig.~\ref{fig6})
\begin{equation*}
    \g_5:=\left\langle\q_5,\q_6,\w_1\right\rangle=\left\langle\begin{array}{ccc}
                                                 0 & 0 & 1 \\
                                                 1 & 1 & 12 \\
                                                 1 & 2 & 12
                                               \end{array}
    \right\rangle\quad,\quad\g_{10}:=\left\langle\q_2,\q_3,\w_2\right\rangle=\left\langle\begin{array}{ccc}
                                                 1 & 1 & 1 \\
                                                 2 & 6 & 6 \\
                                                 0 & 0 & 4
                                               \end{array}
    \right\rangle
\end{equation*}
which give the fans of faces of maximal cones in the following lists
\begin{eqnarray*}
% \nonumber to remove numbering (before each equation)
   &\Si_5(4)&  =\{\{2, 3, 5, 7\}, \{2, 3, 4, 7\}, \{2, 3, 5, 6\}, \{2, 3, 4, 6\}, \{1, 3, 5, 7\}, \{1, 3, 4, 7\}, \\
   &&  \{1, 3, 5, 6\}, \{1, 3, 4, 6\}, \{1, 2, 5, 7\}, \{1, 2, 4, 7\}, \{1, 2, 5, 6\}, \{1, 2, 4, 6\}\}\\
   &\Si_{10}(4)& =\{\{2, 3, 5, 7\}, \{2, 3, 5, 6\}, \{2, 3, 6, 7\}, \{2, 4, 5, 7\}, \{2, 4, 5, 6\}, \{2, 4, 6, 7\}, \\
   && \{1, 3, 5, 7\}, \{1, 3, 5, 6\}, \{1, 3, 6, 7\}, \{1, 4, 5, 7\}, \{1, 4, 5, 6\}, \{1, 4, 6, 7\}\}
\end{eqnarray*}
\item $\g_5$ is maxbord w.r.t. the hyperplane $H_1$ and $\g_{10}$ is maxbord w.r.t. the hyperplane $H_3$,
\item the maxbord hyperplane $H_1$ gives the nef primitive collection $\pc'=\{\v_1,\v_2,\v_3\}$ for $\Si_{5}=\Si_{\g_{5}}$ whose numerically effective primitive relation $\k'=r_{\Z}(\pc')$ gives the first row of $Q$,
\item the maxbord hyperplane $H_3$ gives the nef primitive collection $\pc=\{\v_5,\v_6,\v_7\}$ for $\Si_{10}=\Si_{\g_{10}}$ whose numerically effective primitive relation $\k=r_{\Z}(\pc)$ gives the last row of $Q$.
\end{itemize}
Let us start by studying $X(\Si_{10})$.

\noindent The left--upper $2\times 4$ submatrix $Q'=\left(
                            \begin{array}{cccc}
                              1 & 1 & 1 & 0 \\
                              0 & 2 & 6 & 2 \\
                            \end{array}
                          \right)
$ is a positive matrix which does not satisfy condition (b) in the Def.~\ref{def:Wmatrix}: we are then in case (c) in the proof of Thm.~\ref{thm:maxbord}. Using the same notation therein, we then get
\begin{eqnarray*}
  A&=&\diag(1,1/2,1/2)\in\GL_3(\Q) \\
  B&=&\diag(1,1,1,1,2,2,2)\in\GL_7(\Q)\cap\mathbf{M}_7(\Z) \\
\end{eqnarray*}
which give
\begin{eqnarray*}
  \widetilde{Q}=AQB&=&\left(
                             \begin{array}{ccccccc}
                               1 & 1 & 1 & 0 & 0 & 0 & 0\\
                               0 & 1 & 3 & 1 & 1 & 1 & 0\\
                               0 & 0 & 0 & 0 & 1 & 2 & 1\\
                             \end{array}
                           \right)\\
                           \widetilde{V}=\G\left(\widetilde{Q}\right)&=&\left(
        \begin{array}{ccccccc}
          1 & 0 & -1 & 0 & 0 & 3 & -6 \\
          0 & 1 & -1 & 0 & 0 & 2 & -4 \\
          0 & 0 & 0 & 1 & 0 & -1 & 2 \\
          0 & 0 & 0 & 0 & 1 & -1 & 1 \\
        \end{array}
      \right)
\end{eqnarray*}
As in the previous Ex.~\ref{ex:nototmaxbord}, let $U\in \GL_{7}(\Z)$ such that $U\cdot V^T = H=\left(
                                                                 \begin{array}{c}
                                                                   I_4 \\
                                                                   \mathbf{0}_{3,4} \\
                                                                 \end{array}
                                                               \right)$. Then
\begin{equation*}
    V^T\cdot C=B\cdot \widetilde{V}^T\ \Longrightarrow\ C=\ ^4U\cdot  B\cdot \widetilde{V}^T=\left(
            \begin{array}{cccc}
              1 & 0 & 0 & 0 \\
              0 & 1 & 0 & 0 \\
              0 & 0 & 1 & 0 \\
              0 & 0 & 0 & 2 \\
            \end{array}
          \right)\,.
\end{equation*}
Since $\det C=2$, calling $\overline{f}:N\to\widetilde{N}$ be the $\Z$--linear map represented by $C^T$, one gets that $\overline{f}(N)$ turns out to be a subgroup of index 2 in $\widetilde{N}$. Consider the fan $\widetilde{\Si}=\overline{f}_{\R}(\Si_{10})$ and the associated $\Q$--factorial projective toric variety $\widetilde{X}\left(\widetilde{\Si}\right)$.
The left--upper $2\times 4$ submatrix $\widetilde{Q}'=\left(
                            \begin{array}{cccc}
                              1 & 1 & 1 & 0 \\
                              0 & 1 & 3 & 1 \\
                            \end{array}
                          \right)
$ is now a $W$--matrix, meaning that $\widetilde{Q}$ satisfies condition (a) in the proof of Theorem \ref{thm:maxbord}, and $\widetilde{X}$ is either a WPTB or a WPTwB.
Subtracting the third row to the second one, we get
$$\widetilde{Q}\sim\left(
                             \begin{array}{ccccccc}
                               1 & 1 & 1 & 0 & 0 & 0 & 0\\
                               0 & 1 & 3 & 1 & 0 & -1 & -1\\
                               0 & 0 & 0 & 0 & 1 & 2 & 1\\
                             \end{array}
                           \right)
$$
and $\widetilde{X}$ is a WPTB if and only the columns of the right upper $2\times 3$ submatrix $\widetilde{Q}''=\left(
                                                                                                                \begin{array}{ccc}
                                                                                                                  0 & 0 & 0 \\
                                                                                                                  0 & -1 & -1 \\
                                                                                                                \end{array}
                                                                                                              \right)
$ belongs to $\Pic(Y)$, where $Y$ is the toric surface of rank 2 determined by the fan associated with unique chamber of $$\Mov\left(\G\left(\widetilde{Q}'\right)\right)=\left\langle\begin{array}{cc}
                                                         1 & 1 \\
                                                         1 & 3
                                                       \end{array}
\right\rangle\,.$$
Recalling \cite[Thm.~2.9(2)]{RT-LA&GD} a basis of $\Pic(Y)$ can be determined by following the procedure described in \cite[\S~1.2.3]{RT-LA&GD} and obtaining:
$$\Pic(Y)=\mathcal{L}(L_1,L_2)\cong\Z^2\quad\text{where}\quad L_1=\left(
                                                           \begin{array}{c}
                                                             2 \\
                                                             0 \\
                                                           \end{array}
                                                         \right)\ ,\ L_2=\left(
                                                           \begin{array}{c}
                                                             -1 \\
                                                             3 \\
                                                           \end{array}
                                                         \right)\,.
$$
Hence $-(L_1+2L_2)=\left(
         \begin{array}{c}
           0 \\
           -6 \\
         \end{array}
       \right)
$, meaning that the Cartier indexes of Weil divisors whose classes are the columns of $\widetilde{Q}''$ are $l_0=1,l_1=l_2= 6$,
respectively. Recalling Prop.~\ref{prop:WPTwB}, $\widetilde{X}$ turns out to be a WPTwB and in particular a toric cover of the WPTB $\P^{W'}(\mathcal{E})$ where $W'=W=(1,2,1)$ is the reduced weight vector of $(l_0w_0,l_1w_1,l_2w_2)=(1,12,6)$ and $\mathcal{E}=\mathcal{O}_Y\oplus\mathcal{O}_Y(6D'_4)^{\oplus 2}$.

\noindent To explicitly determine the toric cover $g:\widetilde{X}\to\P^W(\mathcal{E})$ one has to determine matrices $\Delta,\Lambda,\Phi$ as in the proof of Prop.~\ref{prop:WPTwB}. Namely
\begin{eqnarray*}
  \Delta&=&\diag(1,1,1/6)\in\GL_3(\Q) \\
  \Lambda&=&\diag(1,1,1,1,6,6,6)\in\GL_7(\Q)\cap\mathbf{M}_7(\Z) \\
\end{eqnarray*}
which give
\begin{eqnarray*}
  \widetilde{\widetilde{Q}}=\Delta\widetilde{Q}\Lambda&=&\left(
                             \begin{array}{ccccccc}
                               1 & 1 & 1 & 0 & 0 & 0 & 0\\
                               0 & 1 & 3 & 1 & 6 & 6 & 0\\
                               0 & 0 & 0 & 0 & 1 & 2 & 1\\
                             \end{array}
                           \right)\\
                           \widetilde{\widetilde{V}}=\G\left(\widetilde{\widetilde{Q}}\right)&=&\left(
        \begin{array}{ccccccc}
          1 & 0 & -1 & 3 & 0 & 0 & 0 \\
          0 & 1 & -1 & 2 & 0 & 0 & 0 \\
          0 & 0 & 0 & 6 & 0 & -1 & 2 \\
          0 & 0 & 0 & 0 & 1 & -1 & 1 \\
        \end{array}
      \right)
\end{eqnarray*}
Hence
$$\P^W(\mathcal{E})=\P^{(1,2,1)}\left(\mathcal{O}_Y\oplus\mathcal{O}_Y(6D'_4)^{\oplus 2}\right)=\widetilde{\widetilde{X}}\left(\widetilde{\widetilde{\Si}}\right),\ \text{with}\ \widetilde{\widetilde{\Si}}=\overline{g}_{\R}\left(\widetilde{\Si}\right)=\overline{(g\circ f)}_{\R}\left(\Si_{10}\right)\,.$$
Moreover, calling $\widetilde{U}\in \GL_{7}(\Z)$ such that $\widetilde{U}\cdot \widetilde{V}^T = \widetilde{H}=\left(
                                                                 \begin{array}{c}
                                                                   I_4 \\
                                                                   \mathbf{0}_{3,4} \\
                                                                 \end{array}
                                                               \right)$, then
\begin{equation*}
    \widetilde{V}^T\cdot \Phi=\Lambda\cdot \widetilde{\widetilde{V}}^T\ \Longrightarrow\ \Phi=\ ^4\widetilde{U}\cdot  \Lambda\cdot \widetilde{\widetilde{V}}^T=\left(
            \begin{array}{cccc}
              1 & 0 & 0 & 0 \\
              0 & 1 & 0 & 0 \\
              3 & 2 & 6 & 0 \\
              0 & 0 & 0 & 6 \\
            \end{array}
          \right)\,.
\end{equation*}
Then $\overline{g}\left(\widetilde{N}\right)$ is a subgroup of index $\det(\Phi)=36$ of $\widetilde{\widetilde{N}}$.
Therefore $X(\Si_{10})$ turns out to admit the following geometric structure:
\begin{equation*}
    \xymatrix{X\left(\Si_{10}\right)\ar@{>>}[r]_-f^-{2:1}&\widetilde{X}\ar@{>>}[r]_-g^-{36:1}&\P^{(1,2,1)}\left(\mathcal{O}_Y\oplus\mathcal{O}_Y(6D'_4)^{\oplus 2}\right)\ar@{>>}[r]_-{\varphi_{\widetilde{\widetilde{\k}}}}&Y}
\end{equation*}
where:
\begin{itemize}
  \item $f$ is a double toric cover ramified along the torus invariant Weil divisors $D_5,D_6,D_7$ of $X$, as one can immediately deduce from the diagonal matrix $B$,
  \item $g$ is a $36:1$ toric cover ramified along the torus invariant Weil divisors $\widetilde{D}_5,\widetilde{D}_6,\widetilde{D}_7$ of $\widetilde{X}$, as one can immediately deduce from the diagonal matrix $\Lambda$,
  \item $\varphi_{\widetilde{\widetilde{\k}}}$ is the contraction morphism of the contractible class $\widetilde{\widetilde{\k}}=g\circ f(\k)$, under the notation introduced in Def.~\ref{def:pseudocontraibile}, meaning that $\k$ is a pseudo-contractible class.
\end{itemize}
For what concerns the further maxbord chamber $\g_5$, the study of $X(\Si_5)$ goes over in the same way as just done for $X(\Si_{10})$, after exchanging each other the first and third rows of $Q$ and reordering the columns to still get a REF matrix, hence obtaining
\begin{equation}\label{Qnuova}
    Q\sim\left(
                             \begin{array}{ccccccc}
                               1 & 2 & 1 & 0 & 0 & 0 & 0\\
                               0 & 1 & 1 & 2 & 6 & 2 & 0\\
                               0 & 0 & 0 & 0 & 1 & 1 & 1\\
                             \end{array}
                           \right)
\end{equation}
Let us reassign $Q$ as the right matrix in (\ref{Qnuova}). Then we get an analogous reassignment for
\begin{equation*}
    V=\G(Q)=\left(
        \begin{array}{ccccccc}
          1 & 1 & -3 & 0 & 0 & 1 & -1 \\
          0 & 2 & -4 & 0 & 0 & 1 & -1 \\
          0 & 0 & 0 & 1 & 0 & -1 & 1 \\
          0 & 0 & 0 & 0 & 1 & -3 & 2 \\
        \end{array}
      \right)
\end{equation*}
Now the left upper $2\times 4$ submatrix $Q'=\left(
           \begin{array}{cccc}
             1 & 2 & 1 & 0 \\
             0 & 1 & 1 & 2 \\
           \end{array}
         \right)$ is a $W$--matrix, implying that $X(\Si_5)$ is already either a WPTB or a WPTwB over the toric surface $Y'$ of rank 2 and determined by the fan associated with unique chamber of $$\Mov\left(\G\left(\widetilde{Q}'\right)\right)=\left\langle\begin{array}{cc}
                                                         2 & 1 \\
                                                         1 & 1
                                                       \end{array}
\right\rangle\,.$$
Still applying \cite[Thm.~2.9(2)]{RT-LA&GD} we get
$$\Pic(Y')=\mathcal{L}(L_1,L_2)\cong\Z^2\quad\text{where}\quad L_1=\left(
                                                           \begin{array}{c}
                                                             4 \\
                                                             0 \\
                                                           \end{array}
                                                         \right)\ ,\ L_2=\left(
                                                           \begin{array}{c}
                                                             0 \\
                                                             2 \\
                                                           \end{array}
                                                         \right)\,.
$$
By subtracting 6 times the third row to the second one in $Q$ we get
$$Q\sim\left(
                             \begin{array}{ccccccc}
                               1 & 2 & 1 & 0 & 0 & 0 & 0\\
                               0 & 1 & 1 & 2 & 0 & -4 & -6\\
                               0 & 0 & 0 & 0 & 1 & 1 & 1\\
                             \end{array}
                           \right)
$$
and we see that the column of the right upper $2\times 3$ submatrix $\left(
                                                                       \begin{array}{ccc}
                                                                         0 & 0 & 0 \\
                                                                         0 & -4 & -6 \\
                                                                       \end{array}
                                                                     \right)
$ belongs to $\Pic(Y')$. Then $X(\Si_5)$ is a WPTB over $Y'$ and better
$$\xymatrix{X(\Si_5)=\P\left(\mathcal{O}_{Y'}\oplus\mathcal{O}_{Y'}(2D'_4)\oplus\mathcal{O}_{Y'}(3D'_4)\right)\ar@{>>}[rr]^-{\varphi_{\k'}}&& Y'}$$
is actually a PTB over $Y'$, whose fibers are isomorphic to $\P^2$ since $W=(1,1,1)$: the fibration morphism $\varphi_{\k'}$ is given by the contraction of the contractible class $\k'=r_{\Z}(\pc')$.

For the remaining ten fans $\Si_i$, $i=1,\ldots,4,6,\ldots,9,11,12$, by Theorem \ref{thm:birWPTB} we can only say that $X(\Si_i)$ is a toric flip either of $X(\Si_{10})$, hence of a toric cover of a WPTB, or of $X(\Si_5)$, hence of a PTB.
  \end{example}

\begin{remark}\label{rem:FP}
For smooth threefolds O.~Fujino and S.~Payne \cite{FP} proved that $\Nef(X)\neq 0$ for $r\leq 4$. In the previous Example \ref{ex:WPTB(c)} the fan $\Si_{13}$ is associated with the trivial cone $\langle 0\rangle\subset \gkz$, giving a $4$--dimensional $\Q$--factorial complete toric variety $X$ with Picard number $r=3$ such that $\Nef(X)=0$, hence showing that the Fujino--Payne inequality does no more hold when dropping the smoothness hypo\-the\-sis. One might object that the given example has dimension 4, while the result of Fujino and Payne applies in dimension 3, but in the forthcoming paper \cite{RT-r=2} we will provide an example of a $3$--dimensional $\Q$--factorial complete toric variety $X$ with Picard number $r=3$ such that $\Nef(X)=0$: for reasons of space and in\-he\-ren\-ce we do not propose that example in the present paper. Anyway we claim that example to be a sharp one, both for the dimension $n$ and the rank $r$ since, on the one hand, it is well known that every complete toric variety of dimension $\leq 2$ is projective (see e.g. \cite[\S~8, Prop.~8.1]{MO}) and, on the other hand, in the same paper we will prove that \emph{a $\Q$--factorial, complete, toric variety of rank $r=2$ is projective.}
\end{remark}

\section{Classifying $\Q$--factorial projective toric varieties}\label{sez:Qfproj}

In the present section we will apply the results obtained in \S~\ref{sez:Batyrev-rivisto} for a PWS to the case of a singular $\Q$--factorial projective variety, recalling that the latter is always a finite quotient of a PWS, after \cite[Thm.~2.2]{RT-QUOT}.

\subsection{1-coverings of $\Q$-factorial complete tric varieties}
For reader convenience we begin this section by presenting some definitions and results from \cite{RT-QUOT}, needed for the comprehension of results given in the next \S~\ref{ssez:classificazione}.

For ease, let us here assume $X$ and $Y$ be normal and complete algebraic varieties, which is enough for our purpose.

\begin{definition}\label{def:1-covering}[see Def.~1.9 in \cite{RT-QUOT}] A finite surjective morphism $\varphi:Y\rightarrow X$ is called a \emph{covering in codimension $1$} (or simply a \emph{$1$--covering}) if it is unramified in codimension $1$, that is, there exists a subvariety $V\subseteq X$ such that $\codim_X V\geq 2$ and $\varphi|_{Y_V}$ is a topological covering, where $Y_V:=\varphi^{-1}(X\setminus V)$. Moreover a \emph{universal covering in codimension $1$} is a $1$-covering $\varphi:Y\rightarrow X$ such that for any $1$--covering $\phi:X'\rightarrow X$ of $X$ there exists a $1$--covering $f:Y\rightarrow X'$ such that $\varphi=\phi\circ f$.
\end{definition}

Recall notation introduced in Definitions \ref{def:Fmatrice}, \ref{def:Wmatrix} and \ref{def:PWS}. A $\Q$--factorial complete toric variety $X$ is given by a reduced $F$--matrix $V$ and a fan $\Si\in\SF(V)$ such that $X=X(\Si)$. Let $Q=\G(V)$ be a positive REF $W$--matrix. Then, by \cite[Prop.~3.12(1)]{RT-LA&GD}, $\widehat{V}=\G(Q)$ is a $CF$--matrix and the choice of $\Si\in\SF(V)$ uniquely determines a fan $\widehat{\Si}\in\SF(\widehat{V})$ such that $Y=Y(\widehat{\Si})$ is a PWS, so giving  a canonical universal 1-covering of $X$. This is, in short, one of the main results of \cite{RT-QUOT}. Namely:

\begin{theorem}[Thm.~2.2 in \cite{RT-QUOT}]\label{thm:covering&quotient} A $\Q$--factorial, complete toric variety $X$ admits a canonical universal 1-covering, $Y$ which is a PWS and such that the 1-covering morphism $\varphi:Y\rightarrow X$ is equivariant with respect to the torus actions. In particular every $\Q$--factorial, complete toric variety $X$ can be canonically described as a finite geometric quotient $X\cong Y/\pi_1(X_\text{reg})$ of a PWS $Y$ by the torus--equivariant action of $\pi_1(X_\text{reg})\cong \Tors(\Cl(X))$.
\end{theorem}

 In particular, if $X$ is projective then, by construction, fans $\Si$ and $\widehat{\Si}$ are associated with the same chamber i.e.
\begin{equation*}
    \g_{\Si}=\g_{\widehat{\Si}}\subseteq \Mov(V)\subseteq\gkz=\langle Q\rangle\,.
\end{equation*}

\begin{remark}\label{rem:Gamma}
 The action of $\pi_1(X_\text{reg})$ on $Y$ can be (non-canonically) described by means of a \emph{torsion matrix} $\Ga$ representing the \emph{torsion part} of the class morphism
 \begin{equation*}
   \xymatrix{d_X=f_X\oplus\tau_X:\Weil(X)\ar[r]^-{Q\oplus\Ga}& \Cl(X)\cong F\oplus\Tors(\Cl(X))}
 \end{equation*}
where $F$ is a free part of the class group $\Cl(X)$. Namely:
\begin{enumerate}
  \item the torsion matrix $\Ga$ is constructed as follows:
  \begin{itemize}
    \item choose fan matrices $V$ and $\widehat{V}=\G(\G(V))$ of $X$ and $Y$, respectively, such that there exists a diagonal matrix
        $$\Delta=\diag(c_1,\ldots,c_n)\in\GL_n(\Q)\cap\mathbf{M}_n(\Z)$$
        with
      \begin{itemize}
            \item[-] $1=c_1\,|\,\ldots\,|\,c_n$\,,
            \item[-] $V=\Delta\cdot\widehat{V}$\,,
            \item[-] $\Tors(\Cl(X))\cong\bigoplus_{i=1}^n\Z/c_i\Z=\bigoplus_{k=1}^s\Z/\tau_k\Z$\,,\\ according with the decomposition of $\Cl(X)$ given by the fundamental theorem of finitely generated abelian groups,
         \begin{equation*}
          \Cl(X)=F\oplus\Tors(\Cl(X))\cong\Z^r\oplus\bigoplus_{k=1}^s\Z/\tau_k\Z
         \end{equation*}
                where $s< n\,,\,\tau_k=c_{n-s+k}>1\,,\,c_1=\cdots=c_{n-s}=1$
          \end{itemize}
          (this is possible by \cite[Thm.~3.2(4)]{RT-QUOT});
    \item recalling notation on submatrices given in list \ref{ssez:lista}, consider
    \begin{eqnarray*}
    % \nonumber to remove numbering (before each equation)
      U_Q&\in&\GL_{n+r}(\Z)\ :\ U_Q\cdot Q^T = \HNF (Q^T)\\
      U&:=&\left(
           \begin{array}{c}
             ^rU_Q \\
             \widehat{V} \\
           \end{array}
         \right)\in\GL_{n+r}(\Z)\\
      W &\in&\GL_{n+r}(\Z) \ :\  W\cdot ({^{n+r-s}U})^T=\HNF\left(({^{n+r-s}U})^T\right) \\
      G &:=& {_s\widehat{V}}\cdot\ ({_{s}W})^T \in \mathbf{M}_s(\Z)\\
      U_G&\in&\GL_{s}(\Z) \ :\  U_G\cdot G^T =\HNF(G^T)\,.
    \end{eqnarray*}
    then define
    \begin{equation}\label{Gamma}
        \Ga = {U_G}\cdot\ {_{s}W} \mod {\boldsymbol\tau}
    \end{equation}
    where this notation means that the $(k,j)$--entry of $\Ga$ is given by the class in $\Z/\tau_k\Z$ represented by the corresponding $(k,j)$--entry of ${^sU_G}\cdot\ {_{s}W}$, for every $1\leq k\leq s\,,\,1\leq j\leq n+r$ \cite[Thm.~3.2(6)]{RT-Qerr}
  \end{itemize}
  \item consider the action of $\pi_1(X_\text{reg})$ defined by means of its dual group $\mu(X):=\Hom(\Tors(\Cl(X)),\C^*)$ and induced by the natural complex multiplication of $\Hom(\Weil(Y),\C^*)$ on $Y$ and the injection $\mu(X)\hookrightarrow\Hom(\Weil(Y),\C^*)$ dually determined by $\Ga$ \cite[\S~4]{RT-QUOT}
\end{enumerate}
Such an action gives rise to a good geometric quotient
$\xymatrix{Y\ar@{>>}[r]&X=Y/\mu}$,
due to the famous Cox's result \cite{Cox}.
\end{remark}

\subsection{A Batyrev type classification}\label{ssez:classificazione}

 The previous \ref{thm:covering&quotient}, jointly with Theorem~\ref{thm:maxbord}, Theorem~\ref{thm:birWPTB} and Theorem~\ref{thm:recmaxbord}, respectively, allows us to prove the following statements.

\begin{theorem}\label{thm:quot-maxbord} Given a reduced $n\times(n+r)$ $F$--matrix $V$, with $r\geq 2$, a chamber $\g\in\mathcal{A}_{\Gamma}(V)$ is maximally bordering if and only if the associated $\Q$--factorial projective toric variety $X(\Si_{\g})$ is a finite abelian quotient of a toric cover $Y(\widehat{\Si})$ of a weighted projective toric bundle $\P^W(\mathcal{E})$. In particular the quotient map $Y\twoheadrightarrow X$ gives a Galois covering ramified in codimension $\geq 2$, whose Galois group is $\mu(X)$ and described as above by a torsion matrix $\Ga$ determined as in (\ref{Gamma}).
\end{theorem}

\begin{remark}\label{rem:geometrico_quot} Recalling \S~\ref{ssez:geometrico} and Remark~\ref{rem:geometrico}, given a $\Q$-factorial projective toric variety $X$ whose fan is associated with a maxbord chamber we find the following situation
\begin{equation*}
  \xymatrix{X=Y/\mu&&Y\ar[ll]_-{\begin{array}{c}
                                     _{\text{universal}} \\
                                     ^{\text{1-covering}}
                                   \end{array}
  }\ar[dd]^-{\phi}_-{\begin{array}{c}
                                     _{\text{fake WPS}} \\
                                     ^{\text{fibering}}
                                   \end{array}}\ar[rr]_-f^-{\begin{array}{c}
                                     _{\text{toric}} \\
                                     ^{\text{cover}}
                                   \end{array}
  }&&\P^W(\mathcal{E})\ar[dd]_-{\varphi}^-{\text{WPTB}}\\
  &&&&\\
            &&X_0\ar[rr]^-{f_0}_-{\text{finite}}&&X'}
\end{equation*}
In particular starting from a fan matrix $V$ of $X$, both the universal 1-covering and the right hand side composition of toric morphisms $\varphi\circ f$ are explicitly described.
\end{remark}

\begin{remark}\label{rem:fan_fibrato_2} Notice that the previous Theorem~\ref{thm:quot-maxbord} provides a definitive answer to the question left open in Remark \ref{rem:fan_fibrato}, about the the geometric interpretation of the toric variety $X(\Si)$ constructed from a fan $\Si$ generated by fibred cones as in Prop.~\ref{prop:fan fibrato} and admitting a fan matrix as in (\ref{Vfibrata}) which is a $F$ non-$CF$ matrix.
\end{remark}

\begin{theorem}\label{thm:quot-birWPTB} Let $V$ be a reduced $n\times(n+r)$ $F$--matrix and $X(\Si)$ be a $\Q$--factorial projective toric variety, with $\Si\in\P\SF(V)$. Then $X$ is a toric flip of a finite abelian quotient $X'$ of a toric cover $Y'\twoheadrightarrow\P^W(\mathcal{E})$ of a WPTB if and only if $\Mov(V)$ is maxbord w.r.t. an hyperplane $H\subseteq F^r_{\R}$. In particular:
\begin{enumerate}
  \item calling $Y$ the PWS giving the universal 1-covering of $X$, as in Theorem~\ref{thm:covering&quotient}, the toric flip $X\dashrightarrow X'$ uniquely lifts to giving a toric flip $Y\dashrightarrow Y'$ between 1-coverings and giving rise to the following commutative diagram
      \begin{equation*}
        \xymatrix{Y\ar@{-->}[r]\ar@{>>}[d]&Y'\ar@{>>}[d]\\
                  X\ar@{-->}[r]& X'}
      \end{equation*}
      in which vertical maps represent Galois coverings ramified in codimension $\geq 2$, both with Galois group
      $\mu(X)=\mu(X')$
      and both described by a same torsion matrix $\Ga$ determined as in (\ref{Gamma});
  \item $X'$ has associated chamber $\g'\subseteq\Mov(V)\subseteq\gkz$ which is maxbord w.r.t. the hyperplane $H$;
  \item the finite abelian quotient $Y'\twoheadrightarrow X'$ is described by the previous Theorem \ref{thm:quot-maxbord}.
\end{enumerate}
 \end{theorem}

 \begin{theorem}\label{thm:quot-recmaxbord} A $\Q$--factorial projective variety $X(\Si)$ is a finite abelian quotient of a PWS, say $Y$, which is produced from a toric cover of a WPS by a sequence of toric covers of WPTB's if and only if the corresponding chamber $\g_{\Si}$ is recursively maxbord. In particular the quotient map $Y\twoheadrightarrow X$ gives a Galois covering ramified in codimension $\geq 2$, whose Galois group is $\mu(X)$ and described by a torsion matrix $\Ga$ determined as in (\ref{Gamma}).
\end{theorem}

As above, let $V$ be a reduced $F$--matrix, $Q=\G(V)$ be a positive, REF, $W$--matrix and $\widehat{V}=\G(Q)$ be a $CF$--matrix. Let $X(\Si)$ be the $\Q$--factorial projective toric variety given either by the choice of a fan $\Si\in\P\SF(V)$ or by the choice of a chamber $\g=\g_{\Si}\in\mathcal{A}_{\Ga}(V)$.  Let $Y(\widehat{\Si})$ be the PWS giving the universal 1-covering of $X$, which is $\widehat{\Si}=\widehat{\Si}_{\g}\in\P\SF(\widehat{V})$. Let $\pc=\{V_P\}$, for some $P\subseteq\{1,\ldots,n+r\}$, be a nef primitive collection for $\Si$. Then $\widehat{\pc}:=\{\widehat{V}_P\}$ is such that $\pc^*=\widehat{\pc}^*\in\gkz(1)$ meaning that $\widehat{\pc}$ is a nef primitive collection for $\widehat{\Si}$ if and only if $\pc$ is a nef primitive collection for $\Si$. Then:
\begin{itemize}
  \item[]\emph{$\k:=r_{\Z}(\pc)\in A_1(X)\cap\overline{\NE}(X)$ is a numerically effective primitive relation for $\Si$ if and only if $\widehat{\k}:=r_{\Z}(\widehat{\pc})\in A_1(Y)\cap\overline{\NE}(Y)$ is a numerically effective primitive relation for $\widehat{\Si}$.}
\end{itemize}
In the following, the class $\widehat{\k}$ is called \emph{the universal lifting of} $\k$.

\begin{theorem}\label{thm:quot-contraibile} Let $V$ be a reduced $n\times(n+r)$ $F$--matrix  and $\Si\in\P\SF(V)$. Assume that there exists a primitive collection $\pc$ for $\Si$ whose primitive relation $\kappa=r_{\Z}(\pc)$ is numerically effective. Then the universal lifting $\widehat{\k}$ of $\k$ is either contractible or pseudo--contractible if and only if one of the following equivalent conditions occurs:
\begin{itemize}
  \item[(I)] $\g_{\Si}$ is a maxbord chamber,
  \item[(II)] for every primitive collection $\pc'\neq\pc$, for $\Si$, then $\pc'\cap\pc=\emptyset$.
\end{itemize}
In particular, $\widehat{\k}$ is either contractible or pseudo--contractible depending on which condition in the statement of Theorem \ref{thm:contraibile} is satisfied.
\end{theorem}

\begin{example}\label{ex:quotient} Let us consider the following $4\times 7$ matrix
\begin{equation*}
    V=\left(
        \begin{array}{ccccccc}
          9&11&13&-33&9&44&-97 \\
          10&12&14&-36&10&48&-106 \\
          54&63&75&-192&51&258&-567 \\
          310&365&430&-1105&295&1485&-3265 \\
        \end{array}
      \right)
\end{equation*}
First of all we need to understand if $V$ is a $F$--matrix: if this is the case then $V$ is a reduced $F$-matrix since the $\gcd$ of entries in every column is always 1.

\noindent A matrix $U_V\in\GL_7(\Z)$ such that $\HNF(V^T)=U_V\cdot V^T$ is given by
\begin{equation*}
    U_V=\left(\begin{array}{ccccccc}
          -4&0&1&-1&-1&0&0 \\
          9&2&5&7&7&0&0 \\
          3&-4&5&2&2&0&0 \\
          -1&1&-2&-1&-1&0&0 \\
          -1&-1&-1&-1&0&0&0 \\
          -1&-1&0&1&1&1&0 \\
          3&3&1&-1&-1&0&1
        \end{array}\right)
\end{equation*}
whose bottom three rows give the matrix
\begin{equation*}
    _3U_V=\left(\begin{array}{ccccccc}
          -1&-1&-1&-1&0&0&0 \\
          -1&-1&0&1&1&1&0 \\
          3&3&1&-1&-1&0&1
        \end{array}\right)\sim\left(
                                \begin{array}{ccccccc}
                                  1&1&1&1&0&0&0 \\
                                  0&0&1&2&1&1&0 \\
                                  0&0&0&0&1&2&1 \\
                                \end{array}
                              \right)=:Q
\end{equation*}
Notice that the equivalence $_3U_V\sim Q$ is realized by means of the matrix
$$M=\left(
                                                                    \begin{array}{ccc}
                                                                      -1&0&0 \\
                                                                      -1&1&0 \\
                                                                      1&2&1 \\
                                                                    \end{array}
                                                                  \right)\in\GL_3(\Z)\quad :\quad Q=M\cdot\ _3U_V
$$
Since $Q=\G(V)$ is a reduced $W$--matrix then $V$ is a reduced $F$--matrix by \cite[Prop.~3.12(2)]{RT-LA&GD}. In particular $Q$ is the same positive $\REF$ $W$--matrix of Example \ref{ex:WPTB(b)}, meaning that $\widehat{V}=\G(Q)$ is given by matrix $V$ in Ex.~\ref{ex:WPTB(b)}. Therefore we have a unique maxbord chamber given by
\begin{equation*}
    \g_8=\langle\q_1=\q_2,\q_3,\w\rangle =\left\langle\begin{array}{ccc}
                                                        1 & 1 & 1 \\
                                                        0 & 1 & 1 \\
                                                        0 & 0 & 1
                                                      \end{array}
    \right\rangle
\end{equation*}
which is also a recursively maxbord chamber: we are then in the situation described by Theorem \ref{thm:quot-recmaxbord}.
Calling $\Si\in\P\SF(V)$ and $\widehat{\Si}\in\P\SF(\widehat{V})$ the corresponding fans, the covering $Y(\widehat{\Si})$ is given by the PWS $X(\Si_8)$ described in Ex.~\ref{ex:WPTB(b)} i.e.
\begin{equation*}
    \xymatrix{Y\ar[r]^-{2:1}_-f&\P^{(1,2,1)}\left(\mathcal{O}_S(2h)\oplus\mathcal{O}_S(f+2h)\oplus\mathcal{O}_S(f)\right)\stackrel{\varphi_{\widetilde{\k}}}{\twoheadrightarrow}
    S\cong\P\left(\mathcal{O}_{\P^1}\oplus\mathcal{O}_{\P^1}(1)\right)\twoheadrightarrow\P^1}
\end{equation*}
where $f,h$ are the generators of $\Pic(S)$ given by the fibre and the pull--back of the Picard generator $\mathcal{O}_{\P^1}(1)$ of the base $\P^1$, respectively, and $\widetilde{\k}=f(\k)$ is the contractible class image of the pseudo--contractible class $\k$ which is the numerically effective primitive relation given by the bottom row of $Q$: recalling Theorem \ref{thm:quot-contraibile}, $\k$ is the universal lifting of the primitive relation $r_{\Z}(\pc)$, associated with the nef primitive collection $\pc=\{\v_5,\v_6,\v_7\}$ for $\Si$.

Let us here better describe the toric cover $f:Y\stackrel{2:1}{\twoheadrightarrow}\P^W(\mathcal{E})$, where $W=(1,2,1)$ and $\mathcal{E}=\mathcal{O}_S(2h)\oplus\mathcal{O}_S(f+2h)\oplus\mathcal{O}_S(f)$. Both $Y$ and $\P^W(\mathcal{E})$ are PWS, meaning that they can be completely described as Cox geometric quotients by means of $W$--matrices $Q$ and $\widetilde{Q}$ given in Ex.~\ref{ex:WPTB(b)}. Let us denote by $Z\subseteq\C^7$ the zero-locus of the irrelevant ideal associated with the fan $\Si$ (see \cite{Cox} for further details) then:
\begin{itemize}
\item $Y$ is the geometric quotient obtained by the following action of $(\C^*)^3$
\begin{equation*}
    g:(\C^*)^3\times(\C^7\setminus Z)\longrightarrow(\C^7\setminus Z)
\end{equation*}
defined by setting
\begin{equation*}
    g\left((t_1,t_2,t_3),(x_1,\ldots ,x_7)\right):=
    \left(t_1\ x_1,t_1\  x_2,t_1t_2\ x_3, t_1t_2^2\ x_4,t_2t_3\ x_5,t_2t_3^2\ x_6, t_3\ x_7 \right)
\end{equation*}
\item $\P^W(\mathcal{E})$ is the geometric quotient obtained by the following action of $(\C^*)^3$
\begin{equation*}
    l:(\C^*)^3\times(\C^7\setminus Z)\longrightarrow(\C^7\setminus Z)
\end{equation*}
defined by setting
\begin{equation*}
    l\left((s_1,s_2,s_3),(y_1,\ldots ,y_7)\right):=
    \left(s_1\ y_1,s_1\  y_2,s_1^2s_2\ y_3, s_1s_2^2\ y_4,s_2s_3\ y_5,s_2s_3^2\ y_6, s_3\ y_7 \right)
\end{equation*}
\item calling $[X_1:\ldots:X_7]$ and $[Y_1:\ldots:Y_7]$ the associated homogeneous coordinates on $Y$ and $\P^W(\mathcal{E})$, respectively, and recalling matrices $A^{-1}$ and $B$ in (\ref{A,B}), the toric cover $f$ is given by setting:
    \begin{equation*}
        Y_i=X_i\ \text{for $i=1,2,4$,}\quad Y_j=X_j^2\ \text{for $j=3,5,6,7$.}
    \end{equation*}
    One can easily check that the latter is consistent with the given actions $g$ and $l$.
\end{itemize}
Finally we need to describing the finite quotient $Y\twoheadrightarrow X$ to ending up the geometric description of the $\Q$--factorial projective variety $X(\Si)$ as given in Thm.~\ref{thm:quot-maxbord}. For this purpose we have to determine the torsion matrix $\Ga$ as in (\ref{Gamma}). Then
\begin{eqnarray*}
    H&=&\HNF\left(V\right)=\left(
                                               \begin{array}{ccccccc}
                                                 1&0&0&-1&10&-8&6 \\
                                                 0&1&0&-1&27&-25&23 \\
                                                 0&0&1&-1&24&-23&22 \\
                                                 0&0&0&0&30&-30&30 \\
                                               \end{array}
                                             \right)\\
                                             U&=& \left(
          \begin{array}{cccc}
            -13&36&7&-2\\
            -26&92&16&-5\\
            -22&83&17&-5\\
            -30&105&20&-6\\
          \end{array}
        \right)\in\GL_4(\Z): \quad  U\cdot V = H
\end{eqnarray*}
\begin{equation*}
    \widehat{H}=\HNF\left(\widehat{V}\right)=\left(
                                               \begin{array}{ccccccc}
                                                 1&0&0&-1&0&2&-4 \\
                                                 0&1&0&-1&0&2&-4 \\
                                                 0&0&1&-1&0&1&-2 \\
                                                 0&0&0&0&1&-1&1 \\
                                               \end{array}
                                             \right)=\widehat{V}\Longrightarrow\widehat{U}=I_4
\end{equation*}
By \cite[Prop.~3.1(3)]{RT-LA&GD} there exists $\beta,\beta_H\in \mathbf{M}(4,4;\Z)\cap\GL(4,\Q)$ such that $V=\beta \widehat{V}$ and $H=\beta_H\widehat{H}$. Namely
\begin{equation*}
    \beta_H=\left(
              \begin{array}{cccc}
                1&0&0&10 \\
                0&1&0&27 \\
                0&0&1&24 \\
                0&0&0&30 \\
              \end{array}
            \right)\Longrightarrow \b=U^{-1}\cdot\b_H\cdot\widehat{U}=\left(
                                                                        \begin{array}{cccc}
                                                                          9&11&13&9 \\
                                                                          10&12&14&10 \\
                                                                          54&63&75&51 \\
                                                                          310&365&430&295 \\
                                                                        \end{array}
                                                                      \right)
\end{equation*}
Therefore $\Delta=\SNF(\b)$ and $\mu,\nu\in\GL_4(\Z):\Delta=\mu\cdot\b\cdot\nu$ are given by
\begin{eqnarray*}
    \Delta&=&\diag(1,1,1,30)\ \Longrightarrow\ \Tors(\Cl(X))\cong\Z/30\,\Z\quad \text{and}\quad s=1\\
    \mu&=&\left(
          \begin{array}{cccc}
            -1&1&0&0\\
            14&-18&1&0\\
            8&-22&-3&1\\
            -30&105&20&-6\\
          \end{array}
        \right)\\
    \nu&=&\left(
          \begin{array}{cccc}
            1&-1&4&20\\
            0&1&-5&-27\\
            0&0&1&6\\
            0&0&0&1\\
          \end{array}
        \right)
\end{eqnarray*}
Define
\begin{eqnarray*}
    \widehat{V}'&=&\nu^{-1}\cdot\widehat{V}=\left(
                                            \begin{array}{ccccccc}
                                              1&1&1&-3&1&4&-9 \\
                                              0&1&5&-6&-3&10&-17 \\
                                              0&0&1&-1&-6&7&-8 \\
                                              0&0&0&0&1&-1&1 \\
                                            \end{array}
                                          \right)\\
    V'&=&\mu\cdot V =\left(
                       \begin{array}{ccccccc}
                         -1&1&0&0&0&0&0 \\
                         14&-18&1&3&0&-7&14 \\
                         8&-22&-3&17&1&-32&63 \\
                         -30&105&20&-95&-6&176&-346 \\
                       \end{array}
                     \right)
\end{eqnarray*}
Then $V'=\Delta \widehat{V}'$, as in item (1) of Remark~\ref{rem:Gamma}.
A matrix $U_Q\in\GL_7(\Z)$ such that $U_Q\cdot Q^T=\HNF(Q^T)$ is given by
\begin{equation*}
  U_Q=\left(\begin {array}{ccccccc}
  1&0&0&0&0&0&0\\
  -2&1&1&0&0&0&0\\
  3&-2&-1&0&1&0&0\\
  3&-2&-2&1&0&0&0\\
  -1&1&0&0&0&0&0\\
  -4&3&1&0&-2&1&0\\
  -3&2&1&0&-1&0&1\end {array}\right)
\end{equation*}
so giving
\begin{equation*}
  U=\left(\begin{array}{c}
            ^3U_Q \\
            \widehat{V}'
          \end{array}
  \right)=\left(\begin {array}{ccccccc}
  1&0&0&0&0&0&0\\
  -2&1&1&0&0&0&0\\
  3&-2&-1&0&1&0&0\\
  1&1&1&-3&1&4&-9\\
  0&1&5&-6&-3&10&-17\\
  0&0&1&-1&-6&7&-8\\
  0&0&0&0&1&-1&1\end {array}\right)\,.
\end{equation*}
The next step is finding $W\in\GL_7(\Z)$ such that $W\cdot(\,^6U)^T=\HNF((\,^6U)^T)$, which is given by
\begin{equation*}
    W=\left(
          \begin{array}{ccccccc}
            1&1&1&1&0&0&0\\
            0&0&1&2&1&1&0\\
            0&0&0&0&1&2&1\\
            0&46&-46&-99&46&153&106\\
            0&2&-2&-5&2&7&5\\
            0&38&-38&-82&38&127&88\\
            0&47&-47&-102&47&157&109
          \end{array}
        \right)\,.
\end{equation*}
Since $s=1$ the matrix $G$ turns out to be the product of the last rows of $\widehat{V}'$ and $W$, rspectively, so giving $G=U_G=(-1)\in\GL_1(\Z)$. Therefore the torsion matrix $\Ga$ is obtained by taking the reduction mod 30 of the opposite of the last row of $W$, that is
\begin{equation*}
  \Ga =   \left(
          \begin{array}{ccccccc}
          [0]_{30}&[13]_{30}&[17]_{30}&[12]_{30}&[13]_{30}&[23]_{30}&[11]_{30}
          \end{array}
        \right)\,.
\end{equation*}
Therefore the finite quotient giving $X$ is obtained by the following action of $\mu_{30}=\Hom(\Tors(\Cl(X)),\C^*)$ on $Y$
\begin{eqnarray*}
    k:&\mu_{30}\times Y&\longrightarrow\ \hskip 1.3cm Y\\
    &(\varepsilon,[X_1:\ldots:X_7])&\mapsto\
   [X_1:\varepsilon^{13}\  X_2:\varepsilon^{17}\ X_3 :\varepsilon^{12}\ X_4 : \varepsilon^{13}\ X_5:\varepsilon^{23}\ X_6 :\varepsilon^{11}\ X_7 ]
\end{eqnarray*}
Equivalently $X$ can be obtained as a Cox geometric quotient by the following action
\begin{equation*}
    k\circ g:((\C^*)^3\oplus\mu_{30})\times(\C^7\setminus Z)\longrightarrow(\C^7\setminus Z)
\end{equation*}
defined by setting
\begin{eqnarray*}
    &&k\circ g\left((t_1,t_2,t_3),(x_1,\ldots ,x_7)\right):=\\
    &&\left(t_1\ x_1,\varepsilon^{13} t_1\ x_2,\varepsilon^{17}t_1t_2\ x_3,\varepsilon^{12}t_1t_2^2\ x_4,\varepsilon^{13} t_2t_3\ x_5,\varepsilon^{23}t_2t_3^2\ x_6,\varepsilon^{11} t_3\ x_7 \right)
\end{eqnarray*}
and giving the following geometric picture
\begin{equation*}
    \xymatrix{&\C^7\backslash Z\ar[dl]_{\pi_k\circ\pi_g}\ar[d]^{\pi_g}\ar[r]^-\phi&\C^7\backslash Z\ar[d]^{\pi_l}&\\
    X&Y\ar@{>>}[l]_-{30:1}^-{\pi_k}\ar@{>>}[r]^-{2:1}_-f&\P^W(\mathcal{E})\ar@{>>}[r]^-{\varphi_{\widetilde{\k}}}&
    \P\left(\mathcal{O}_{\P^1}\oplus\mathcal{O}_{\P^1}(1)\right)\ar@{>>}[r]&\P^1}
\end{equation*}
where $\pi_g,\pi_k,\pi_l$ are the quotient maps associated with the actions $g,k,l$, respectively, and the map $\phi$ is given, recalling (\ref{A,B}), by the exponential action of matrices $B^T$ and $(A^{-1})^T$ on the coordinates of $\C^7$ and of the acting $(\C^*)^3$, respectively.

As already observed for the toric cover $Y$ in Ex.~\ref{ex:WPTB(b)}, also the $F$--matrix $V$ admits seven further projective and simplicial fans different from $\Si$, i.e. $|\P\SF(V)|=8$. By Theorem \ref{thm:quot-birWPTB}, for every $\Si'\in\P\SF(V)$ if $\Si'\neq\Si$ then $X'=X'(\Si')$ is a toric flip of $X(\Si)$. Moreover, since $\g'=\g_{\Si'}'$ is not a maxbord chamber, Theorem \ref{thm:quot-maxbord} guarantees that $X'$ cannot admit a universal 1--covering PWS which is either a WPTB or a toric cover of a WPTB.
\end{example}

\end{document}